\documentclass[a4paper,11pt,oneside,fleqn]{article}

\usepackage{amsmath}
\usepackage{amssymb}
\usepackage{amsthm}
\newtheorem{theorem}{Theorem}
\newtheorem{proposition}[theorem]{Proposition}
\newtheorem{lemma}[theorem]{Lemma}
\newtheorem{corollary}[theorem]{Corollary}
\newtheorem*{claim}{Claim}
\theoremstyle{remark}
\newtheorem*{remark}{Remark}

\usepackage{mathtools}
 \newcommand{\dd}{\textrm{d}}
 \newcommand{\Nat}{\mathbb{N}}
 \newcommand{\Real}{\mathbb{R}}
 \renewcommand{\theta}{\vartheta}
 \newcommand{\ph}{\varphi}
 \DeclareMathOperator{\RE}{Re}
 \DeclareMathOperator{\IM}{Im}
 \DeclareMathOperator*{\argmax}{arg\,max}

 \DeclareMathOperator{\Hf}{H}
 \DeclareMathOperator{\HH}{HH}
 \DeclareMathOperator{\PD}{PD}
 \DeclareMathOperator{\NS}{NS}

\usepackage{enumitem}

\usepackage{graphicx,setspace}
\usepackage{epstopdf}
\usepackage{float}
\usepackage[section]{placeins}
\usepackage[bf,small,textfont=rm]{caption}

\usepackage{harvard}
\citationstyle{dcu}
\usepackage{hyperref}
\hypersetup{pdfborder={0 0 1},
hypertexnames=false,
breaklinks=true,
colorlinks=true,
linkcolor=blue,
urlcolor=red,
citecolor=cyan,
pdftitle={Travelling waves and their bifurcations in the Lorenz-96 model},
pdfauthor={Dirk L. van Kekem and Alef E. Sterk},
pdfsubject={Analytical study of the Lorenz-96 model proves the existence of infinitely many Hopf and Hopf-Hopf bifurcations, generating travelling waves. The dynamics beyond these bifurcations, for F>0, is explored numerically.},
pdfkeywords={Lorenz-96 model, Dynamical systems, Hopf bifurcation, Hopf-Hopf bifurcation, fold bifurcation, period-doubling bifurcation, Neimark-Sacker bifurcation, Chaos, First Lyapunov coefficient, Lyapunov exponent, Travelling waves, Attractor, Organising Center}
}

\makeatletter
\gdef\harvardcite#1#2#3#4{%
  \global\@namedef{HAR@fn@#1}{\hyper@@link[cite]{}{cite.#1}{#2}}%
  \global\@namedef{HAR@an@#1}{\hyper@@link[cite]{}{cite.#1}{#3}}%
  \global\@namedef{HAR@yr@#1}{\hyper@@link[cite]{}{cite.#1}{#4}}%
  \global\@namedef{HAR@df@#1}{\csname HAR@fn@#1\endcsname}%
}
\makeatother

\author{Dirk L. van Kekem \& Alef E. Sterk}
\title{Travelling waves and their bifurcations in the Lorenz-96 model\\
{\normalsize Published in \emph{Physica D}, \href{https://doi.org/10.1016/j.physd.2017.11.008}{doi:10.1016/j.physd.2017.11.008}}}
\date{Accepted manuscript (\today)}

\begin{document}

\maketitle

\begin{abstract}
\noindent In this paper we study the dynamics of the monoscale Lorenz-96 model using both analytical and numerical means. The bifurcations for positive forcing parameter $F$ are investigated. The main analytical result is the existence of Hopf or Hopf-Hopf bifurcations in any dimension $n\geq4$. Exploiting the circulant structure of the Jacobian matrix enables us to reduce the first Lyapunov coefficient to an explicit formula from which it can be determined when the Hopf bifurcation is sub- or supercritical. The first Hopf bifurcation for $F>0$ is always supercritical and the periodic orbit born at this bifurcation has the physical interpretation of a travelling wave. Furthermore, by unfolding the codimension two Hopf-Hopf bifurcation it is shown to act as an organising centre, explaining dynamics such as quasi-periodic attractors and multistability, which are observed in the original Lorenz-96 model. Finally, the region of parameter values beyond the first Hopf bifurcation value is investigated numerically and routes to chaos are described using bifurcation diagrams and Lyapunov exponents. The observed routes to chaos are various but without clear pattern as $n\rightarrow\infty$.
\end{abstract}

%\tableofcontents
%\newpage
%-----------------------------------------------------
\section{Introduction}
\subsection{Setting of the problem}
In his 1996 paper \cite{Lorenz06}, Edward Lorenz introduced two models to study fundamental issues regarding the predictability of the atmosphere and weather forecasting. The so-called monoscale Lorenz-96 model is defined by the equations
\begin{subequations}\label{eq:Lorenz96}
\begin{equation}
    \dot{x}_j =  x_{j-1}(x_{j+1}-x_{j-2}) - x_j + F, \qquad j = 1, \ldots, n,\label{eq:Lorenz96eq}
\end{equation}
where we take the indices modulo $n$ by the following `boundary condition'
\begin{equation}
    x_{j-n} = x_{j+n} = x_j,\label{eq:Lorenz96bc}
\end{equation}
\end{subequations}
resulting in a model with circulant symmetry. For the multiscale model, which will not be discussed in this paper, the reader is referred to \cite{Lorenz06}. The model~\eqref{eq:Lorenz96} can be interpreted as a model for atmospheric waves travelling along a circle of constant latitude. Lorenz interpreted the variables $x_j$ as values of some meteorological quantity (e.g., temperature, pressure, or vorticity) in $n$ equal sectors of a latitude circle, where the index $j$ plays the role of longitude. The continuous parameter $F$ represents external forcing and can be used as a bifurcation parameter.

Although the Lorenz-96 model is not derived from physical principles it still has features which are commonly found in geophysical models: forcing, dissipation and energy preserving quadratic terms. Moreover, unlike the traditional Lorenz-63 model \cite{Lorenz63} which has only one positive Lyapunov exponent, the Lorenz-96 model has multiple positive Lyapunov exponents for suitable choices of the parameters $F$ and $n$. For those reasons, and for the simplicity of the equations, the model is important and widely used nowadays and sometimes even called ``the archetype of large deterministic systems displaying chaotic behavior'' \cite{Fatkullin04} or ``a hallmark representative of nonlinear dynamical systems'' \cite{Frank14}. The applications of the Lorenz-96 model are broad and range from geophysical applications like data assimilation and predictability to studies in spatiotemporal chaos. Table~\ref{tab:AppLz96} gives an overview of recent papers in which the Lorenz-96 has been used together with the values of the parameters that were used.

In contrast to its importance, only a few studies have investigated the dynamics of this model. In \cite{Karimi10}, the high-dimensional chaotic dynamics has been explored by means of the fractal dimension. A recent study on patterns of order and chaos in the multiscale model has reported the existence of regions with standing waves~\cite{Frank14}. Bifurcation diagrams in low dimensions of the Lorenz-96 model have been studied in \cite{Orrell03}, although the emphasis of their work was on methods to visualise bifurcations by means of spectral analysis, rather than exploring the dynamics itself. The previous works already revealed an extraordinarily rich structure of the dynamical behaviour of the Lorenz-96 model for specific values of $n$. However, there has been no systematic study of the dynamics of this model yet. In this paper we fill this gap by studying the dynamical nature of the Lorenz-96 model in greater detail and give analytical proofs of some basic properties for all dimensions and of the existence of Hopf and Hopf-Hopf bifurcations. These results are complemented by numerical explorations, that includes the dynamics beyond these bifurcations as well.

The Lorenz-96 model is a \emph{family} of dynamical systems parameterised by the discrete parameter $n\in\Nat$ which gives the dimension of their state space. This setup is analogous to a discretised partial differential equation. In fact, in some works the Lorenz-96 model is interpreted as such \cite{Basnarkov12,Reich15}. In \cite{Lucarini07}, a discretised quasi-geostrophic model for the atmosphere was studied. In particular, they numerically observed that the parameter value at which the first Hopf bifurcation occurs typically increases with the truncation order of their discretisation method. In pseudo-spectral discretisations of Burgers' equation \cite{Basto06} qualitative differences in dynamics were observed depending on whether the dimension of state space was even or odd. We may expect similar phenomena for the Lorenz-96 model. Hence, in this paper we focus in particular on the question which quantitative and qualitative features of the dynamics will persist for (almost) all $n\in\Nat$. Answers to these questions may be helpful in selecting appropriate values of $n$ and $F$ for the specific applications listed in Table~\ref{tab:AppLz96}. For example, there is a direct relation between the dimension of attractors and the statistics of extreme events \cite{Holland12}. Although the study of this paper is unable to provide the entire picture, it offers a partial dynamical inventory using both analytical and numerical means.

\begin{table}[b]
\makebox[\textwidth][c]{
\begin{tabular}{llrr}
\hline
Reference & Application & $n$ & $F$ \\
\hline
\citeasnoun{Danforth06}    & Making forecasts                   & $40$ & 8 \\
\citeasnoun{Dieci11}       & Approximating Lyapunov exponents   & $40$ & 8 \\
\citeasnoun{Gallavotti14}  & Non-equilibrium ensembles          & $32$ & $\geq 8$ \\
\citeasnoun{Hansen00}      & Operational constraints            & $40$ & 8 \\
\citeasnoun{Leeuw17}       & Data assimilation                  & $36$ & 8 \\
\citeasnoun{Lorenz98}      & Optimal sites                      & $40$ & 8 \\
\citeasnoun{Lorenz05}      & Designing chaotic models           & $30$ & 10 (2.5, \ldots, 40) \\
\citeasnoun{Lorenz06}      & Predictability                     & 36 $(4)$ & 8 (15, 18) \\
\citeasnoun{Lucarini11}    & Ruelle linear response theory      & $40$ & 8 \\
\citeasnoun{Ott04}         & Data assimilation                  & 40, 80, $120$ & 8 \\
\citeasnoun{Stappers12}    & Adjoint modelling                  & $40$ & 8 \\
\citeasnoun{Sterk12}       & Predictability of extremes         & 36 & 8 \\
\citeasnoun{Sterk17}       & Predictability of extremes         & 4, 7, 24 & 11.85, 4.4, 3.85 \\
\citeasnoun{Trevisan11}    & Data assimilation                  & 40, 60, $80$ & 8 \\
\hline
\end{tabular}
}
\caption{Recent papers with applications of the monoscale Lorenz-96 model~\eqref{eq:Lorenz96} and the main values of $n$ and $F$ that were used. Almost all values are chosen in the chaotic domain ($F=8$) of dimension $n=36$ or $40$.}
\end{table}\label{tab:AppLz96}

\subsection{Sketch of the results}
The Lorenz-96 model~\eqref{eq:Lorenz96} has an equilibrium solution given by $x_F = (F,\dots,F)$ for all $n \geq 1$ and $F \in \Real$. Clearly, for $F = 0$ this equilibrium is stable. Numerical simulations show that for $F = 1.2$ the dynamics of the model is periodic for all $n \geq 4$. This suggests that for $0 < F < 1.2$ a supercritical Hopf bifurcation occurs at which the equilibrium $x_F$ loses its stability and gives birth to a periodic attractor. Figure~\ref{fig:periods} shows that the period of the periodic attractor at $F=1.2$ is an oscillating function of the dimension $n$. Observe that the oscillations decay with $n$ and that the period seems to converge to a value of approximately 4.5. The spatiotemporal properties of these periodic orbits are further explored in Figure~\ref{fig:hovmoller} by means of so-called Hovm\"oller diagrams \cite{Hovmoller49}. In these diagrams the value of the variables $x_j(t)$ is plotted as a function of time $t$ and ``longitude'' $j$. Clearly, the waves are travelling in the direction of decreasing $j$ and their wave number increases with $n$. The numerical results clearly indicate that spatiotemporal properties of travelling waves in the Lorenz-96 model depend on the dimension $n$.

A natural question is then which properties stabilize in the limit $n\to\infty$. In this paper we will present the following results which provide a (partial) answer to this question:
\begin{itemize}

\item For all $n\geq 4$ we prove that the trivial equilibrium $x_F=(F,\dots,F)$ exhibits several Hopf or Hopf-Hopf bifurcations for the parameter value $F_{\Hf} \in \Real$. In case of a Hopf bifurcation, we also prove whether the bifurcation is sub- or supercritical.

\item For $F>0$ we prove that the first Hopf bifurcation takes place for $F \in (\tfrac{8}{9},1.19)$ and is always supercritical. The periodic attractor born at this bifurcation has the physical interpretation of a travelling wave. At the Hopf bifurcation the period of this wave is an oscillating function of $n$ which tends to $T_\infty\approx 4.86$ as $n\to \infty$ and the wave number increases linearly with $n$.

\item We encounter further bifurcations of the stable orbit beyond the value $F_{\Hf}$ for which the first Hopf bifurcation takes place. Eventually, this leads to chaotic behaviour. The diagram in Figure~\ref{fig:bifLC} shows the bifurcations following only the stable orbit for various dimensions $n$. Also, the parameter value where chaos sets in is indicated. A clear pattern for all $n$ can not be discerned from the diagram, though a pattern is observed for dimensions $n\leq100$ where $n$ is a multiple of $5$.

\item To unfold the codimension two Hopf-Hopf bifurcation we add an extra parameter to the original model~\eqref{eq:Lorenz96} via a Laplace-like diffusion term in such a way that the original model is easily retrieved. The thus obtained two-parameter system clarifies the role of the Hopf-Hopf bifurcation as organising centre and so it sheds more light on the original model, especially for $n=12$ in which case the Hopf-Hopf bifurcation is the first bifurcation for $F>0$.
\end{itemize}

From our results we can conclude that the spatiotemporal properties of the Lorenz-96 model depend on $n$. This again shows the importance of selecting an appropriate value of $n$ in specific applications. The presence of a Hopf or Hopf-Hopf bifurcation persists for all $n\geq 4$ and the $F_{\Hf}$-value of the first of these bifurcations converges to $\tfrac{8}{9}$ as $n\to\infty$. However, since the resulting waves have different wave numbers the subsequent bifurcation patterns vary with $n$. The case $n=5m$ seems to follow a more regular pattern and will be discussed in Section~\ref{sec:n=5}. Furthermore, the linear increase of the wave number with $n$ indicates that the Lorenz-96 model cannot be interpreted as a discretised partial differential equation.

\begin{figure}[ht!]
  \centering
  \includegraphics[width=\textwidth]{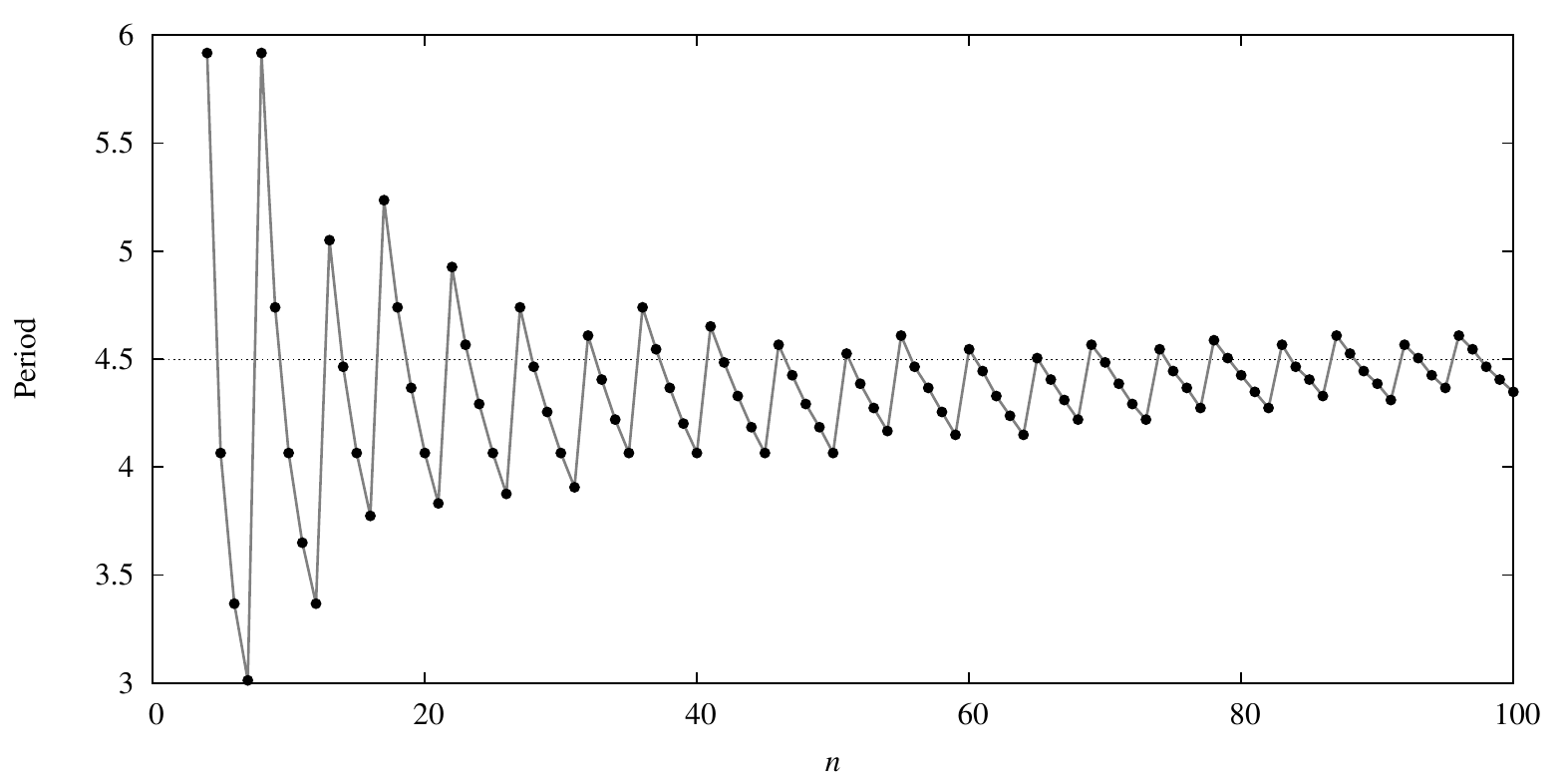}\\
  \caption{The period of the periodic attractor of the Lorenz-96 model detected for $F=1.2$ plotted as a function of the dimension $n$. Note that the period converges to approximately 4.5 as $n\rightarrow\infty$. See Figure~\ref{fig:PeriodComparison} for a comparison with the theoretical period at the Hopf bifurcation.}\label{fig:periods}
\end{figure}

\begin{figure}[ht!]
  \centering
  \includegraphics[width=0.48\textwidth]{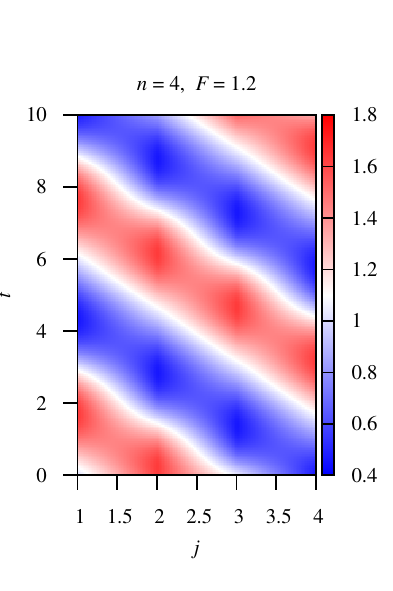}
  \includegraphics[width=0.48\textwidth]{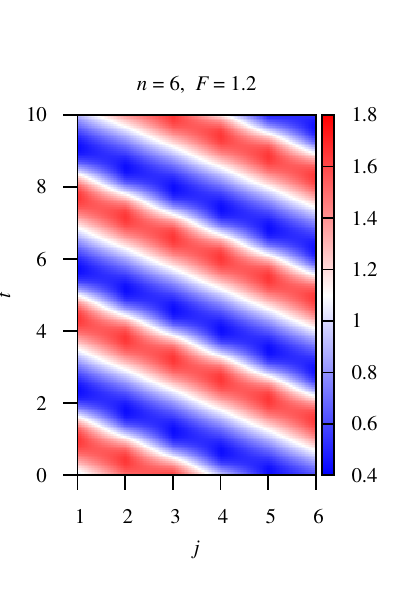} \\
  \includegraphics[width=0.48\textwidth]{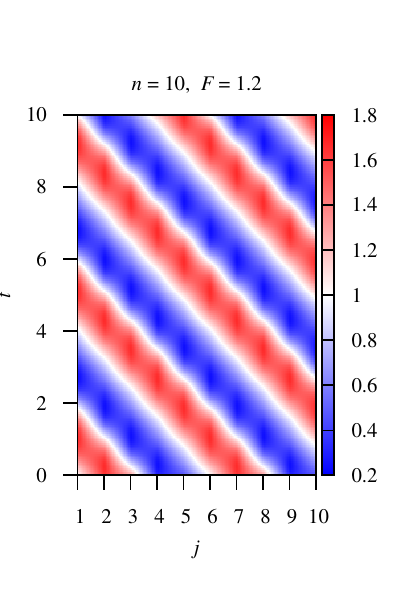}
  \includegraphics[width=0.48\textwidth]{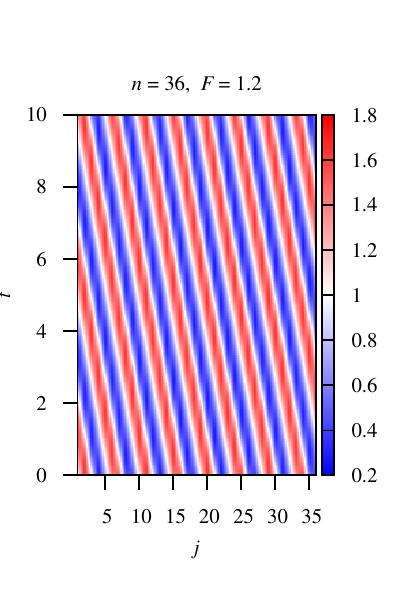}
  \caption{Hovm\"oller diagrams of periodic attractors in the Lorenz-96 model for various dimensions $n$ and parameter value $F$ right after the first Hopf bifurcation. The value of $x_j(t)$ is plotted as a function of $t$ and $j$. For visualization purposes linear interpolation between $x_j$ and $x_{j+1}$ has been applied in order to make the diagram continuous in the variable $j$. Note that both the period and the wave number depend on the choice of $n$.}\label{fig:hovmoller}
\end{figure}

\begin{figure}[p]
  \centering
  \includegraphics[width=\textwidth]{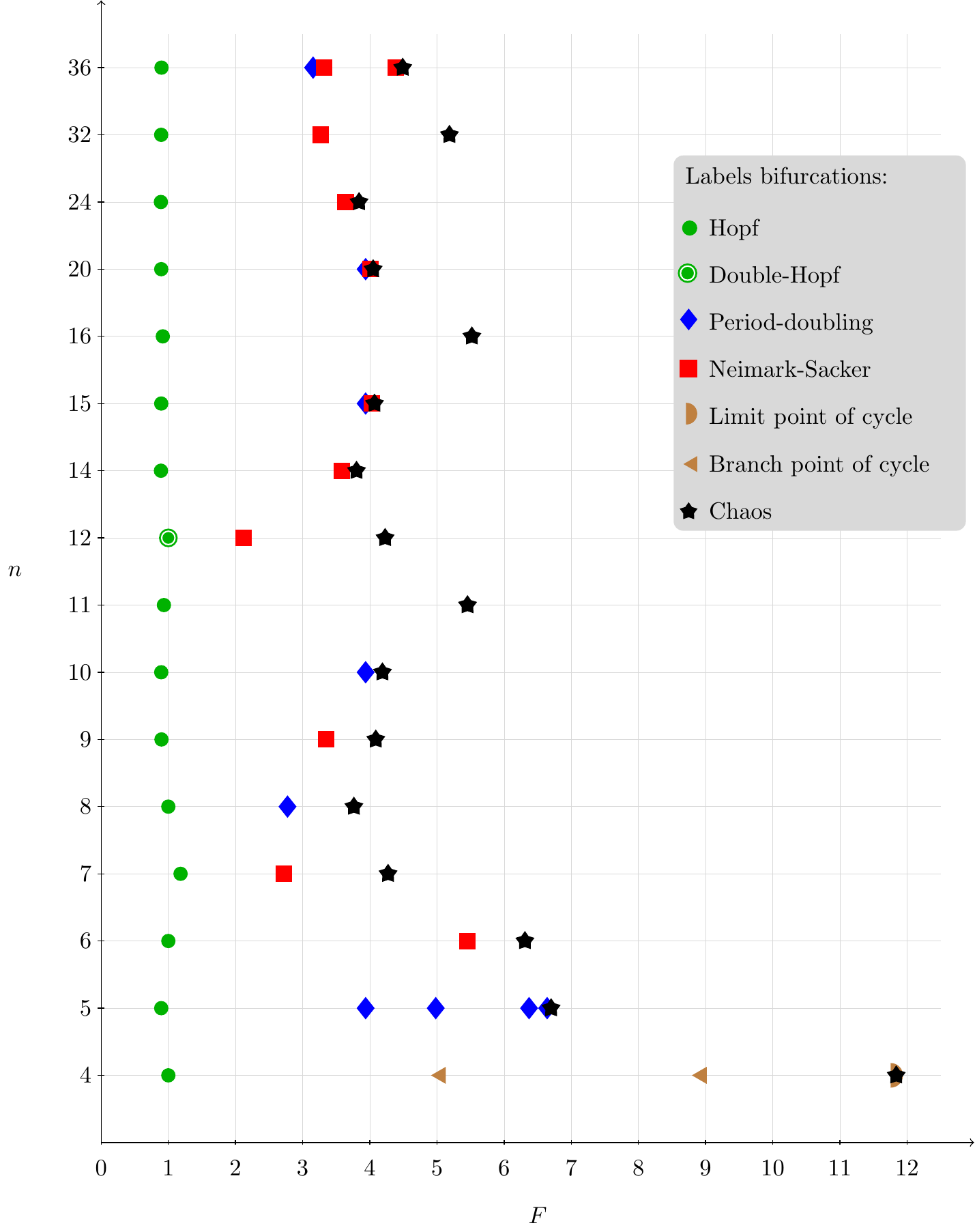}
  \caption{Diagram showing the bifurcations of the stable attractor for $F\in[0,12]$ and for various values of $n$. Each symbol denotes a bifurcation or onset of chaos at the corresponding value of $F$. The type of bifurcation is shown by the legend at the right. Note that we only show (visible) bifurcations of the stable orbits which lead eventually to chaos. Also, we do not include bifurcations of other stable branches, arising from fold bifurcations, for example.}\label{fig:bifLC}
\end{figure}

\subsection{Outline}

This paper has been organised as follows. Section~\ref{sec:AnalyticalResults} presents the analytical results of the research. It starts with some general properties of the Lorenz-96 model, followed by the main theorems of this paper, which concern the Hopf and Hopf-Hopf bifurcations. Subsequently, we show that the periodic orbits, resulting from these bifurcations, can be interpreted as travelling waves. After this, we describe the unfolding of the Hopf-Hopf bifurcation. Next, Section~\ref{sec:NumericalResults} is concerned with a numerical analysis of the dynamics after the first bifurcation for positive $F$. The proofs of all analytical results follow in Section~\ref{sec:proofs}.

%--------------------------------------------------------
\newpage
\section{Analytical results}\label{sec:AnalyticalResults}
In this section we give an overview of the basic properties of the Lorenz-96 model and state the main theorems of this paper.

\subsection{General properties}\label{sec:General}
\paragraph{Trapping Region}The Lorenz-96 model is constructed such that the quadratic part does not affect the total energy of the system,
\begin{equation*}\label{eq:Lz96totalenergy}
E = \tfrac{1}{2}\sum_j x_j^2.
\end{equation*}
This property implies the existence of a trapping region in all dimensions \cite{Lorenz84a}. We will prove this in Proposition~\ref{prop:Exttrappingregion} in the more general setting of a two-parameter unfolding of the Lorenz-96 model.

\paragraph{Equilibrium}The previous result on the existence of a global attractor shows that the attractors are found in a neighbourhood of the origin. It is easy to see that system~\eqref{eq:Lorenz96} has in any dimension the trivial equilibrium
\begin{equation}\label{eq:Lztriveq}
  x_F = (F,\ldots,F),
\end{equation}
where we take $F\in\Real$.

Due to the symmetry in our system we are able to determine the eigenvalues of $x_F$ explicitly. In any dimension, the Jacobian matrix at this equilibrium is a circulant matrix, which means that each row is a right cyclic shift of the row above it. Let us first look at the case $n < 4$, which differs from the general case $n\geq4$. If $n = 1$ or $3$, then the Jacobian matrix is equal to minus the identity matrix and so the eigenvalues are all equal to $-1$. In case $n = 2$, the first row of the Jacobian matrix is given by $(-1-F, F)$ and, hence, its eigenvalues are $\lambda_0 = -1$ and $\lambda_1 = -1-2F$. Since all these eigenvalues never become complex, one can conclude that we do not have any Hopf bifurcation of the trivial equilibrium for dimensions $n < 4$.

Next, consider the case $n \geq 4$. Denote the first row of the Jacobian matrix at $x_F$ by
\begin{equation*}\label{eq:firstrowcirculant}
(c_0, c_1, \dots, c_{n-1}),
\end{equation*}
where $c_0 = -1$, $c_1 = F$, $c_{n-2} = -F$ and $c_k = 0$ for $k \neq 0, 1, n-2$.
It follows that the eigenvalues are given by \cite{Gray06}
\begin{equation*}
  \lambda_j = \sum_{k=0}^{n-1} c_k \rho_j^k, \qquad \rho_j = \exp\left(-2\pi i \tfrac{j}{n}\right), \qquad j=0,\dots,n-1,
\end{equation*}
which can be expressed in terms of $F, j$ and $n$:
\begin{align}\label{eq:Lzsimpleevn}
  \lambda_j(F,n) &= -1 + F\rho_j^1 - F\rho_j^{n-2}\nonumber\\
    &= -1 + F\left(\exp\left(-2\pi i\tfrac{j}{n}\right) - \exp\left(4\pi i\tfrac{j}{n}\right)\right)\nonumber\\
    &= -1 + F f(j,n) + F g(j,n) i,
\end{align}
where $f$ and $g$ are defined as
\begin{equation}\label{eq:Lzevfg}
  \begin{aligned}
  f(j,n) &= \cos \tfrac{2\pi j}{n} - \cos \tfrac{4\pi j}{n},\\
  g(j,n) &= -\sin \tfrac{2\pi j}{n} - \sin \tfrac{4\pi j}{n}.
\end{aligned}
\end{equation}
The eigenvector $v_j$ corresponding to $\lambda_j$ can also be expressed in terms of $\rho_j$:
\begin{equation}\label{eq:Lzeigenvector}
v_j = \frac{1}{\sqrt{n}}\begin{pmatrix*} 1 \\ \rho_j \\ \rho_j^2 \\ \vdots \\ \rho_j^{n-1} \end{pmatrix*}.
\end{equation}

Observe that the eigenvalue $\lambda_0$ equals $-1$ and its corresponding eigenvector has all entries equal to 1. Due to the fact that $\rho_{n-j} = \bar{\rho}_j$, all the other eigenvalues and eigenvectors form complex conjugate pairs since
\begin{align}\label{eq:Lzevconj}
  \lambda_j &= \bar{\lambda}_{n-j},\\
  v_j &= \bar{v}_{n-j},\nonumber
\end{align}
except when $n$ is even, in which case the eigenvalue for $j = \tfrac{n}{2}$ is real and equals $\lambda_{n/2} = -1 - 2F$. Note as well that when $n$ is a multiple of 3 the eigenvalues for $j = \tfrac{n}{3}, \tfrac{2n}{3}$ both equal $-1$. We call the pair $\{\lambda_j, \lambda_{n-j}\}$ the \emph{$j$-th eigenvalue pair}. In the next subsection we will see that each complex eigenvalue pair has a particular parameter value $F$ for which it crosses the imaginary axis and thus causes a Hopf bifurcation.

\subsection{Hopf Bifurcations}\label{sec:Hbif}
The main part of this section is devoted to bifurcations of the trivial equilibrium~\eqref{eq:Lztriveq} for positive values of $F$. To find bifurcations in the Lorenz-96 model, we take $F$ as the bifurcation parameter and vary it along the real line. In this way we discover the occurrence of several Hopf and Hopf-Hopf bifurcations, which we summarize here in two theorems. These are preceded by a lemma which proves that we have the desired eigenvalue crossing that is needed for both cases. The proofs of these results are postponed to Section~\ref{sec:proofs}.

Before we formulate our results on the Hopf-Hopf and Hopf bifurcation in Theorem~\ref{thm:HHBif1} and~\ref{thm:HBif}, respectively, let us first state the preliminary result:
\begin{lemma}[Eigenvalue crossing]\label{lem:Bifevcrossing}
  Let $n\geq4$ and $l \in \Nat$ such that $0 < l < \tfrac{n}{2}, l\neq \tfrac{n}{3}$, then the following holds:
  \begin{enumerate}
    \item The $l$-th eigenvalue pair $\{\lambda_l, \lambda_{n-l}\}(F,n)$ of the trivial equilibrium $x_F$ of system~\eqref{eq:Lorenz96} crosses the imaginary axis transversally at the parameter value $F_{\Hf}(l,n) := 1/f(l,n)$ and thus the equilibrium changes stability.
    \item $F_{\Hf}(l,n)$ lies in the domain $\left(F_{\min}(n),-\tfrac{1}{2}\right) \cup\left[\tfrac{8}{9},F_{\max}(n)\right)$ with
  \begin{align*}
    F_{\min}(n) &= \left\{
        \begin{array}{ll}
            -\tfrac{1}{2}         &\text{if}\ n = 4, 6,\\
            \tfrac{1}{f(r+1,n)}   &\text{otherwise},
        \end{array}
    \right.\\
    F_{\max}(n) &= \left\{
        \begin{array}{ll}
            \tfrac{1}{f(2,7)} &\quad \text{if}\ n = 7,\\
            \tfrac{1}{f(1,n)} &\quad \text{otherwise},
        \end{array}
    \right.
  \end{align*}
  where $r$ is the quotient of $n$ after division by 3.
  \end{enumerate}
\end{lemma}

Due to the shape of the function $f(j,n)$, at most two eigenvalue pairs can have zero real part simultaneously for a particular value of $F$ (see Figure~\ref{fig:Lzevsplitting}). This indicates that the crossing of eigenvalue pairs, described in Lemma~\ref{lem:Bifevcrossing}, can lead to Hopf bifurcations and Hopf-Hopf bifurcations only.

\paragraph{Hopf-Hopf bifurcations}Let us first describe the Hopf-Hopf case: suppose that we have \emph{two} distinct eigenvalue pairs with $l_1$ and $l_2$ which both cross the imaginary axis at the same parameter value $F_{\HH} := F_{\Hf}(l_1,n) = F_{\Hf}(l_2,n)$. In that case we have a Hopf-Hopf bifurcation:
\begin{theorem}[Hopf-Hopf Bifurcation]\label{thm:HHBif1}
Let $l_1, l_2$ and $n$ satisfy the assumptions in Lemma~\ref{lem:Bifevcrossing} with $l_1 \neq l_2$. Then the trivial equilibrium $x_F$ exhibits a Hopf-Hopf bifurcation at $F_{\HH}$ if and only if $l_1$ and $l_2$ satisfy
  \begin{equation}\label{eq:LzHHbifcond1}
    \cos\tfrac{2\pi l_1}{n} +\cos\tfrac{2\pi l_2}{n} = \tfrac{1}{2}.
  \end{equation}
\end{theorem}
From equation~\eqref{eq:LzHHbifcond1} we can deduce two infinite sequences of dimensions for which a Hopf-Hopf bifurcation takes place:
\begin{corollary}\label{cor:HHBif}
  Let $m\in\Nat$, then a Hopf-Hopf bifurcation occurs if we select $l_1, l_2$ and $n$ according to one of the following criteria:
  \begin{enumerate}
    \item\label{crit:HH1} $n=10 m$ and $l_1 = m, l_2 = 3 m$, which corresponds to $F_{\HH} = 2$;
    \item\label{crit:HH2} $n=12 m$ and $l_1 = 2 m, l_2 = 3 m$, which corresponds to $F_{\HH} = 1$.
  \end{enumerate}
\end{corollary}

\begin{remark}
  Equation~\eqref{eq:LzHHbifcond1} gives a necessary and sufficient condition for the occurrence of a Hopf-Hopf bifurcation. However, the explicit values of $l_1, l_2$ and $n$, given in Corollary~\ref{cor:HHBif}, possibly do not provide all occasions where a Hopf-Hopf bifurcation occurs.
\end{remark}

\paragraph{Hopf bifurcations}On the other hand, if equation~\eqref{eq:LzHHbifcond1} is not satisfied then we have only one eigenvalue pair crossing the imaginary axis, which implies a Hopf bifurcation:
\begin{theorem}[Hopf Bifurcation]\label{thm:HBif}
Let $l$ and $n$ be as in Lemma~\ref{lem:Bifevcrossing}. If the $l$-th eigenpair is the \emph{only one} crossing the imaginary axis at the corresponding parameter value $F_{\Hf}(l,n)$, then the equilibrium $x_F$ exhibits a Hopf bifurcation at $F_{\Hf}$.
  The first Lyapunov coefficient for this bifurcation is given by
  \begin{equation*}
  \ell_1(l,n) = \frac{4}{n}\tan(\tfrac{\pi l}{n})\sin^2(\tfrac{3\pi l}{n})\frac{5\cos(\tfrac{2\pi l}{n})+8\cos(\tfrac{4\pi l}{n})-2\cos(\tfrac{6\pi l}{n})-8}{4\cos(\tfrac{2\pi l}{n})-4\cos(\tfrac{4\pi l}{n})+9}.
  \end{equation*}
  Fix $y_0 \in (0,\pi)$ such that
  \begin{equation*}
    5\cos y_0+8\cos 2 y_0 - 2\cos 3 y_0 -8 = 0,
  \end{equation*}
  then $\ell_1(l,n)$ is
  \begin{itemize}
    \item positive if $l$ and $n$ satisfy $0 < \tfrac{l}{n} < \tfrac{y_0}{2\pi} \approx 0.08825746$, which corresponds to a \emph{subcritical} bifurcation;
    \item negative if $\tfrac{l}{n}\in \left(\tfrac{y_0}{2\pi},\tfrac{1}{2}\right)\setminus\{\tfrac{1}{3}\}$ holds, which corresponds to a \emph{supercritical} bifurcation.
  \end{itemize}
\end{theorem}
The number of possible Hopf bifurcations for a given dimension $n$ is exactly equal to the number of conjugate eigenvalue pairs which satisfy Lemma~\ref{lem:Bifevcrossing}. Using equation~\eqref{eq:Lzevconj}, we can count the number of such eigenvalue pairs by the number of eigenvalues with $0< j < \tfrac{n}{2}$, which gives the number $\lceil n/2-1\rceil$ (we need the ceiling-function here if $n$ is odd). However, as described above, if $n$ is a multiple of 3, then the eigenvalue pair with $j=\tfrac{n}{3}$ is not complex, so in this case the number of such eigenvalue pairs equals $\lceil n/2-2\rceil$. For the actual number of Hopf bifurcations, these numbers should be reduced by the number of Hopf-Hopf bifurcations.

We now restrict our attention to the parameter range $F \geq 0$. For $F=0$ the equilibrium $x_F$ is clearly stable. We are interested in the smallest value of $F>0$ at which the equilibrium bifurcates and becomes unstable. From the previous results we know that this must be either a Hopf or a Hopf-Hopf bifurcation.

\begin{proposition}\label{prop:firsthopf}
Let $n \geq 4$ be fixed. For $F>0$ the \emph{first} Hopf or Hopf-Hopf bifurcation occurs for the eigenpair with index
\begin{equation*}
    l_1(n) = \argmax_{0 < l < n/3} f(l,n),
\end{equation*}
which satisfies the bounds
\begin{equation*}
    \frac{n}{6} \leq l_1(n) \leq \frac{n}{4},
\end{equation*}
except for $n=7$, in which case we have to take $l_1=1$.

In particular, if the first bifurcation is a Hopf bifurcation, then this bifurcation is supercritical.
\end{proposition}

\begin{remark}
For $n=12$ the first bifurcation is not a Hopf bifurcation, but a Hopf-Hopf bifurcation. In this case we have  $l_1(n) = \{2,3\}$ as shown in Corollary~\ref{cor:HHBif}; see Section~\ref{sec:Unfoldingn12} for a more detailed exposition of the dynamics in this case.
\end{remark}

\subsection{Travelling waves}\label{sec:travellingwaves}
A fluid is said to be hydrodynamically unstable when small perturbations of the flow can grow spontaneously, drawing energy from the mean flow. At a Hopf bifurcation an equilibrium loses its stability and gives birth to a periodic orbit. In the context of a fluid this can be interpreted as a steady flow becoming unstable to an oscillatory perturbation, such as a travelling or standing wave. Hopf bifurcations are found in many geophysical models as the first bifurcation which destabilizes a steady flow \cite{Lucarini07,Broer11,Dijkstra05,Sterk10}. In this section we show that periodic orbits of the Lorenz-96 model born at a Hopf bifurcation, as described in the previous section, can be interpreted as travelling waves.

Under the conditions of Theorem~\ref{thm:HBif} a Hopf bifurcation associated with the $l$-th eigenpair occurs at $F_{\Hf} = 1/f(l,n)$. In this case equation~\eqref{eq:Lzsimpleevn} gives
\begin{equation*}
    \lambda_l = -\omega_0 i    \qquad\text{with}\qquad
    \omega_0 = \frac{\cos\tfrac{\pi l}{n}}{\sin\tfrac{\pi l}{n}},
\end{equation*}
where we take $\omega_0$ by convention to be the absolute value of the imaginary part at the bifurcation value.
Observe that
\begin{equation*}
    P(t) = \sqrt{F - F_{\Hf}}\big(\cos(\omega_0 t)\RE(v_l) - \sin(\omega_0 t)\IM(v_l)\big).
\end{equation*}
is a periodic orbit of the Lorenz-96 model which is linearized around the trivial equilibrium $x_F$. For $F-F_{\Hf}>0$ sufficiently small the function $P(t)$ is a good approximation of the periodic orbit of the nonlinear system, which allows us to determine the physical properties of the wave.

Using equation~\eqref{eq:Lzeigenvector} gives the $j$-th component of $P(t)$ as
\begin{equation*}
    P_j(t) = \sqrt{\frac{F - F_{\Hf}}{n}}\cos\left(\omega_0 t - \frac{2\pi j}{n}l\right),
\end{equation*}
which is indeed the expression for a travelling wave. In this expression the integer $l$ is the wave number and $\tfrac{2\pi j}{n}$ plays the role of discrete longitude. The temporal frequency of the wave is given by $\omega_0$ which implies that its period is given by
\begin{equation}\label{eq:period}
    T = \tfrac{2\pi}{\omega_0} = 2\pi\tan\tfrac{\pi l}{n}.
\end{equation}
The level curves of $P_j(t)$ are given by the lines $\omega_0 t - \tfrac{2\pi j}{n}l = \text{constant}$, which are decreasing in the $(j,t)$-plane. This implies that waves travel in the direction of decreasing $j$, which is indeed observed in Figure~\ref{fig:hovmoller}.

Furthermore, it is easy to see that a small $l$ implies a small period and a large wave length (i.e.~a long and fast wave), whereas a larger $l$ gives a larger period and a smaller wave length (i.e.~a short and slow wave). Proposition~\ref{prop:firsthopf} shows that the wave number for the first bifurcation for $F\geq 0$, $l_1(n)$, increases linearly with $n$. In the limit $n\rightarrow\infty$ the following relations hold for the wave number $l_1(n)$ and the period:

\begin{proposition}\label{prop:period-infinity}
In the limit $n\to\infty$, the period of the periodic attractor born at the first Hopf bifurcation is given by
\begin{equation*}
    T_\infty = \lim_{n\to\infty} 2\pi\tan\left(\frac{\pi l_1(n)}{n}\right)
    = 2\pi\tan(\tfrac{1}{2}\arccos(\tfrac{1}{4})) \approx 4.867.
\end{equation*}
Similarly, the quotient of $n$ with the wave number $l_1(n)$ satisfies
\begin{equation*}
    \lim_{n\to\infty}\frac{n}{l_1(n)} = \frac{2\pi}{\arccos(\tfrac{1}{4})} \approx 4.767.
\end{equation*}
\end{proposition}
\begin{proof}
Both results follow immediately from the fact that
\begin{equation*}
    \lim_{n\to\infty} \frac{2\pi l_1(n)}{n} = \arccos(\tfrac{1}{4}),
\end{equation*}
by the maximum of $\tilde{f}(y) = \cos y-\cos 2y$, combined with equation~\eqref{eq:period}.
\end{proof}

Proposition~\ref{prop:period-infinity} explains the features that can be observed in Figures~\ref{fig:periods} and~\ref{fig:hovmoller}. The period of the waves tends to a finite limit as $n\to \infty$ whereas the spatial wave number is unbounded.

Figure~\ref{fig:PeriodComparison} compares the period of the periodic attractor for parameter values near and at the bifurcation value. The black curve is taken from Figure~\ref{fig:periods} and shows the period of the periodic attractor, computed numerically at the value $F=1.2$ for all $n\in[4,100]$. The red curve is the period of the periodic attractor computed via the theoretical formula~\eqref{eq:period} for $l=l_1(n)$ with $n\in[4,100]$, i.e.~exactly at the bifurcation value $F_{\Hf}(l_1,n)$. The difference is caused by the fact that the value $F=1.2$ used for the numerical computation is in most cases more than $0.2$ apart from the exact value $F_{\Hf}(l_1,n)$.

\begin{figure}[h!]
  \centering
  \includegraphics[width=\textwidth]{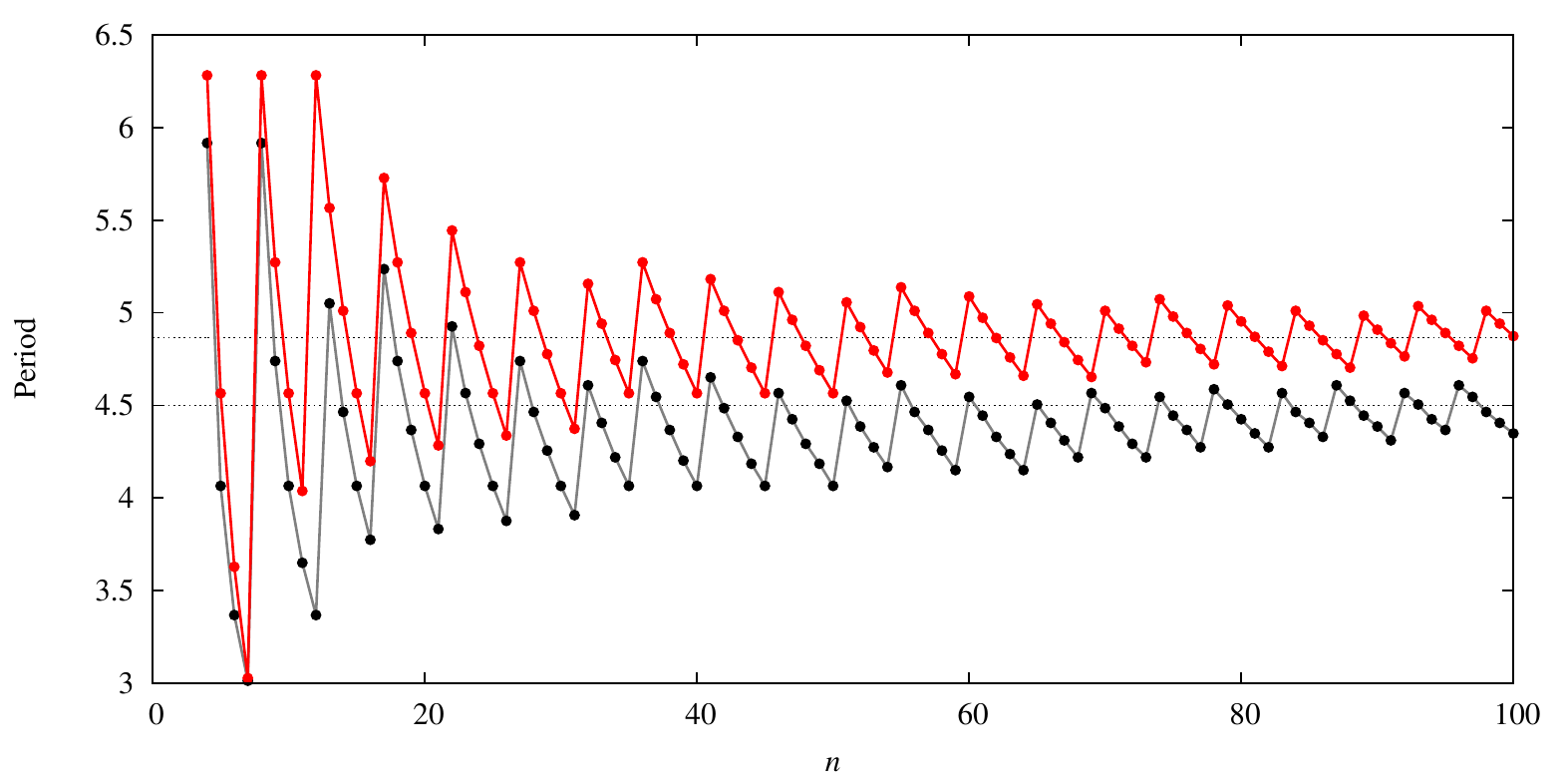}\\
  \caption{Black curve: the period of the periodic attractor detected at $F=1.2$ plotted as a function of $n$. Red curve: the period of the periodic attractor at the Hopf bifurcation as given by equation~\eqref{eq:period}. Note that both curves have a different horizontal asymptote.}
\label{fig:PeriodComparison}
\end{figure}

\subsection{Unfolding for the Hopf-Hopf bifurcation}\label{sec:unfolding}
In many systems one finds bifurcations of codimension less than or equal to the dimension $p$ of the parameter space that are subordinate to certain bifurcation points of codimension $p+1$. To study the qualitative dynamics around such codimension $p+1$ points, one can embed the system in a family of systems parameterised by $p+1$ parameters \cite{Wiggins03}. In this way, the codimension $p+1$ bifurcations act as organising centres of the bifurcation diagram.

Corollary~\ref{cor:HHBif} shows that for $n=12$ the trivial equilibrium $x_F$ loses stability through a codimension two Hopf-Hopf bifurcation at $F_{\HH}=1$. In order to unfold this codimension two  bifurcation completely an extra parameter is needed. In the following we propose an unfolding for general $n$ and take the codimension two Hopf-Hopf bifurcation as an organising centre for this family of systems. Thereafter we will describe the case $n=12$ in more detail and show the role of the Hopf-Hopf bifurcation as organising centre. In section~\ref{sec:n=12} we will show by numerical computations how it dominates the dynamics in its neighbourhood and even influences the phase space for larger parts of the parameter space. A similar approach is taken in \cite{BroerSaleh05,BroerSaleh07,Broer02}, showing the existence and influence of several codimension three, respectively two, points that act as an organising centre.

We choose to add a Laplace-like diffusion term multiplied by a new parameter $G$ to our original equations~\eqref{eq:Lorenz96eq} to obtain the two-parameter system for general dimension $n$ with equations
\begin{equation}\label{eq:Lorenz96eqExt}
  \dot{x}_j =  x_{j-1}(x_{j+1}-x_{j-2}) - x_j + G (x_{j-1}-2x_j+x_{j+1}) + F, \qquad j=1, \ldots, n,
\end{equation}
and keep the boundary condition~\eqref{eq:Lorenz96bc} as it is. The parameter value $G=0$ returns the original Lorenz-96 model.

\subsubsection{General dimensions}
\paragraph{Trapping region}We start the exploration of the two-parameter Lorenz-96 system with the following observation:
\begin{proposition}[Trapping Region]\label{prop:Exttrappingregion}
The two-parameter system~\eqref{eq:Lorenz96eqExt} has a trapping region for any dimension $n\in\Nat$ and for all $F\in\Real$ and $G> -\tfrac{1}{4}$.
\end{proposition}
This result means that we can expect an attractor to exist in the region with $G> -\tfrac{1}{4}$, but in the half-plane below this line an attractor does not necessarily exist. Therefore, the results for the two-parameter system are relevant only for parameter values $G> -\tfrac{1}{4}$.

\paragraph{Equilibrium and Hopf bifurcations}The two-parameter system~\eqref{eq:Lorenz96eqExt} has the same trivial equilibrium~\eqref{eq:Lztriveq} for all $F, G \in \Real$. The Jacobian matrix at this equilibrium is a circulant matrix with first row equal to
  \begin{equation*}
    (c_0, c_1, \dots, c_{n-1}) = (-1-2G, F+G, 0, \ldots, 0, -F, G).
  \end{equation*}
We can express its eigenvalues $\kappa_j$ with $j=0,\dots,n-1$ easily in terms of $F$ and $G$ \cite{Gray06}:
  \begin{align}\label{eq:LzExtevn}
    \hspace{-2.3em}
    \kappa_j(F,G,n) &= \sum_{k=0}^{n-1} c_k \rho_j^k\nonumber\\
        &= -1-2G + (F+G)\rho_j^1 - F\rho_j^{n-2} + G\rho_j^{n-1}\nonumber\\
        &= -1-2G + (F+G)\exp\left(- \tfrac{2\pi j}{n}i\right) - F\exp\left(\tfrac{4\pi j}{n}i\right) + G\exp\left(\tfrac{2\pi j}{n} i\right)\nonumber\\
        &= -1 - 2G\left(1-\cos\tfrac{2\pi j}{n}\right) + F f(j,n) + F g(j,n) i,
  \end{align}
where $f$ and $g$ are given by equation~\eqref{eq:Lzevfg}. The corresponding eigenvector $v_j$ is again given by~\eqref{eq:Lzeigenvector}. Note that equation~\eqref{eq:Lzevconj} holds for $\kappa$ instead of $\lambda$ with similar restrictions, that is, if $j=\tfrac{n}{2}$, we have $\kappa_{n/2} = -1-2F-4G$; and if $j=\tfrac{n}{3}, \tfrac{2n}{3}$, then we have $\kappa_j = -1-3G$. This shows that the case $l=\tfrac{n}{3}$ can not give a Hopf bifurcation and that the additional restriction $0<l<\tfrac{n}{2}$ provides all possible complex eigenvalue pairs. Note also that at $F=0$ all eigenvalues are real, so no Hopf bifurcation is then possible.

The following lemma demonstrates that the two-parameter system~\eqref{eq:Lorenz96eqExt} can exhibit as many different Hopf bifurcations as system~\eqref{eq:Lorenz96}.

\begin{lemma}\label{lem:HopfExt}
Let $n\geq 4$ and $l\in\Nat$ such that $0<l<\tfrac{n}{2}, l\neq\tfrac{n}{3}$, then the trivial equilibrium $x_F$ of system~\eqref{eq:Lorenz96eqExt} exhibits a Hopf bifurcation on the linear bifurcation curves
  \begin{equation}\label{eq:L96ExtHopf}
    G = H_l(F,n) = \frac{F f(l,n)-1}{2(1-\cos\tfrac{2\pi l}{n})},
  \end{equation}
where $F\in\Real\backslash\{0\}$.
\end{lemma}

The Hopf bifurcation points are now turned into straight lines in the $(F,G)$-plane. Along these curves~\eqref{eq:L96ExtHopf} it is possible to determine the type of the bifurcation by computing the first Lyapunov coefficient explicitly in a similar manner as done in the proof of Theorem~\ref{thm:HBif} (see Section~\ref{sec:LC1}), but we will not repeat the procedure here.

\paragraph{Hopf-Hopf bifurcations}It is obvious that the intersections of the Hopf-lines cause Hopf-Hopf bifurcations. One can find all Hopf-Hopf bifurcation points of the trivial equilibrium by equating two Hopf-lines from formula~\eqref{eq:L96ExtHopf} with different $l$. Under the assumption that all nondegeneracy conditions are satisfied, the truncated normal form for a Hopf-Hopf bifurcation reads \cite{Kuznetsov04}
\begin{equation}\label{eq:HHnormalform}
\left\{\begin{aligned}
    \dot{\xi}_2 &= \xi_2(\mu_2 - \sigma \xi_2 - \theta \xi_3 + \Theta\xi_3^2),\\
    \dot{\xi}_3 &= \xi_3(\mu_3 - \sigma \delta\xi_2 - \xi_3 + \Delta\xi_2^2),\\
    \dot{\ph}_2 &= \omega_2,\\
    \dot{\ph}_3 &= \omega_3,
    \end{aligned}\right.
\end{equation}
after a suitable choice of phase variables. (Note that we use indices 2 and 3 for these phase variables which is most convenient for our discussion of the case $n=12$ in Section~\ref{sec:Unfoldingn12}.) Here, $\mu_j$ is defined as $\mu_j := \RE\kappa_j$, $\sigma = \pm 1$ and $\theta, \delta, \Theta, \Delta$ are other normal form coefficients. The sign $\sigma$ mainly determines the type of behaviour near the Hopf-Hopf point \cite{Kuznetsov04}. In any case, two Neimark-Sacker bifurcation (NS) curves emanate from the Hopf-Hopf point. The directions of these NS-curves depend on $\theta$ and $\delta$ and can be computed up to first order via the real part of the eigenvalues at the Hopf-Hopf point \cite{Kuznetsov04}. Note that the type of dynamics we have around the Hopf-Hopf bifurcation point does not depend on the choice of unfolding~\eqref{eq:Lorenz96eqExt}, since the normal form coefficients should be evaluated at the bifurcation point.

\subsubsection{Unfolding for $n = 12$}\label{sec:Unfoldingn12}
In this section we describe the interesting situation of dimension $n=12$ more explicitly, using the results we obtained for general dimensions. Therefore, consider system~\eqref{eq:Lorenz96eqExt} and let $n=12$. The eigenvalues of the Jacobian matrix belonging to the trivial equilibrium $x_F$ are given by equation~\eqref{eq:LzExtevn}.

A Hopf bifurcation for the $l$-th eigenvalue pair (with $0<l<6, l\neq4$) occurs along the Hopf-lines in equation~\eqref{eq:L96ExtHopf}. So, we obtain explicitly the following Hopf bifurcation curves as a function of $F\in\Real\setminus\{0\}$:
\begin{equation}\label{eq:Hopfcurvesn12}
\begin{aligned}
    H_1(F,12) &= \frac{2+(1-\sqrt{3})F}{2\sqrt{3}-4};\\
    H_2(F,12) &= F-1;\\
    H_3(F,12) &= \tfrac{1}{2}(F-1); \\
    H_5(F,12) &= -\frac{2+(1+\sqrt{3})F}{2\sqrt{3}+4}.
\end{aligned}
\end{equation}
The curves for $l=2,3$ intersect each other at $(F,G)=(1,0)$, which is the Hopf-Hopf bifurcation point we discovered in the original system. It is easy to see that for $G=0$ this is the first bifurcation of $x_F$ one encounters by increasing the parameter $F$, since the only Hopf bifurcation with positive $F_{\Hf}$-value has $F_{\Hf}(1,12) \approx 2.732051 > F_{\HH}$. Observe that there is no resonance present at the Hopf-Hopf point (since there $\omega_2 = \IM(\kappa_2(1,H_2(1,12),12)) = \sqrt{3}$ and $\omega_3 = \IM(\kappa_3(1,H_3(1,12),12)) = 1$). Furthermore, in our particular situation, numerical computation of the normal form coefficients using \textsc{MatCont} \cite{Dhooge11} yields the following values
\begin{equation*}
 (\sigma, \vartheta, \delta, \Theta, \Delta) = (1, 1.414, 1.258, -0.200, 0.678),
\end{equation*}
showing that the bifurcation is indeed nondegenerate and that the normal form~\eqref{eq:HHnormalform} is valid.

From the value of $\sigma$ it follows that the dynamics of system~\eqref{eq:Lorenz96eqExt} is of ``type I in the simple case'' as described by \cite{Kuznetsov04}. This means that the two NS-curves are the only bifurcation curves that emanate from the Hopf-Hopf point and in between these curves there exists a region in the $(F,G)$-plane where two stable periodic orbits coexist with an unstable 2-torus. These NS-curves correspond to the limit cycles with $l=2,3$ and are approximated by, respectively,
\begin{align*}
    \mu_3 &= \delta \mu_2 + \mathcal{O}(\mu_2^2), \quad \mu_2 >0,\\
    \mu_2 &= \theta \mu_3 + \mathcal{O}(\mu_3^2), \quad \mu_3 >0.
\end{align*}
Removing higher order terms and solving for $G$ gives the following linear curves in $F$:
\begin{equation}\label{eq:NScurvelinearn12}
\begin{aligned}
    G = T_2(F) &= \frac{1-\delta}{2-\delta}(F-1),\\
    G = T_3(F) &= \frac{1-\theta}{1-2\theta}(F-1).
\end{aligned}
\end{equation}
These lines are tangent to the real NS-curve at the Hopf-Hopf point.

Figure~\ref{fig:Lorenz96n12ExtLocalBifDiagram} displays the local bifurcation diagram with the Hopf-lines and approximated NS-curves for $l=2,3$ together with the phase portraits for each region. In the next section, we verify these results numerically.
  \begin{figure}[ht!]
  \centering
     \includegraphics[width=0.6\textwidth]{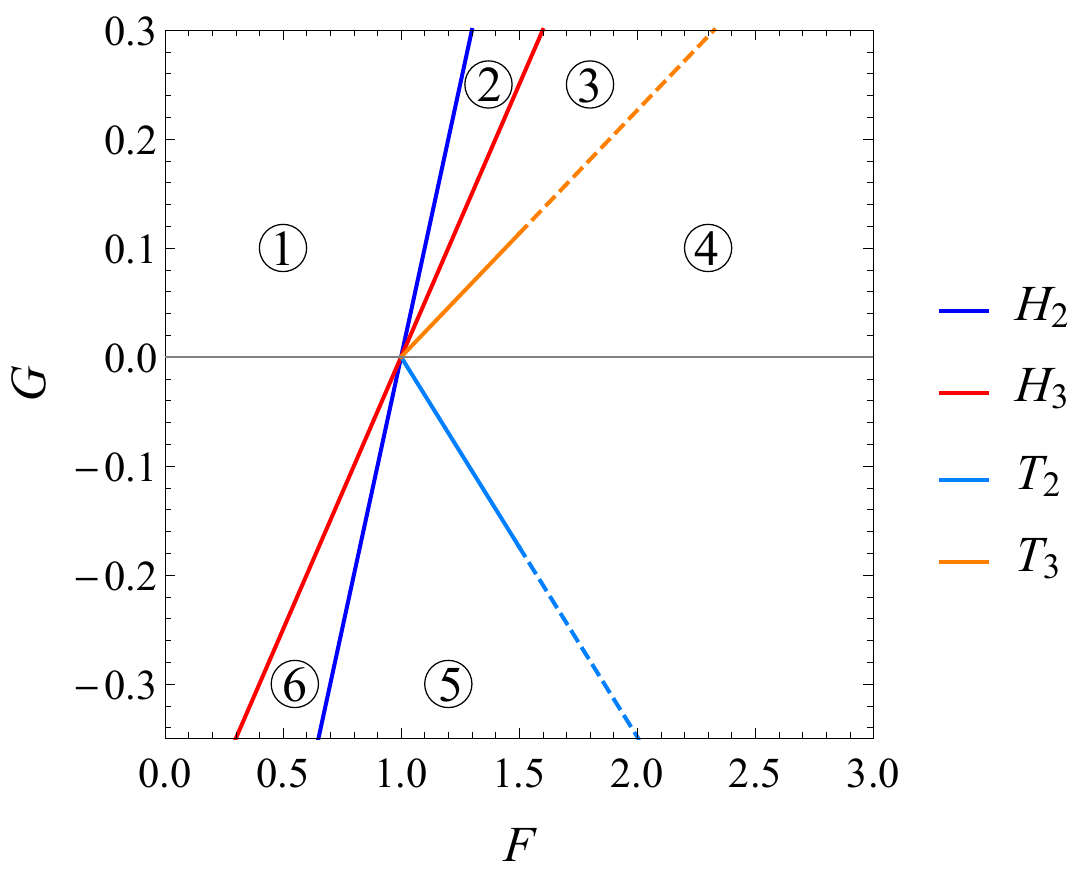}\\
    \vspace{1em}
    \includegraphics[width=0.24\textwidth]{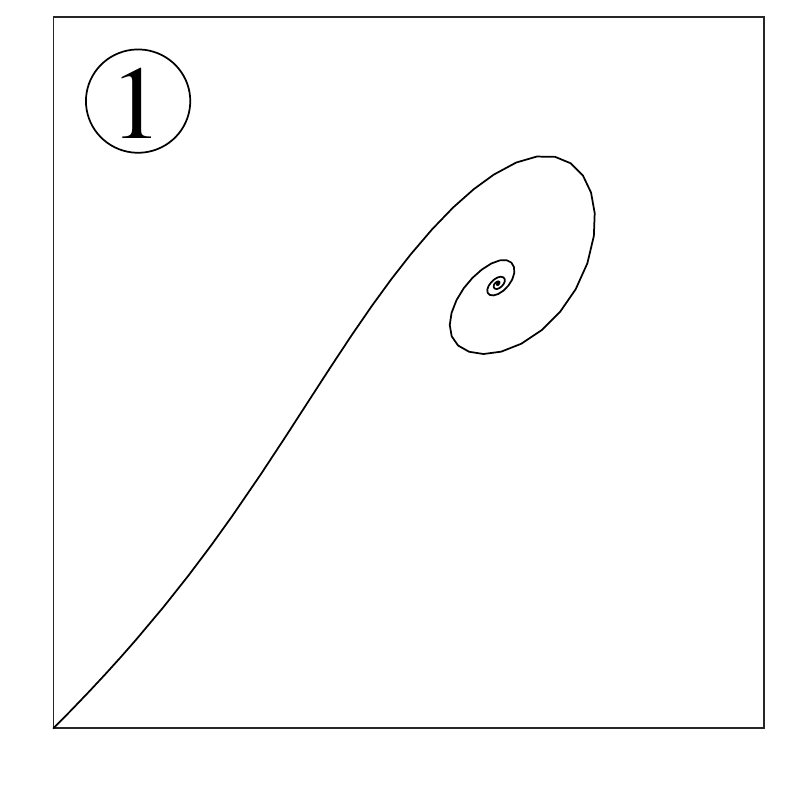}
    \includegraphics[width=0.24\textwidth]{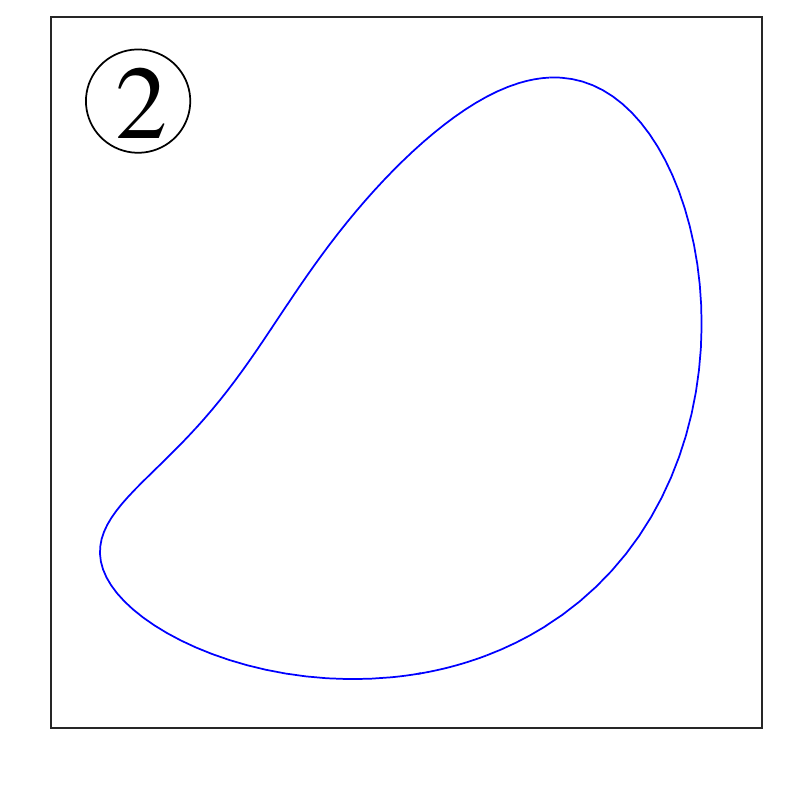}
    \includegraphics[width=0.24\textwidth]{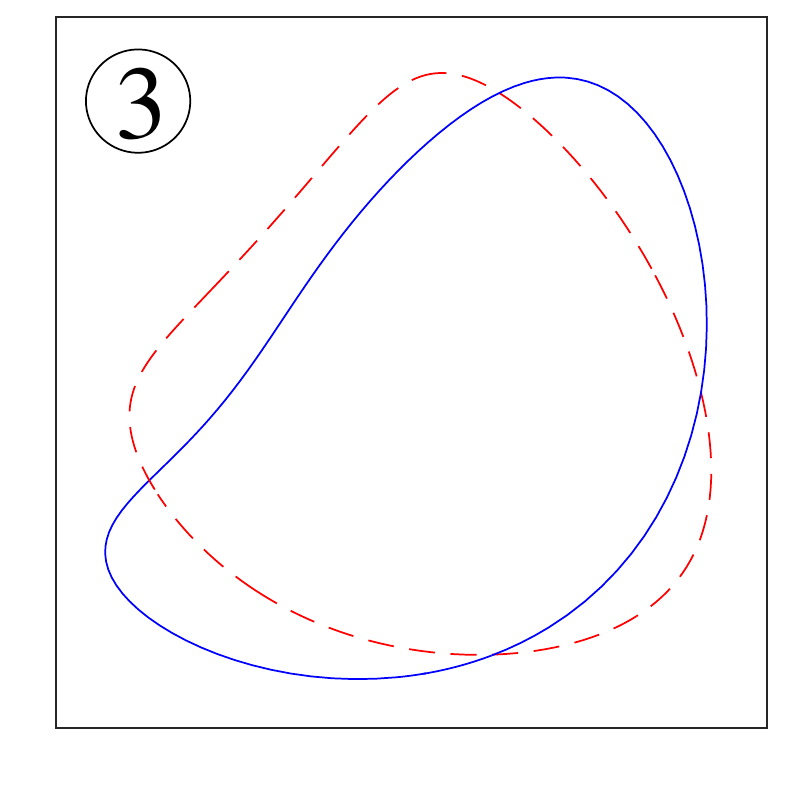}\\
    \includegraphics[width=0.24\textwidth]{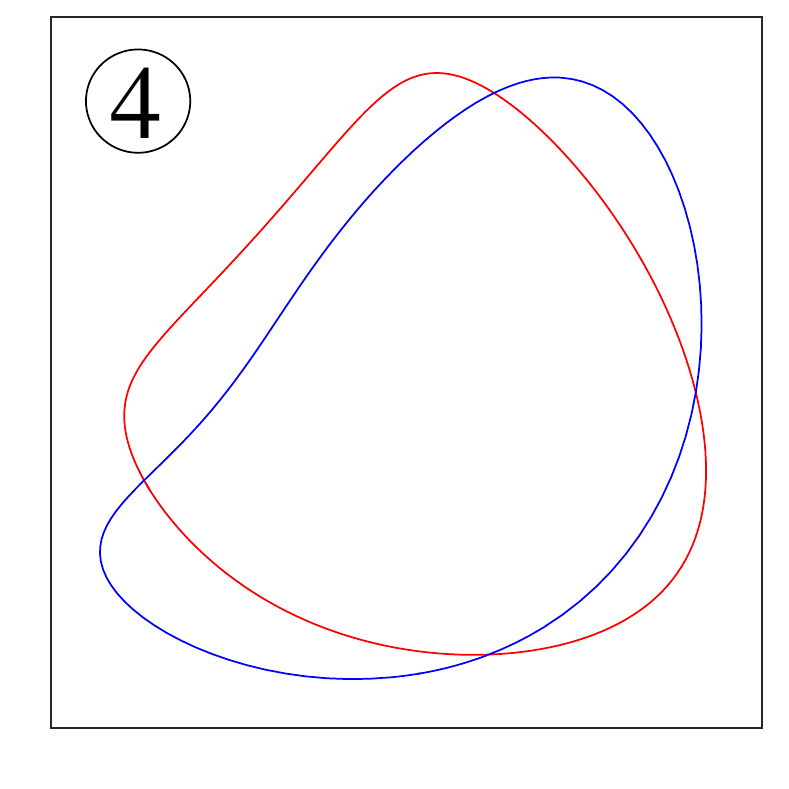}
    \includegraphics[width=0.24\textwidth]{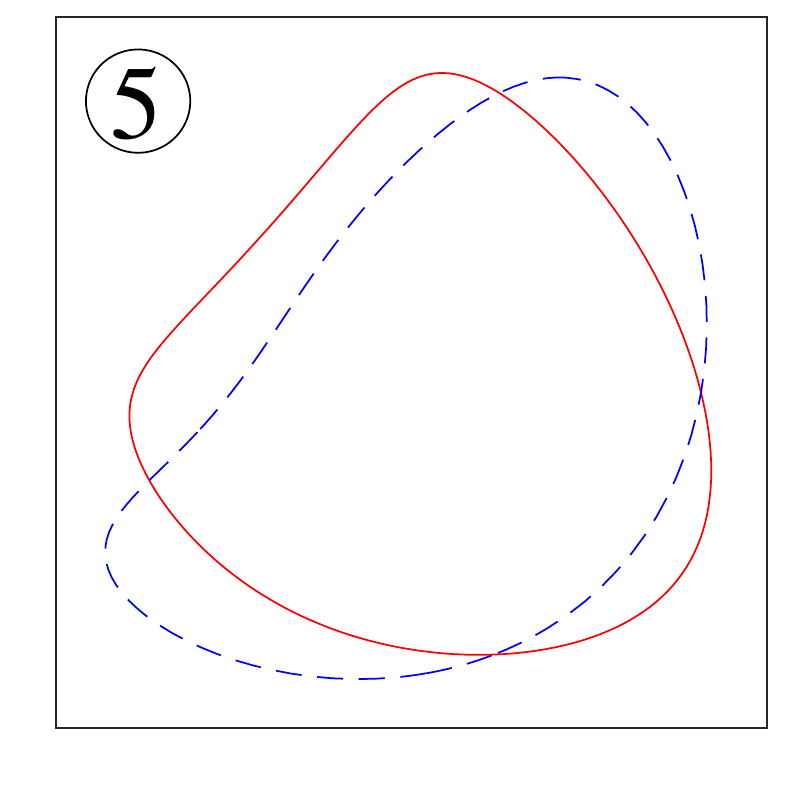}
    \includegraphics[width=0.24\textwidth]{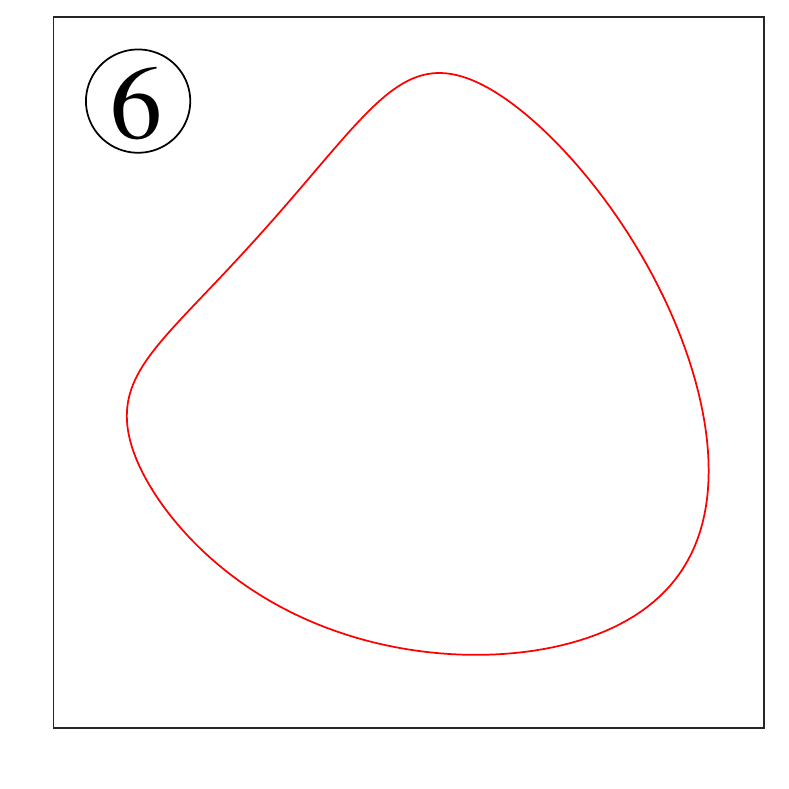}
    \caption{Local bifurcation diagram for $n=12$ near the Hopf-Hopf bifurcation point (top panel). The lines $H_2$ and $H_3$ are the exact Hopf bifurcation curves from~\eqref{eq:Hopfcurvesn12}, whereas the lines $T_2$ and $T_3$ are the linear approximations of the Neimark-Sacker curves~\eqref{eq:NScurvelinearn12}. Compare with the numerical results presented in Figures~\ref{fig:Lorenz96n12ExtBifDiag1C} and~\ref{fig:Lorenz96n12Lyapdiagram}.
    The phase portraits (lower panels) show the dynamics in each of the domains in the top figure. In region 1 there is only one stable equilibrium. In all other domains at least one stable periodic orbit exists. Here, a blue orbit corresponds to wave number 2, a red orbit corresponds to wave number 3, while a dashed line means that the orbit is unstable. Moreover, besides the two stable periodic attractors in region 4 an unstable 2-torus is present, which is not shown in the corresponding phase portrait. Note also that the phase portrait of region 5, respectively 6, is similar to that of region 3, respectively 2, though the stable attractor has wave number 3 instead of 2.}\label{fig:Lorenz96n12ExtLocalBifDiagram}
  \end{figure}

%-----------------------------------------------------
\section{Dynamics beyond the Hopf and Hopf-Hopf bifurcation}\label{sec:NumericalResults}
In Section~\ref{sec:AnalyticalResults} we have proven that the trivial equilibrium $x_F$ of the Lorenz-96 model~\eqref{eq:Lorenz96} loses stability through either a supercritical Hopf or a Hopf-Hopf bifurcation for all dimensions $n \geq 4$. At these bifurcations a periodic attractor is born which has the physical interpretation of a travelling wave (see Section~\ref{sec:travellingwaves}).

In this section we explore the dynamics of the Lorenz-96 model numerically for dimensions up to $n=100$ and beyond the first Hopf bifurcation. Our emphasis is on the bifurcations through which the periodic attractor loses stability. A natural question is to what extent these bifurcations depend on the dimension $n$. Moreover, for a few dimensions we comment on routes to chaos using tools such as continuation, integration, Poincar\'e sections and Lyapunov exponents. The numerical analysis was carried out using mainly the original Lorenz-96 model~\eqref{eq:Lorenz96}. The two-parameter model, however, turns out to be useful to explain the features observed in the one-parameter model. Whenever the two-parameter system~\eqref{eq:Lorenz96eqExt} was used, this is stated explicitly. Otherwise, $G$ is assumed to be equal to $0$. (Recall that we retrieve the original model from the two-parameter system by setting $G=0$.)

An overview of the bifurcations for several dimensions is presented in the diagram in Figure~\ref{fig:bifLC}. As can be seen from this diagram there are various routes to chaos. We will discuss a few of them below. Moreover, for all dimensions shown, except for $n=4$, chaos sets in for $F\in(3,7)$.

\subsection{Dimension $n = 4$}
In the four dimensional Lorenz-96 model there is only one Hopf bifurcation at $F_{\Hf}(1,4)=1$. Continuing the periodic attractor originating from this bifurcation in $F$ and plotting its period against $F$ gives the diagram in Figure~\ref{fig:Lorenz96n4LCF-Period+Stability}. The original periodic orbit disappears through a fold bifurcation at $F_{SN} \approx 11.83823$. Chaos is observed for parameter values $F\geq 11.84$. Figure~\ref{fig:Lorenz96n4LCF1190+1183} compares the periodic attractor for $F=11.83$ with the chaotic attractor for $F=11.9$, while Figure~\ref{fig:Lorenz96n4LCFTimes} shows time series of the first variable for both parameter values. Observe that the  dynamics for $F=11.9$ alternates between approximate periodic behaviour and chaotic behaviour. This is the classical type 1 intermittency scenario as described in \cite{Pomeau80,Eckmann81}. Note that for intermittency we not only need an attractor that has disappeared, but we also need the global dynamics to be such that a typical evolution recurrently visits the part of state space where the attractor disappeared. In our case, this global dynamics might be provided by a heteroclinic structure, as we will show below.

\begin{figure}[h!]
  \centering
    \includegraphics[width=\textwidth]{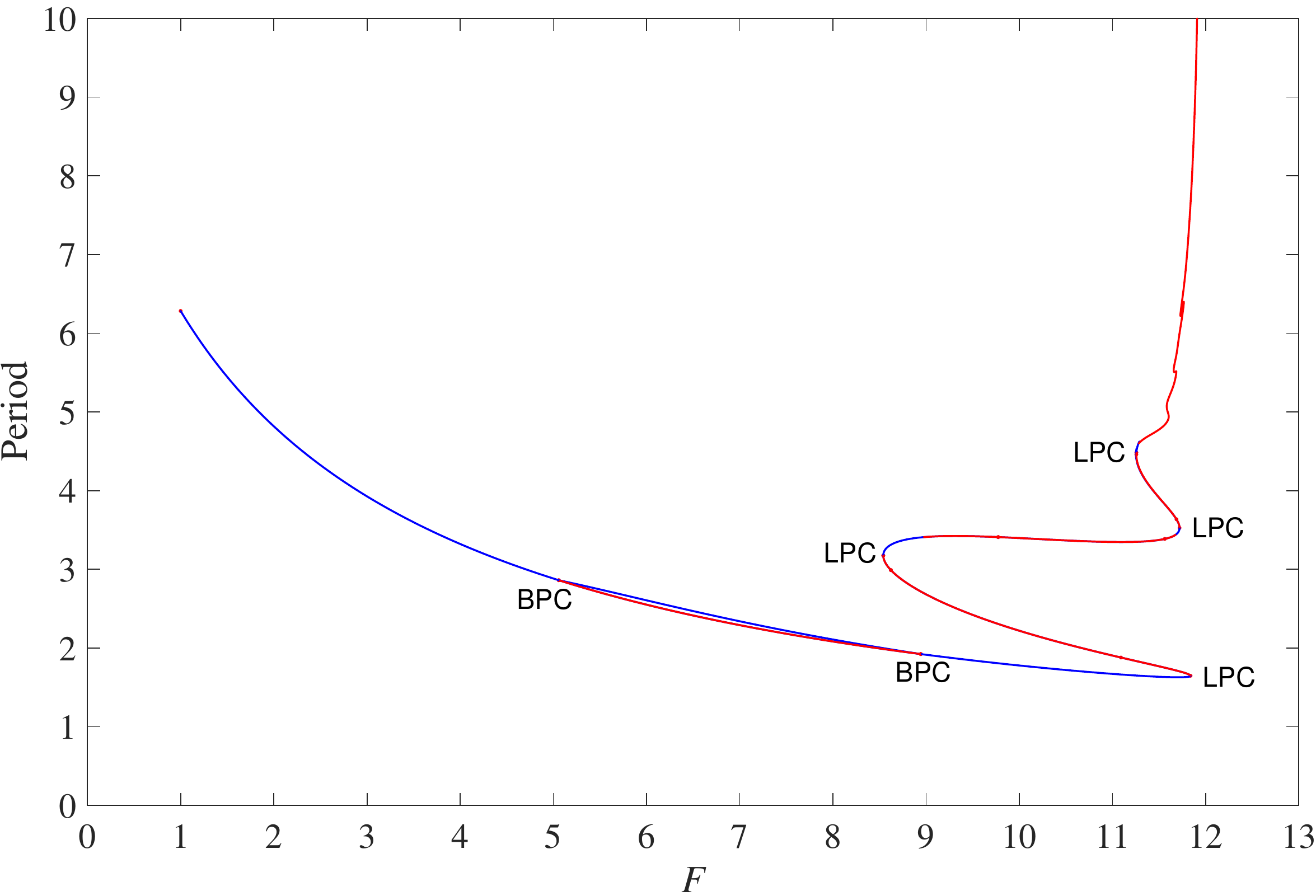}\\
    \caption{Continuation of the periodic orbit originating from the first Hopf bifurcation for $n=4$. For parameter values where the cycle is stable, the curve is coloured blue; where it is unstable, it is coloured red. The periodic attractor remains stable until $F \approx 5.06$ where it exchanges stability with another periodic attractor. However, at $F \approx 8.93$ the original periodic attractor gains stability again. Also, from $F \approx 8.540498$ additional limit cycles are created through fold bifurcations of limit cycles. Finally, at $F_{SN} \approx 11.8382$, it disappears through a saddle-node bifurcation.}\label{fig:Lorenz96n4LCF-Period+Stability}
\end{figure}

\begin{figure}[ht!]
  \centering
    \includegraphics[width=0.49\textwidth]{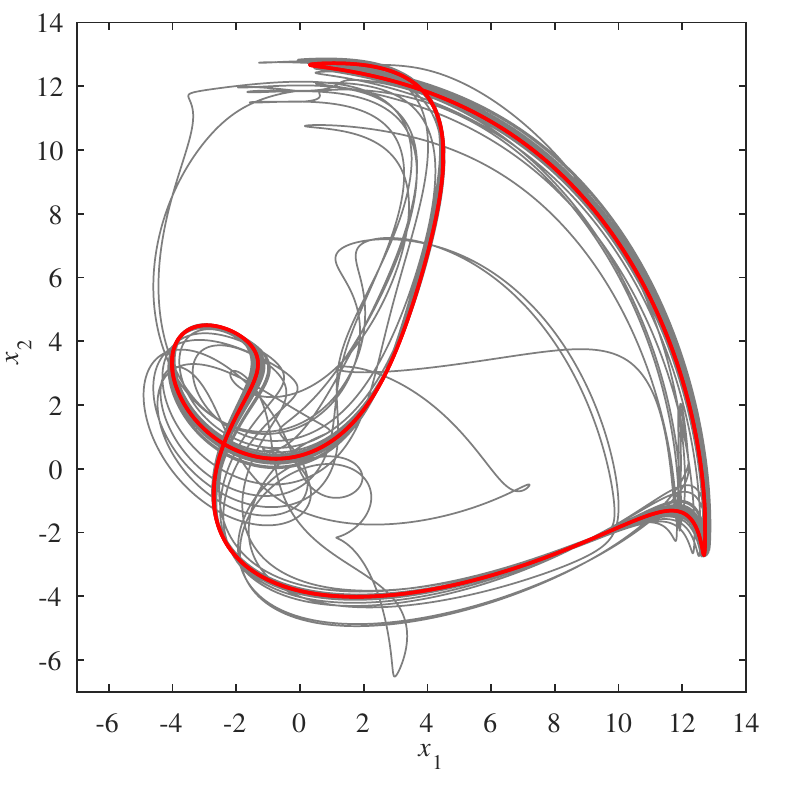}
    \caption{Plot of the attractors for $n=4$ and $F=11.83$ (red) and $F=11.9$ (grey). For the value $F=11.83$ we have a stable periodic orbit, whereas $F=11.9$ gives a chaotic attractor which partly resembles the red, stable periodic orbit. See also Figure~\ref{fig:Lorenz96n4LCFTimes}.}
\label{fig:Lorenz96n4LCF1190+1183}
\end{figure}
\begin{figure}[ht!]
  \centering
  \includegraphics[width=\textwidth]{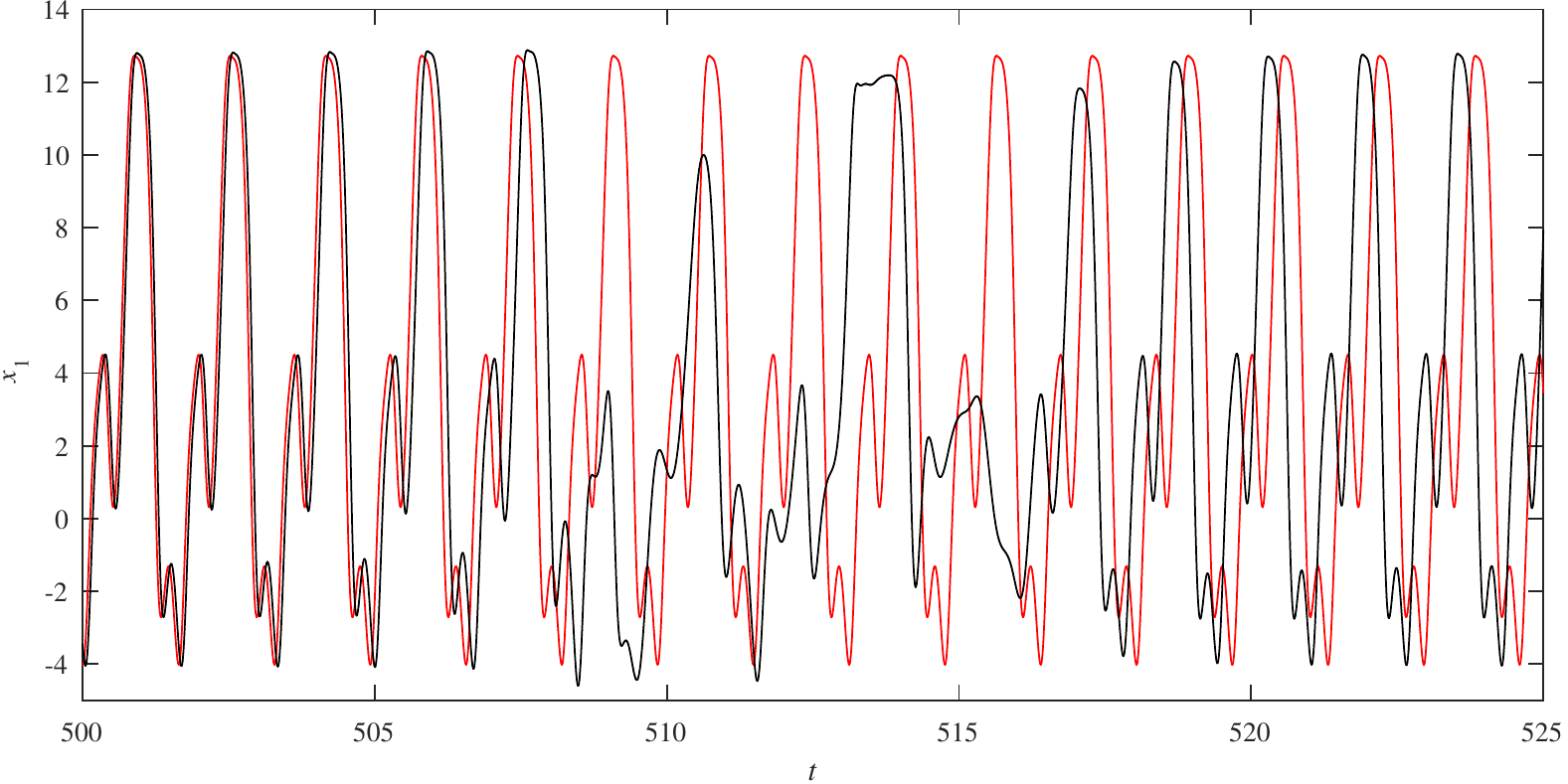}\\
  \caption{Time series of the first coordinate for the attractor with $n=4$ and $F=11.83$ (red, periodic) and $F=11.9$ (black, chaotic). The black curve shows alternating dynamics between periodic and chaotic behaviour which is typical for intermittency.}\label{fig:Lorenz96n4LCFTimes}
\end{figure}
At $F \approx 8.540498$ an additional limit cycle appears through a fold bifurcation of limit cycles, which is stable for only a short interval. This bifurcation is followed by more fold bifurcations, which accumulate at $F \approx 11.77$, as can be seen from Figure~\ref{fig:Lorenz96n4LCF-Period+Stability}.  This phenomenon suggests a homoclinic or heteroclinic structure \cite{Kuznetsov04}. Similar behaviour has been observed in other atmospheric models \cite{Veen03}. Analysis of the system for this parameter value indicates a heteroclinic structure. At $F\approx 8.898979$, namely, two pairs of four equilibria appear through a fold bifurcation. By numerical continuation we found the following coordinates for these equilibria at $F=12.081216$ (the importance of this value will become clear in a moment):
\begin{align}\label{eq:Lz96n4eqfold}
  x^1 &= (-1.182161, -0.233114, 11.543085, 1.126287),\\
  x^2 &= (-2.668230, -1.166341, 6.813306, 1.848383),\nonumber
\end{align}
while the other six equilibria can be obtained by a cyclic shift of the entries. Both types of equilibria are hyperbolic saddles with three, resp.\ two, stable eigenvalues. However, only at $F\approx12.081216$ (which is in the chaotic region) we have numerically detected a heteroclinic cycle between the equilibria $x^1$, using \textsc{MatCont}. A continuation of these connections in the $(F,G)$-plane for the two-parameter system does not yield any other value $F$ for which a heteroclinic cycle exist at $G=0$. The heteroclinic cycle for $(F,G)\approx(12.081216,0)$ is shown in Figure~\ref{fig:Lorenz96n4Heteroclinic}. Notice the similarity between the right panel and the attractor in Figure~\ref{fig:Lorenz96n4LCF1190+1183}.

\begin{figure}[h!]
  \centering
    \includegraphics[width=0.49\textwidth]{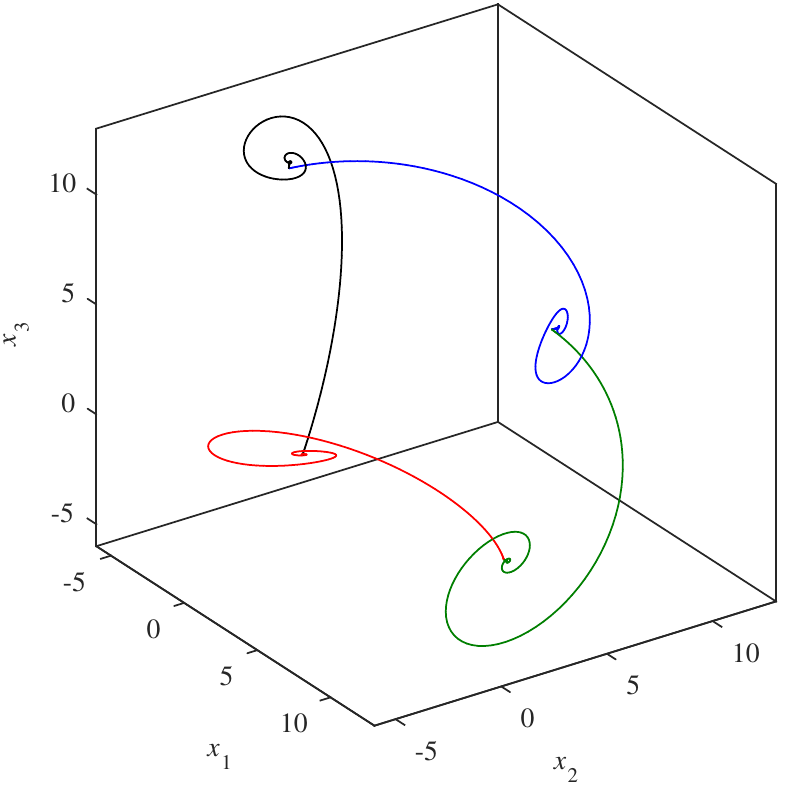}
    \includegraphics[width=0.49\textwidth]{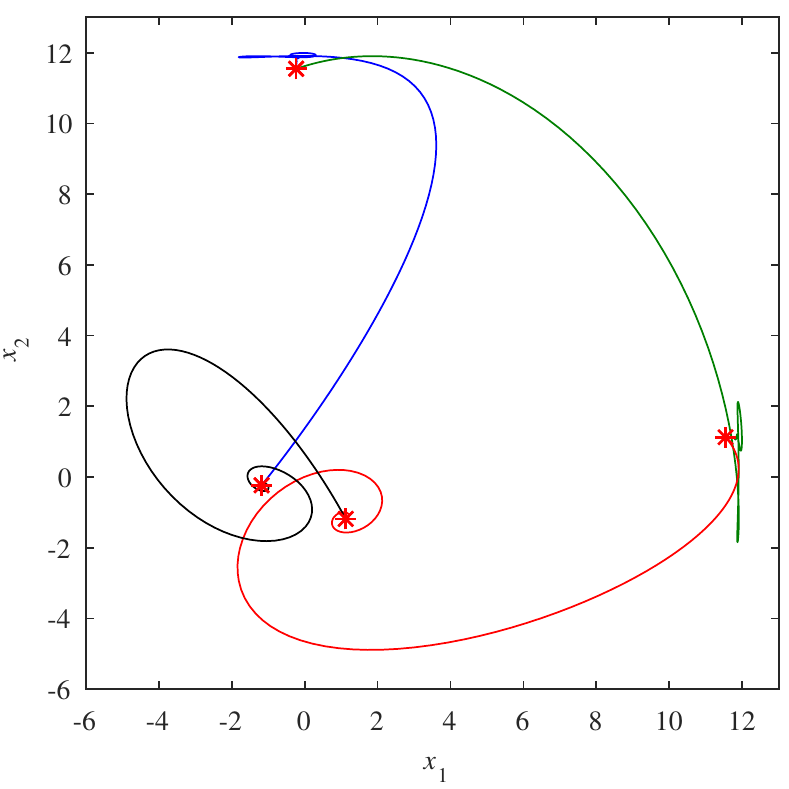}
    \caption{Heteroclinic cycle with four orbits connecting the equilibria~\eqref{eq:Lz96n4eqfold} for $n=4$ and $F=12.081216$ in three coordinates (left panel). The right panel is a projection on the $(x_1,x_2)$-plane and shows also the location of the equilibria. Notice the resemblance to the periodic attractor in Figure~\ref{fig:Lorenz96n4LCFTimes}.}
\label{fig:Lorenz96n4Heteroclinic}
\end{figure}

\subsection{Dimension $n = 5$}\label{sec:n=5}
For $n = 5$, the first bifurcation after the Hopf bifurcation at $F_{\Hf}(1, 5) \approx 0.8944272$ is a period-doubling bifurcation (PD) which occurs at $F_{\PD,1} \approx 3.937853$. The next three PDs occur for the parameter values  $F_{\PD,2}\approx 4.981884$, $F_{\PD,3} \approx 6.371496$, $F_{\PD,4} \approx 6.640968$, respectively.

The bifurcation diagrams in Figure~\ref{fig:dim5-bifos} suggest that a cascade of period-doubling bifurcations takes place. After the cascade, a chaotic attractor is detected at $F=6.72$, see Figure~\ref{fig:dim5-sa}. The Poincar\'e section of this attractor appears to have the structure of a fattened curve. This suggests that the attractor is of H\'enon-like type, which means that it is the closure of an unstable manifold of an unstable periodic point of the Poincar\'e map. We have numerically detected an unstable periodic orbit at $F=6.72$ which corresponds to an unstable period-3 point for the Poincar\'e return map to the section $\Sigma = \{x_1 = 5\}$. The unstable manifold of this period-3 point was computed with standard numerical techniques which are described in \cite{Simo90}. Figure~\ref{fig:dim5-wu} shows a magnification of the unstable manifold along
with the attractor of the Poincar\'e map. The two plots are in very good agreement with each other. Therefore, we conjecture the attractor in Figure~\ref{fig:dim5-sa} to be the closure of the unstable manifold of a saddle periodic orbit.

The PDs persist for each $n = 5m$ with $m=1,\dots,10$. In all these cases, the bifurcation values of the first Hopf bifurcation and the first period-doubling are exactly the same as in the case of $n = 5$ (see Figure~\ref{fig:bifLC5m}). From $n = 55$ on the pattern deviates, because the parameter value of the first Hopf bifurcation changes: now a Neimark-Sacker bifurcation (NS) is the first bifurcation after the Hopf bifurcation, but the torus originating from this bifurcation disappears for slightly larger $F$ and we seem to have again the PD-pattern with a PD at $F_{\PD} = 3.937853$ as before.

We conjecture that this phenomenon can be explained by the wave number of the periodic attractor after the Hopf bifurcation. It turns out that for $n = 5m$ with $m=1,\dots,10$ the wave number of this attractor is exactly equal to $m$. Hence, $n/l\in \Nat$ and therefore the attractor has repeating coordinates with $x_{j+5} = x_j$ for all $1\leq j \leq n$ and indices modulo $n$. This implies that the travelling wave is contained in an invariant subspace of $\Real^n$ defined by
\begin{equation}\label{eq:invmanWl}
    W^5 = \{x\in\Real^n: x_{j+5} = x_j\ \textrm{for all}\ 1\leq j \leq n\},
\end{equation}
where the index $j$ has to be taken modulo $n$. For any $n=5m$, the dynamics restricted to $W^5$ is governed by the Lorenz-96 model for $n=5$. We will further explore this property in forthcoming work \cite{Kekem17b}. However, for $n\geq 55$ this phenomenon breaks down, since the wave number $l$ of the periodic attractor no longer satisfies the relation $n/l \in\Nat$. Nevertheless, it can happen that this periodic attractor becomes unstable and that again the periodic attractor with wave number $m$ takes up stability. However, this is not guaranteed to happen in general, especially for high dimensions, since the quotient $n/l_1(n)$ converges to a non-integer number for $n\rightarrow\infty$ as shown in Proposition~\ref{prop:period-infinity}. Therefore, for increasing $n$ there is an increasing number of periodic attractors whose wave numbers $l$ satisfy $\tfrac{n}{l_1(n)} < \tfrac{n}{l} < 5$ and so it will become more rare to find a stable periodic attractor which inherits its dynamics from the case $n=5$.

\begin{figure}[ht!]
\includegraphics[width=0.49\textwidth]{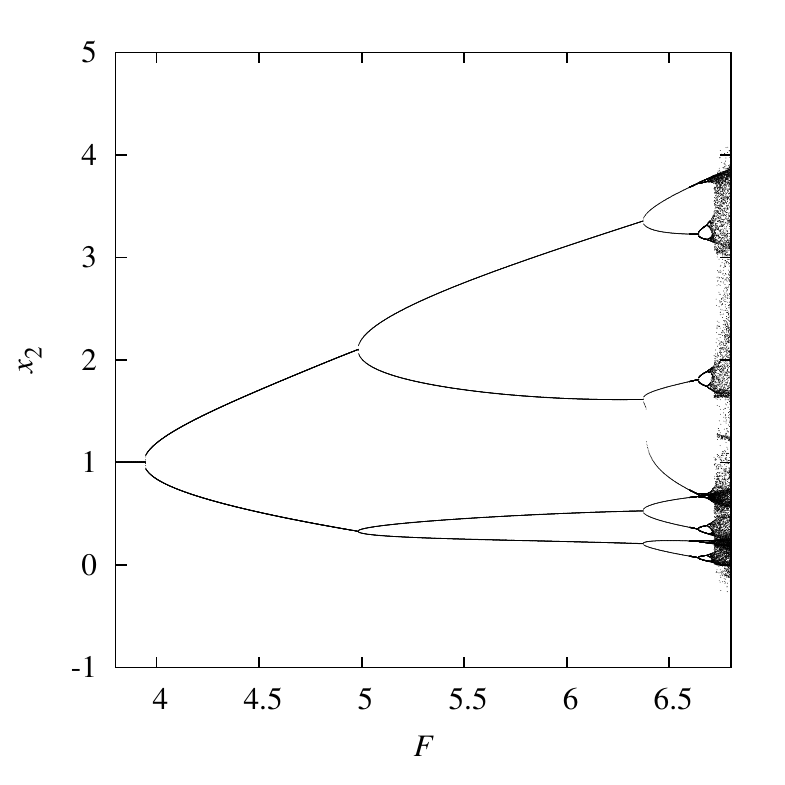}
\includegraphics[width=0.51\textwidth]{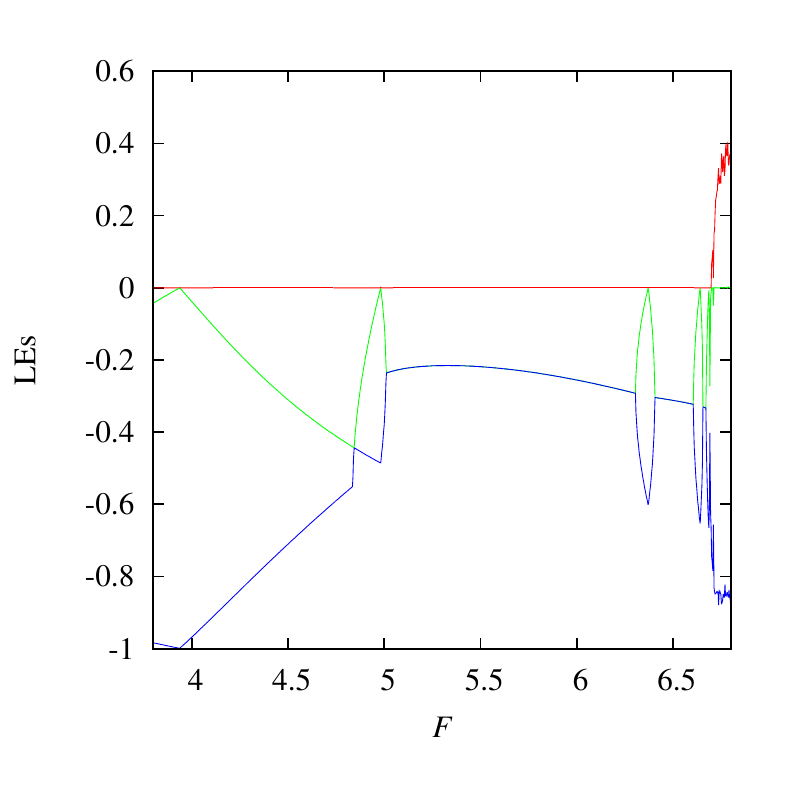}
\caption{Bifurcation diagrams for dimension $n = 5$. The left panel shows the attractors of the Poincar\'e return map defined by the section $\Sigma = \{x_1 = 0\}$; the right panel shows the three largest Lyapunov exponents of the Lorenz-96 model as a function of the parameter $F$.}\label{fig:dim5-bifos}
\end{figure}

\begin{figure}[ht!]
\includegraphics[width=0.49\textwidth]{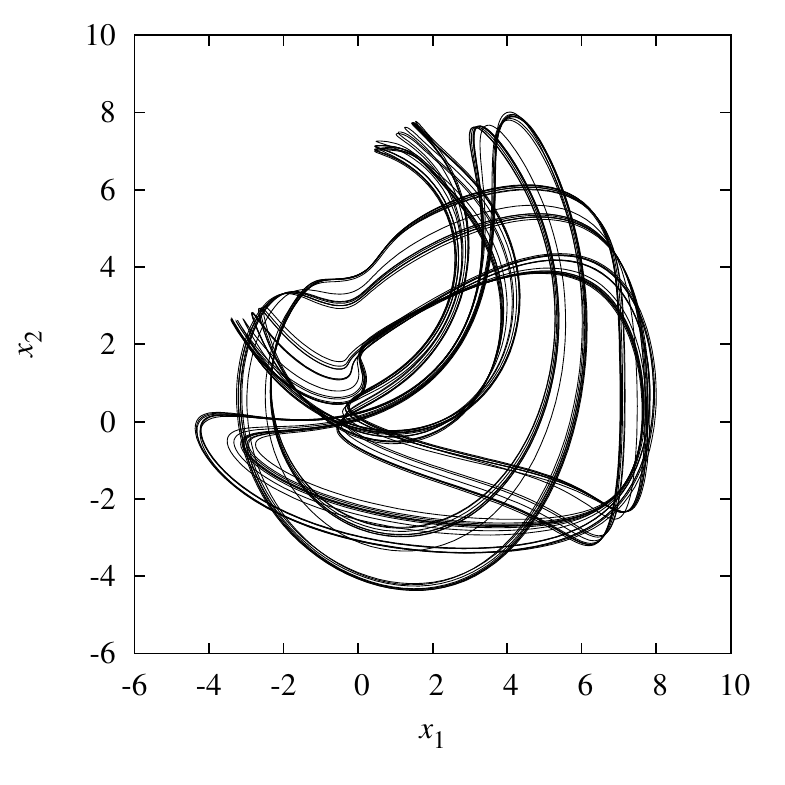}
\includegraphics[width=0.49\textwidth]{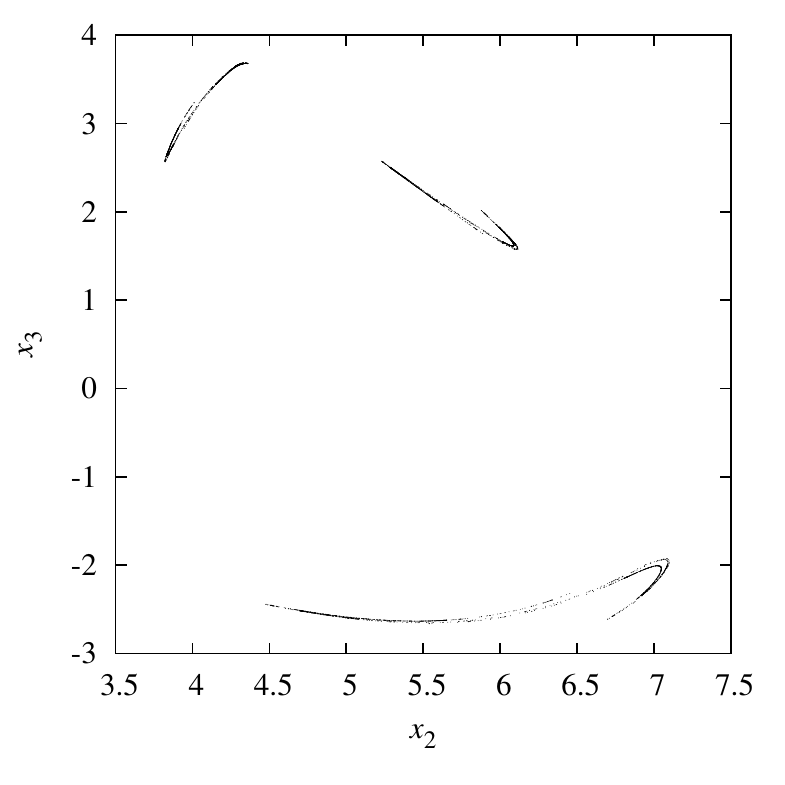}
\caption{A chaotic attractor (left panel) for $(n,F) = (5,6.72)$, which is after the PD-cascade, and the corresponding Poincar\'e section defined by $\Sigma = \{x_1 = 5\}$ (right panel). The latter appears to have the structure of a fattened curve, see also Figure~\ref{fig:dim5-wu}.}\label{fig:dim5-sa}
\end{figure}

\begin{figure}[ht!]
\includegraphics[width=0.49\textwidth]{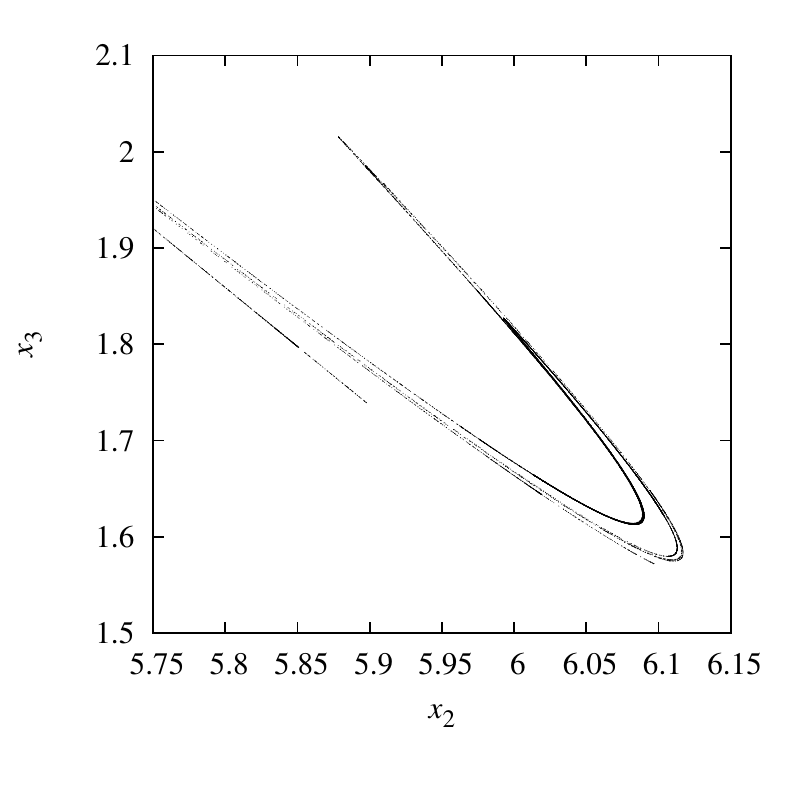}
\includegraphics[width=0.49\textwidth]{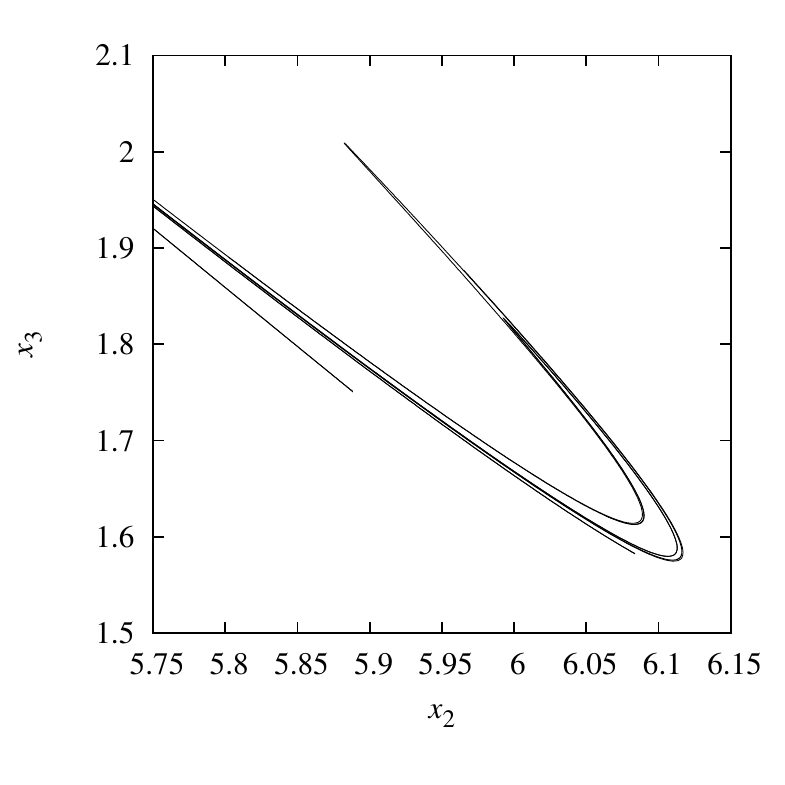}
\caption{Magnification of the Poincar\'e section in the right panel of Figure~\ref{fig:dim5-sa} (left panel) and the unstable manifold of the period-3 point of the Poincar\'e return map at the same parameter values (right panel). The plots agree very well with each other which suggests that the attractor in Figure~\ref{fig:dim5-sa} is the closure of the unstable manifold of the period-3 point.}\label{fig:dim5-wu}
\end{figure}

\begin{figure}[p]
\centering
\includegraphics[width=\textwidth]{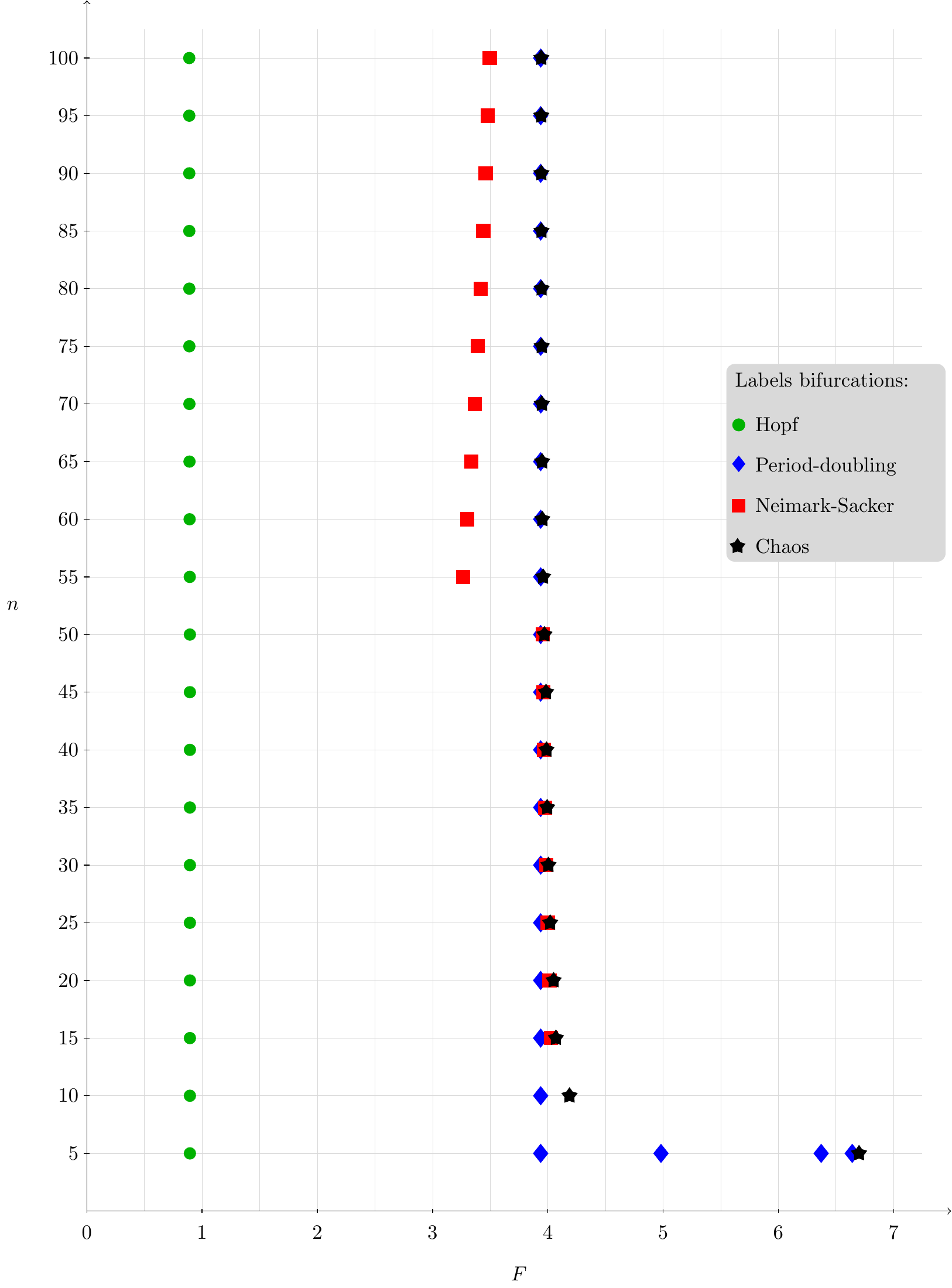}
\caption{Diagram showing the bifurcations of the trivial equilibrium and the stable attractor originating from one or more subsequent bifurcations for $n = 5m$, $m = 1,\ldots,20$. Each shape denotes a bifurcation at the corresponding value of $F$ or represents the point where chaos sets in. The type of bifurcation is shown by the legend at the right. Note that we only show (visible) bifurcations of the stable orbits which lead eventually to chaos. For $n\geq55$, the Hopf bifurcation value changes and a Neimark-Sacker bifurcation appears before the usual period-doubling bifurcation that persists up to at least dimension $n=100$.}\label{fig:bifLC5m}
\end{figure}

\subsection{Dimension $n = 6$}\label{sec:n=6}
For $n=6$ the first bifurcation after the Hopf bifurcation at $F_{\Hf}(1, 6) = 1$ is an NS, which occurs at $F_{\NS} \approx 5.456661$. At this bifurcation the periodic attractor loses stability and gives birth to a quasi-periodic attractor in the form of a 2-dimensional torus, see Figure~\ref{fig:dim6-2torus}. The Lyapunov diagram in Figure~\ref{fig:dim6-bifos} clearly shows alternating intervals of periodic behaviour and quasi-periodic behaviour. This phenomenon can be clarified by the two-parameter system. In the $(F,G)$-plane this alternation organizes itself in the form of the well-known Arnol'd resonance tongues, which emanate from the NS-curve \cite{Kuznetsov04}. For a better visualization of these tongues the affine transformation $(F,G) = (U+6V+1, 0.35\,V-0.25)$ has been used to obtain the right panel of Figure~\ref{fig:dim6-bifos}.

\begin{figure}[ht!]
\includegraphics[width=0.49\textwidth]{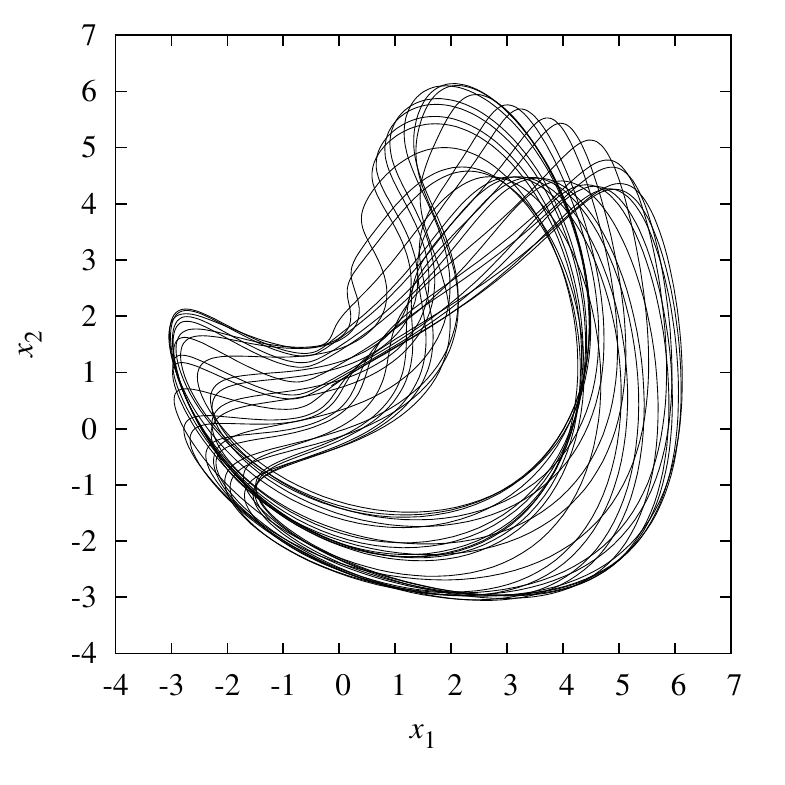}
\includegraphics[width=0.49\textwidth]{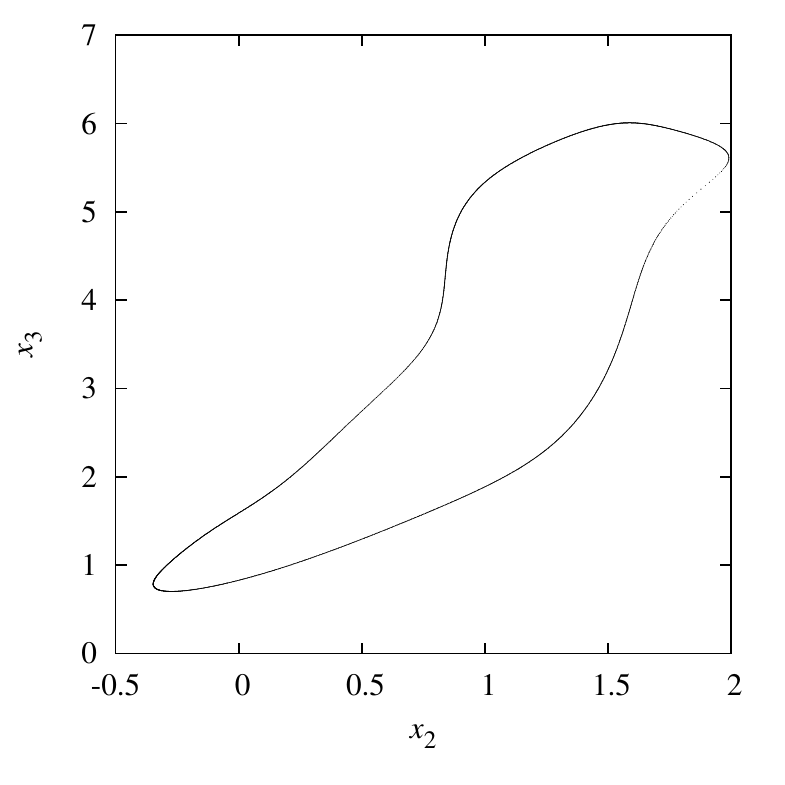}
\caption{A 2-torus attractor (left panel) for $(n,F,G) = (6, 5.6, 0)$ after the NS-bifurcation and the corresponding invariant circle of the Poincar\'e return map defined by the section $\Sigma=\{x_1 = 0\}$ (right panel).}\label{fig:dim6-2torus}
\end{figure}

\begin{figure}[ht!]
\includegraphics[width=0.49\textwidth]{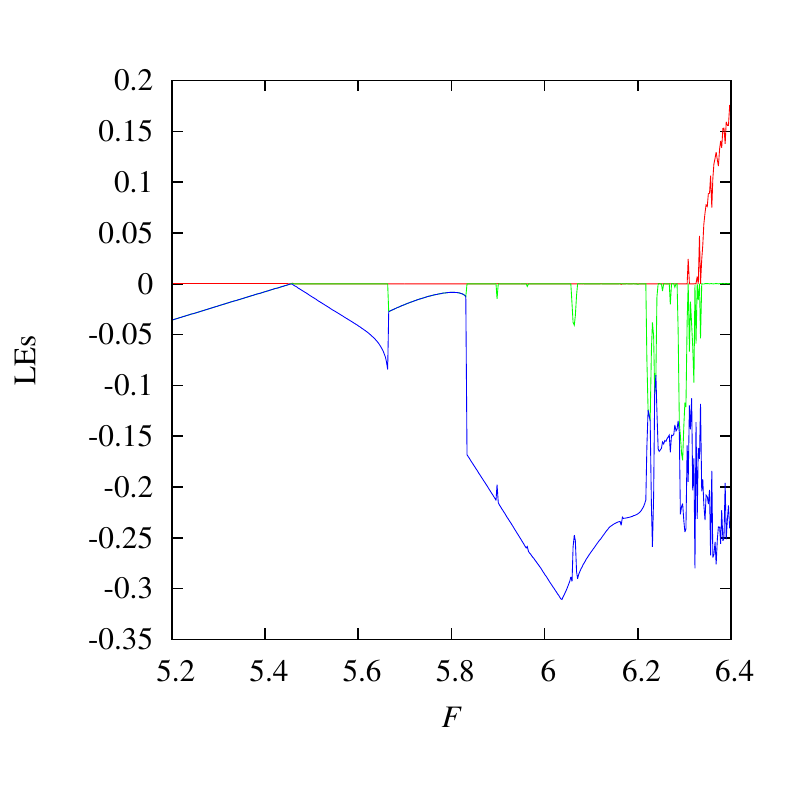}
\includegraphics[width=0.49\textwidth]{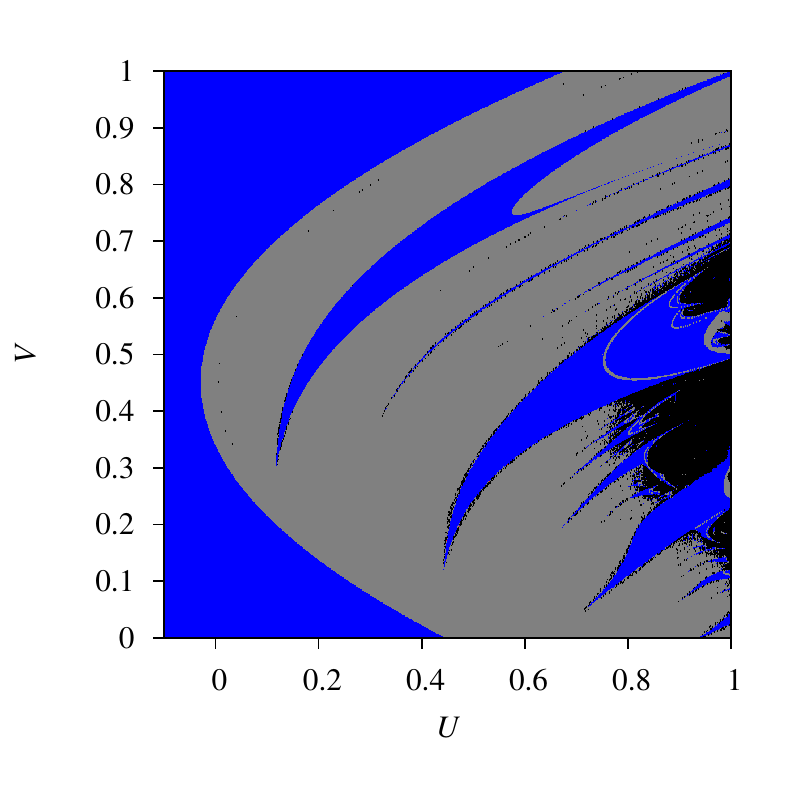}
\caption{The three largest Lyapunov exponents of the Lorenz-96 model as a function of the parameter $F$ for $n=6$ and $G=0$ (left panel) and a Lyapunov diagram in the parameters $(U,V)$ defined by the affine transformation $(F,G) = (U+6V+1, 0.35\,V-0.25)$ (right panel). The colour coding for the right panel is almost the same as in Table~\ref{tab:LegendLyapdiag}, except that blue indicates a periodic attractor for wave number $l=1$. The Arnol'd tongues emanating from the NS-curve are clearly visible.}\label{fig:dim6-bifos}
\end{figure}

\subsection{Dimension $n = 12$}\label{sec:n=12}
Part of the dynamics for this dimension is already explained in Section~\ref{sec:Unfoldingn12}. Here we present the results of our numerical exploration which support the analytical results very well.

Recall that the first bifurcation of the trivial equilibrium for $G=0$ is a Hopf-Hopf bifurcation, which is rather exceptional for the Lorenz-96 model. This codimension two point acts as an organising centre, as explained in Section~\ref{sec:unfolding}. Two codimension one NS-curves originate from this bifurcation point, each corresponding to one of the wave numbers $l=2$ or $3$. The local bifurcation diagram obtained by  \textsc{MatCont} is presented in Figure~\ref{fig:Lorenz96n12ExtBifDiag1C} and should be compared with the analytically computed bifurcation diagram in Figure~\ref{fig:Lorenz96n12ExtLocalBifDiagram}.
\begin{figure}[ht!]
\centering
\includegraphics[width=0.5\textwidth]{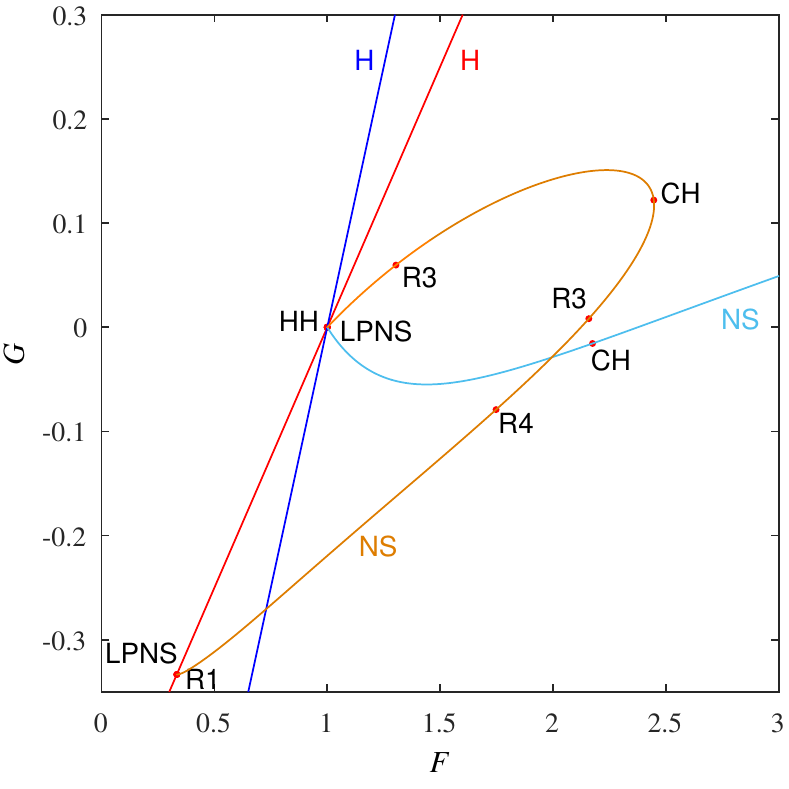}\\
\caption{Local bifurcation diagram for $n=12$ near the Hopf-Hopf bifurcation point obtained by numerical continuation. The blue and red lines are the Hopf bifurcation curves for $l=2$ and $l=3$, respectively. The light-blue and orange curves are NS-curves for the periodic orbit originating from the Hopf bifurcation with $l=2$ and $l=3$, respectively. The NS-curve for $l=3$ ends at the corresponding Hopf line. Also compare with Figure~\ref{fig:Lorenz96n12ExtLocalBifDiagram} and~\ref{fig:Lorenz96n12Lyapdiagram} both indicating the (global) dynamics in each region.}\label{fig:Lorenz96n12ExtBifDiag1C}
\end{figure}

In the region enclosed by both NS-curves multistability occurs, due to the coexistence of the periodic attractors for both $l=2$ and $l=3$. Both attractors are plotted for the same parameter values $(F,G) = (1.5,0)$ in Figure~\ref{fig:dim12-wave}. Together with their Hovm\"oller diagrams in Figure~\ref{fig:dim12-hov}, this shows that both waves are of a different nature. Multistability is also reflected by the Lyapunov diagrams in Figure~\ref{fig:Lorenz96n12Lyapdiagram}. The left (resp.\ right) panel is obtained by fixing the value of the parameter $F$ and increasing (resp.\ decreasing) the parameter $G$. Along each vertical line in the parameter plane we have used the last point on the attractor detected in the previous step as an initial condition for the next one. In both diagrams we have used a grid of size 1000 by 1000. See Table~\ref{tab:LegendLyapdiag} for an explanation of the colouring for each region. Figure~\ref{fig:Lorenz96n12Lyapdiagram} clearly shows that there is a region in the parameter plane where two different periodic attractors coexist. Also note that the bifurcation curves of Figure~\ref{fig:Lorenz96n12ExtBifDiag1C} are clearly visible in these diagrams.  Lastly, Figure~\ref{fig:Lorenz96n12Lyapdiagram} shows that the Hopf-Hopf bifurcation influences a large portion of the parameter space as well as the phase space.

\begin{figure}[ht!]
  \centering
    \includegraphics[width=0.49\textwidth]{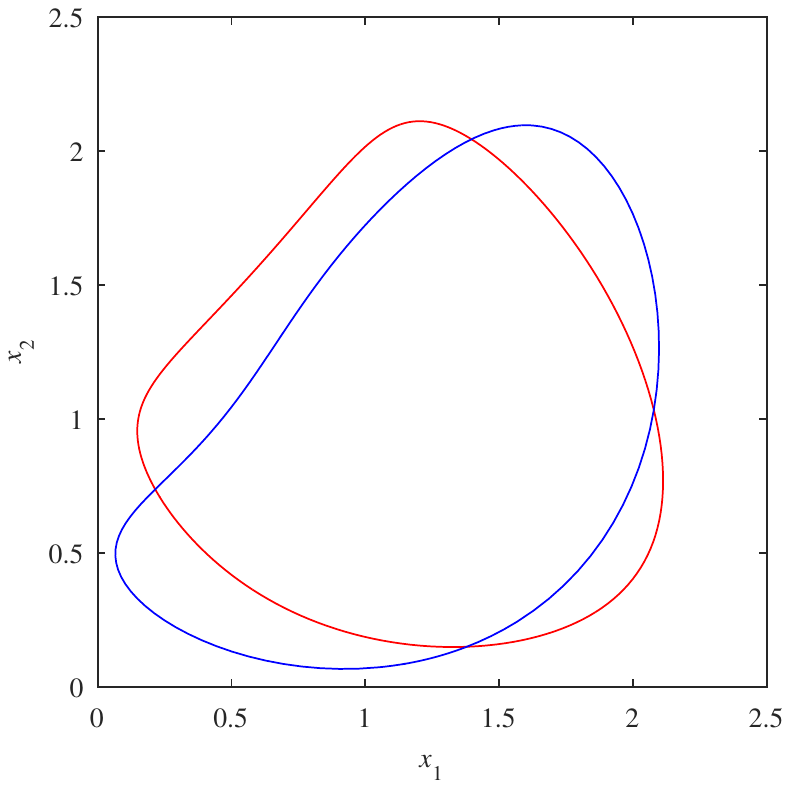}
    \caption{Periodic attractors with wavenumbers $l=2$ (blue) and $l=3$ (red) for $(n,F,G)=(12,1.5,0)$, which is in the region enclosed by the two NS-curves where multistability occurs.}\label{fig:dim12-wave}
\end{figure}
\begin{figure}[ht!]
  \centering
    \includegraphics[width=0.49\textwidth]{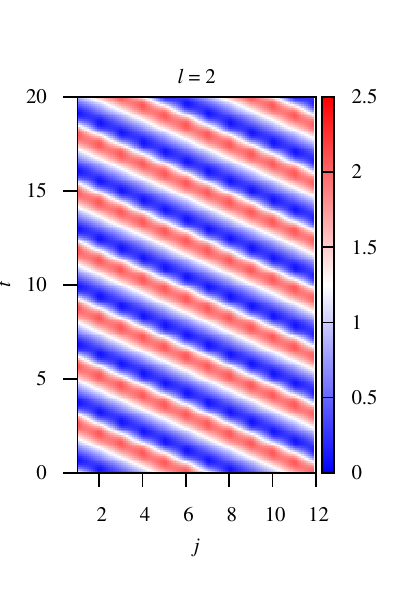}
    \includegraphics[width=0.49\textwidth]{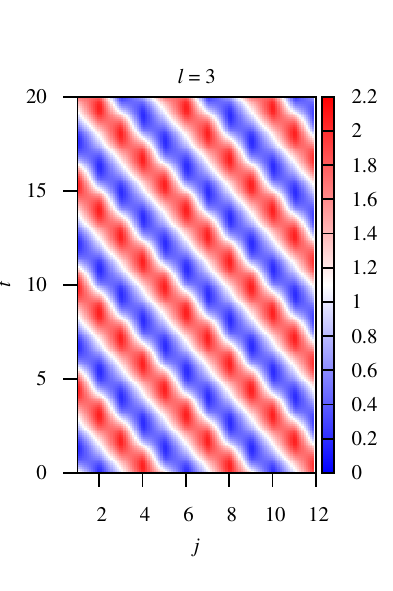}\\
    \caption{Hovm\"oller diagrams of the periodic attractors from Figure~\ref{fig:dim12-wave} with $l=2$ (left panel) and $l=3$ (right panel) for $(n,F,G)=(12,1.5,0)$. The value of $x_j(t)$ is plotted as a function of $t$ and $j$. For visualization purposes linear interpolation between $x_j$ and $x_{j+1}$ has been applied in order to make the diagram continuous in the variable $j$. Note that the difference in both the period and the wave number is clearly visible.}\label{fig:dim12-hov}
\end{figure}

\begin{figure}[p]
  \centering
    \includegraphics[width=0.49\textwidth]{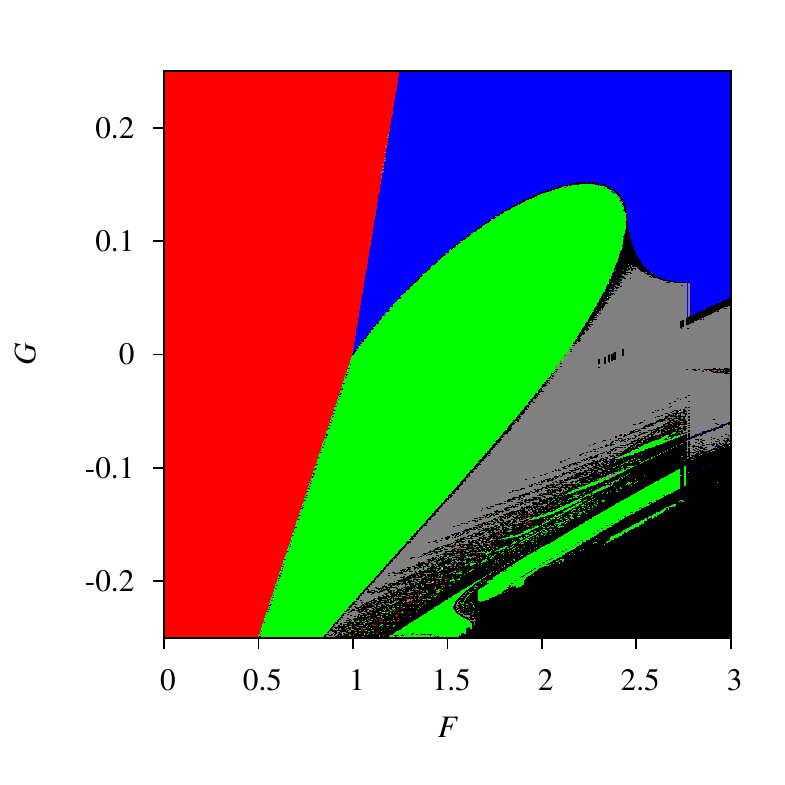}
    \includegraphics[width=0.49\textwidth]{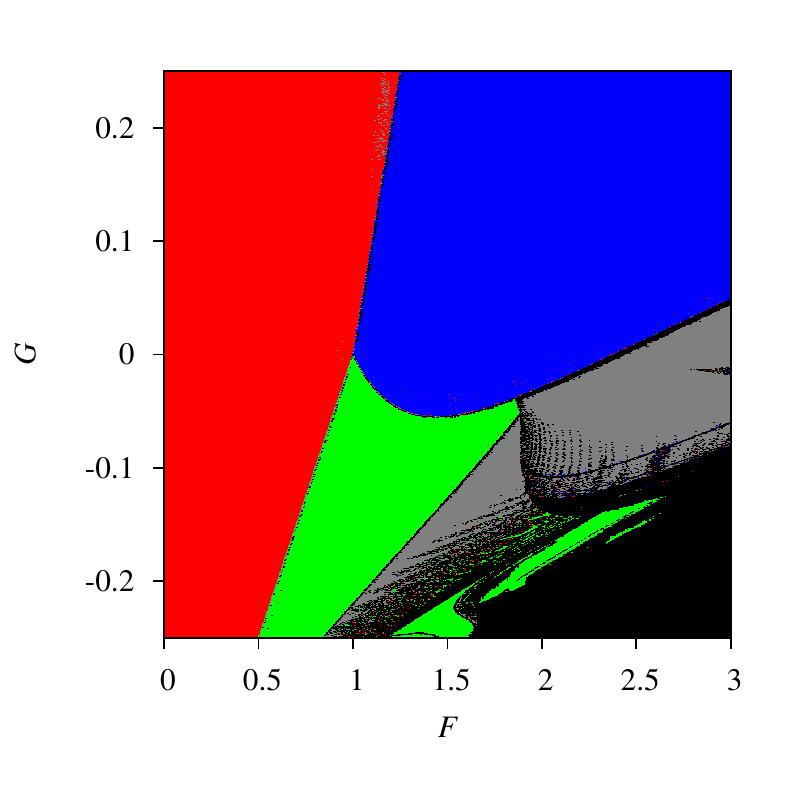}\\
    \caption{Lyapunov diagrams for $n=12$ and domain $(F,G) \in [0,3]\times[-0.25,0.25]$, computed from bottom to top (left panel) and from top to bottom (right panel). See Table~\ref{tab:LegendLyapdiag} for the colour coding. Note that the bifurcation curves shown in Figure~\ref{fig:Lorenz96n12ExtLocalBifDiagram} are clearly visible.}\label{fig:Lorenz96n12Lyapdiagram}
\end{figure}

\begin{table}[p]\centering
\caption{Colour coding for the Lyapunov diagram in Figure
\ref{fig:Lorenz96n12Lyapdiagram}.}
\label{tab:LegendLyapdiag}
\begin{tabular}{ll}
\hline\noalign{\smallskip}
Colour & Type of attractor\\
\noalign{\smallskip}\hline\noalign{\smallskip}
Red   & Stable equilibrium \\
Blue  & Periodic attractor for $l=2$ \\
Green & Periodic attractor for $l=3$ \\
Grey  & Quasi-periodic attractor \\
Black & Chaotic attractor \\
\noalign{\smallskip}\hline
\end{tabular}
\end{table}

\subsection{Dimension $n = 36$}
For $n=36$, we observe again coexistence of attractors, like in the case $n=12$. The Hopf-Hopf bifurcation that induces this phenomenon occurs at the intersection of the Hopf-lines for wave numbers $l=7$ and $l=8$ where $(F,G) = (0.919586, 0.0144084)$, i.e.~close to the $F$-axis. Note that these wave numbers correspond to the first two Hopf bifurcations of the trivial equilibrium for $F>0$ and $G=0$. The Hopf-Hopf bifurcation is of the same type as for $n=12$ (see Section~\ref{sec:Unfoldingn12}), meaning that only two NS-curves arise from the codimension two point. The local bifurcation diagram in Figure~\ref{fig:Lorenz96n36ExtBifDiag} shows these two curves together with their corresponding Hopf-lines. The blue NS-curve (corresponding to $l=7$) intersects the line $G=0$ at $F \approx 0.9092541$, so we can observe multistability in the one-parameter model~\eqref{eq:Lorenz96} for $F$ larger than this value. Again, the Hopf-Hopf bifurcation point acts as an organising centre.

In Figures~\ref{fig:dim36-wave7} and~\ref{fig:dim36-wave8} the Lyapunov diagrams are shown for $l=7$ and $l=8$, respectively, with $G=0$ fixed. For wave number $l=7$, the first bifurcation after the Hopf bifurcation at $F_{\Hf}(7,36)\approx0.902474$ is the mentioned NS at $F \approx 0.9092541$, which is followed by another NS at $F\approx4.389139$. The resulting quasi-periodic attractor then bifurcates to a 3-torus (see below). For $l=8$, a stable periodic attractor originates from a Hopf bifurcation at $F_{\Hf}(8,36) \approx0.898198$. This attractor exhibits a PD at $F\approx3.155456$ and becomes unstable via a subcritical NS at $F\approx3.162597$, which can be seen from the right panel of Figure~\ref{fig:dim36-wave8}. The only stable attractor after $F\approx3.162597$ is the one with wave number $l=7$. This is reflected in the Lyapunov diagrams of Figure~\ref{fig:dim36-wave8}, where the Lyapunov exponents take up the values for $l=7$ right after the subcritical NS at $F\approx3.162597$ (compare with Figure~\ref{fig:dim36-wave7}). These observations show that the region of multistability is bounded for $G=0$.

The Lyapunov diagram in the right panel of Figure~\ref{fig:dim36-wave7} suggests that for $G=0$ a 3-torus exists in a small interval of $F$-values before chaotic attractors are observed. Figure~\ref{fig:dim36-3torus} shows a 3-torus attractor for $(n,F,G)=(36,4.45,0)$ together with a Poincar\'e section defined by $\Sigma =\{x_1=2\}$. This type of behaviour has also been observed for $n=24$ (not shown). Newhouse, Ruelle and Takens \cite{Newhouse78} proved that small perturbations of a quasi-periodic flow on the 3-torus can lead to strange Axiom A attractors. Concrete routes of the NRT-scenario were reported in \cite{Broer08,Broer08A} in the setting of a model map for the Hopf-saddle-node bifurcation in diffeomorphisms. Some techniques to study bifurcations of 3-tori in continuous-time dynamical systems are described in \cite{Kamiyama15}. Unravelling the bifurcations of 3-tori and the associated routes to chaos in the Lorenz-96 model is left for future research.

\begin{figure}[ht!]
\centering
\includegraphics[width=\textwidth]{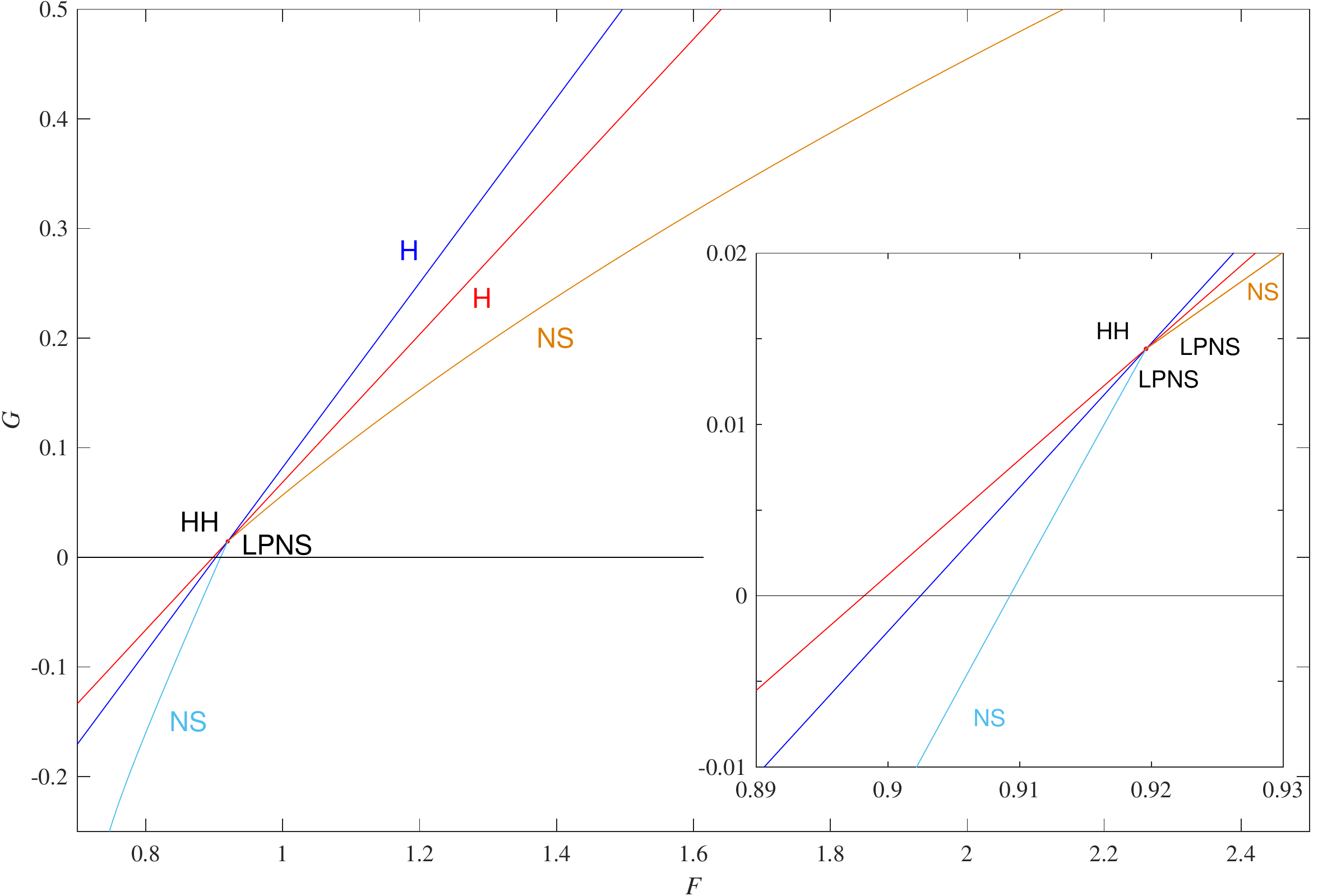}
\caption{Local bifurcation diagram for $n=36$ near the Hopf-Hopf bifurcation point at $(F,G) = (0.919586, 0.0144084)$ obtained by numerical continuation. The blue and red lines are the Hopf bifurcation curves for $l=7$ and $l=8$, respectively. The light-blue and orange curves are NS-curves for the periodic orbit originating from the Hopf bifurcation with $l=7$ and $l=8$, respectively. The box magnifies the region around the Hopf-Hopf point and the line $G=0$.}
\label{fig:Lorenz96n36ExtBifDiag}
\end{figure}

\begin{figure}[p]
\includegraphics[width=0.49\textwidth]{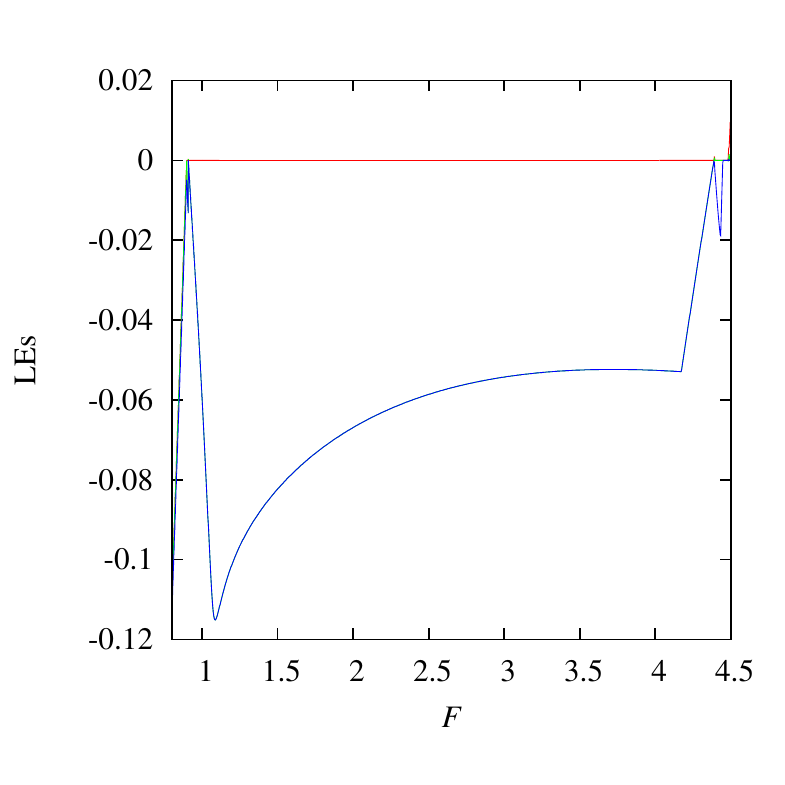}
\includegraphics[width=0.49\textwidth]{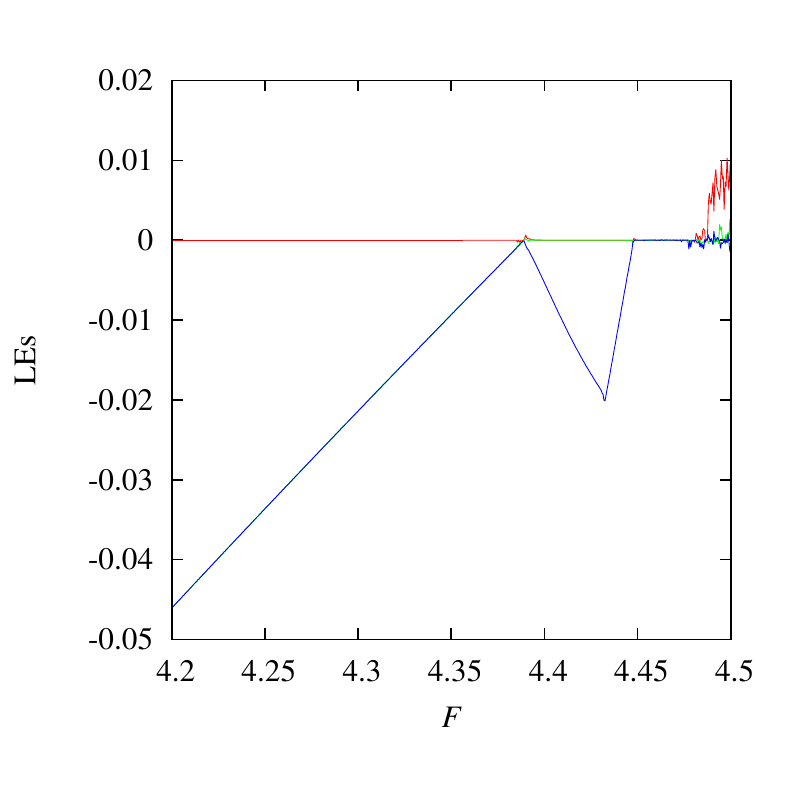}
\caption{The three largest Lyapunov exponents of the Lorenz-96 model as a function of the parameter $F$ for $n=36$ and wave number $l=7$ (left panel).  The right panel is a magnification of the right part of the left panel, which displays the appearance of a 3-torus for $F\in[4.45, 4.48]$. In both panels $G=0$.}\label{fig:dim36-wave7}
\end{figure}

\begin{figure}[p]
\includegraphics[width=0.49\textwidth]{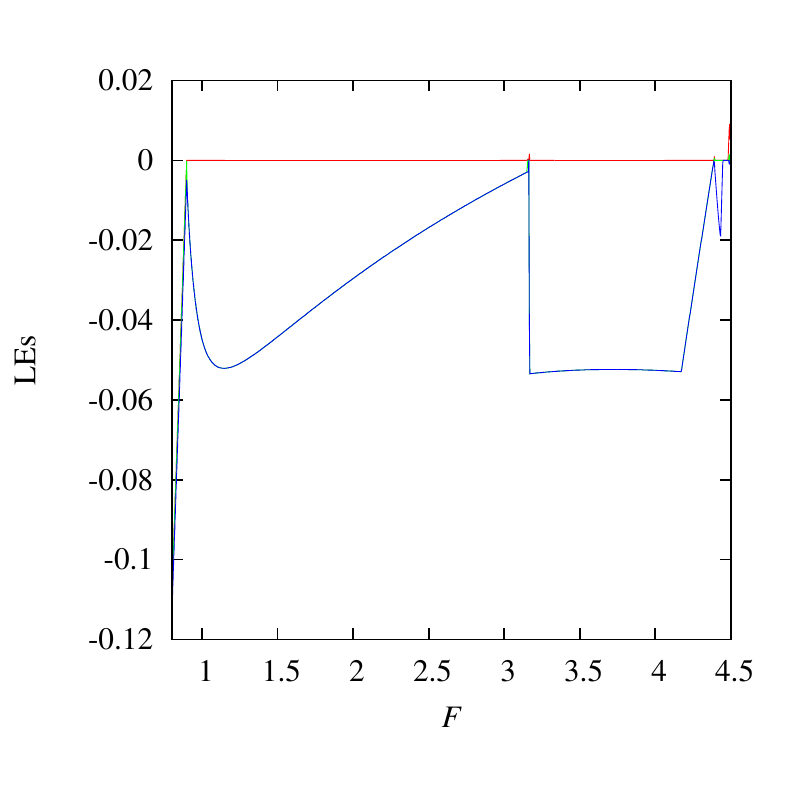}
\includegraphics[width=0.49\textwidth]{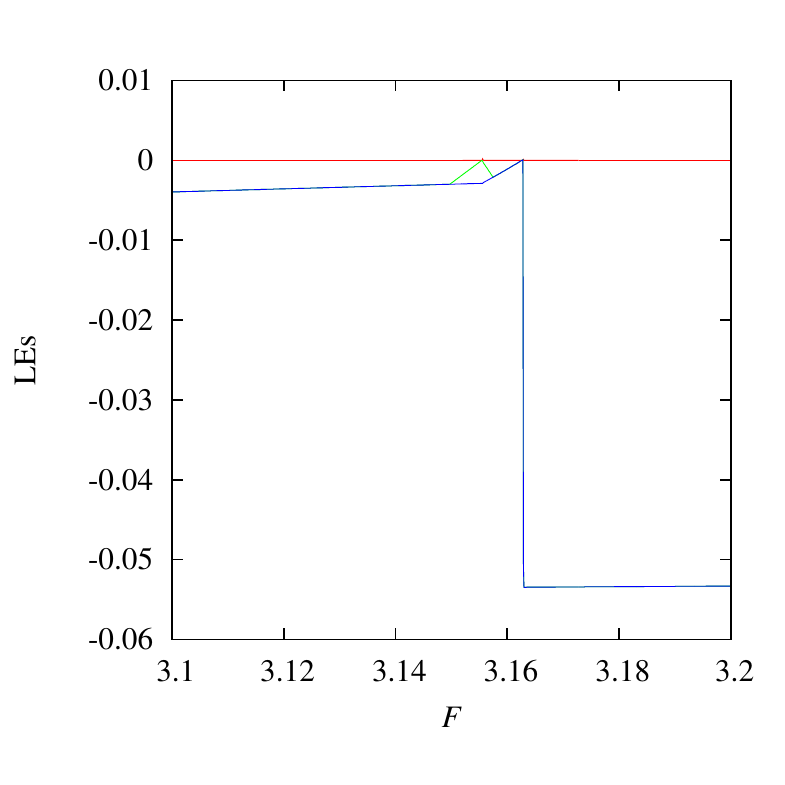}
\caption{The three largest Lyapunov exponents of the Lorenz-96 model as a function of the parameter $F$ for $n=36$ and wave number $l=8$ (left panel). The right panel shows a magnification of the left panel around $F=3.15$, showing the disappearance of the stable attractor for $l=8$ at $F> 3.163$. For larger $F$ the Lyapunov exponents take up the values of the stable attractor with wavenumber $l=7$. In both panels $G=0$.}\label{fig:dim36-wave8}
\end{figure}

\begin{figure}[ht!]
\includegraphics[width=0.49\textwidth]{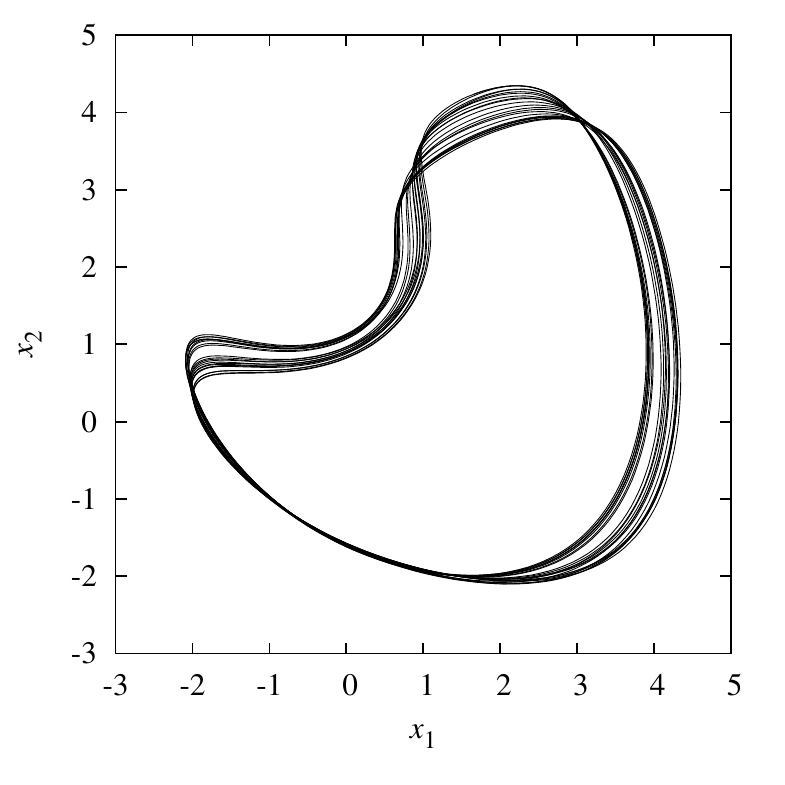}
\includegraphics[width=0.49\textwidth]{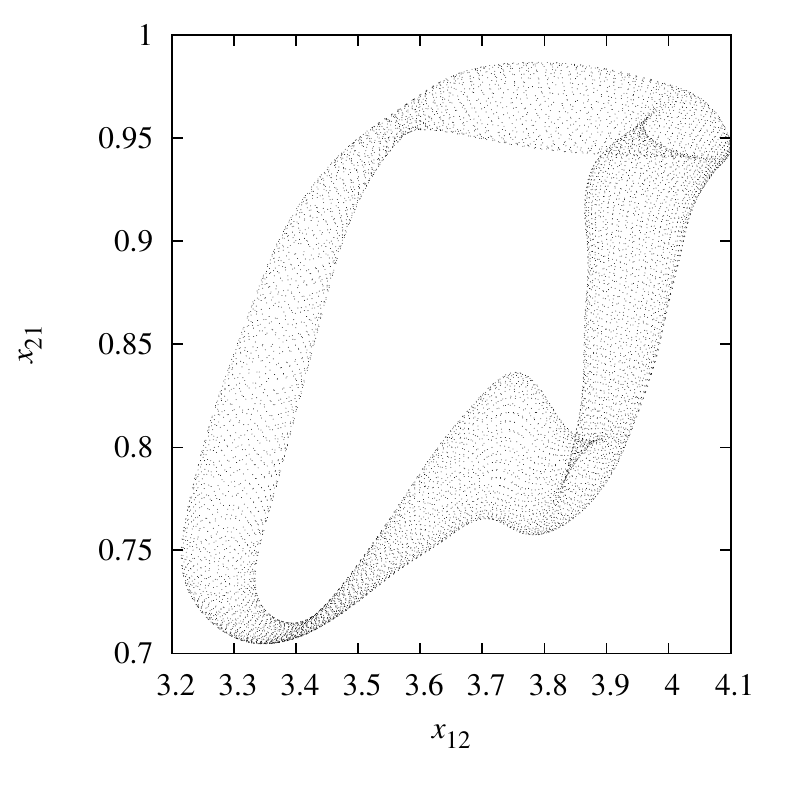}
\caption{A 3-torus attractor (left panel) for $(n,F,G) = (36, 4.45, 0)$ and the corresponding 2-torus attractor of the Poincar\'e return map defined by the section $\Sigma=\{x_1 = 2\}$ (right panel).}
\label{fig:dim36-3torus}
\end{figure}

%-----------------------------------------------------
\section{Proofs of the analytical results}\label{sec:proofs}
In this section we prove the statements in Section~\ref{sec:AnalyticalResults}, in the same order.

\subsection{Proof of Lemma~\ref{lem:Bifevcrossing}}\label{sec:evcrossing}
\begin{proof}[Proof of statement 1]
    Let $n\geq4$ and $l\in \Nat$. To investigate the $l$-th eigenvalue pair of $x_F$ we need to consider $0<l< \tfrac{n}{2}$ only, by equation~\eqref{eq:Lzevconj}. Let the $l$-th eigenvalue be written as $\lambda_l(F,n) := \mu(F) + \omega(F)i$, with real and imaginary parts as in~\eqref{eq:Lzsimpleevn}. Then $\lambda_l (F,n)$ has zero real part if and only if
    \begin{equation*}
      \mu(F) := -1 + F f(l,n) = 0.
    \end{equation*}
    Since the value of $f$ is already fixed by choosing $l$ and $n$, this is achieved only if $F$ equals
    \begin{equation}\label{eq:Hopfbifurcationvalue}
        F_{\Hf}(l,n) = \frac{1}{f(l,n)},
    \end{equation}
    where we need the additional condition that $l\neq\tfrac{n}{3}$, since $f(n/3,n) = 0$ (see Figure~\ref{fig:Lzevsplitting}).

    The eigenvalues cross the imaginary axis with nonzero speed by the fact that
    \begin{equation*}
        \mu'(F_{\Hf}) = f(l,n) \neq 0,
    \end{equation*}
    due to the constraints on $l$. Moreover, let $\omega_0 = |\omega(F_{\Hf})| := -F_{\Hf} g(l,n)$ denote the absolute value of the imaginary part at the Hopf bifurcation. Then, by the restrictions on $l$, $\omega_0$ is nonzero as well.
\end{proof}
\begin{figure}
  \centering
  \includegraphics[width=\textwidth]{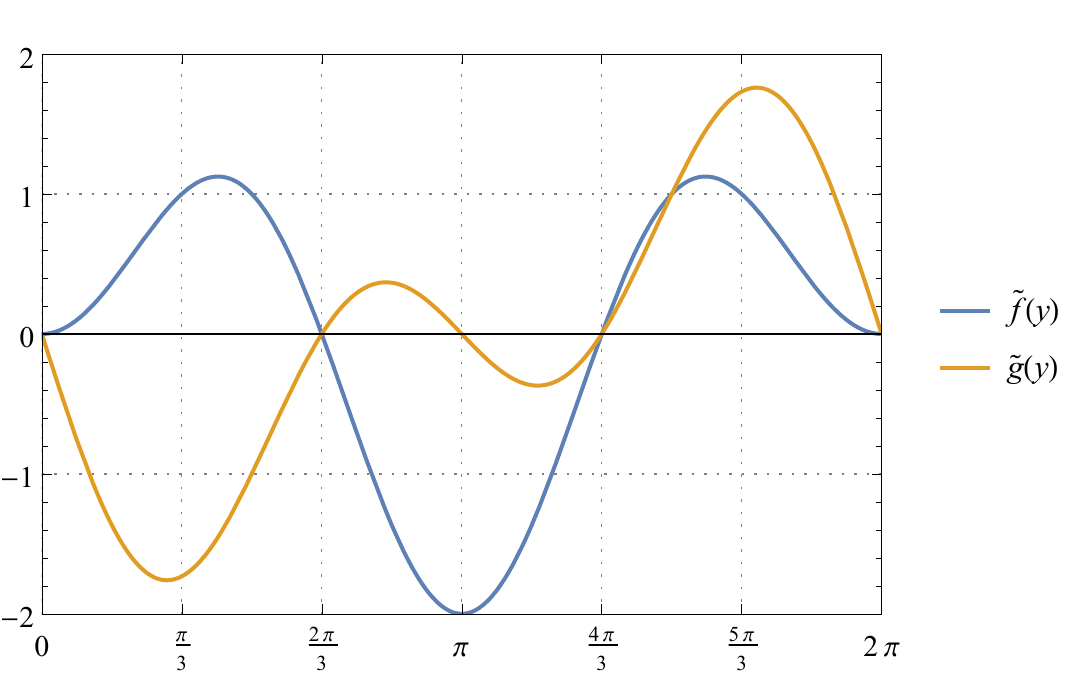}\\
  \caption{Graph of the functions $f$ and $g$ defined by equation~\eqref{eq:Lzevfg}, with the discrete points~$\tfrac{2\pi j}{n}$ replaced by the continuous variable $y \in [0,2\pi]$.}\label{fig:Lzevsplitting}
\end{figure}

For later purposes, we note that at the eigenvalue crossing the dependence of the eigenvalues on $F$ can be replaced by the dependence on $l$, by substituting $F = F_{\Hf}(l,n)$. This gives
\begin{equation}\label{eq:eigenvaluejln}
  \lambda_j(l,n) = -1 + \frac{f(j,n)}{f(l,n)} + i\frac{g(j,n)}{f(l,n)},
\end{equation}
for all $j = 0, \ldots, n-1$. One can easily see that for $j=l$, $\lambda_l$ is purely imaginary at $F_{\Hf}$, i.e.\ $\lambda_l(l,n) = -i \omega_0$, where $\omega_0$ is also expressed in $l$ and $n$:
\begin{equation}\label{eq:imaginarypartev}
  \omega_0 (l,n) = -\frac{g(l,n)}{f(l,n)} = \frac{\cos\tfrac{\pi l}{n}}{\sin\tfrac{\pi l}{n}},
\end{equation}
for $0<l<\tfrac{n}{2}, l\neq\tfrac{n}{3}$.

\begin{proof}[Proof of statement 2]
  Let $\tilde{f}$ be defined by $\tilde{f}(y) := \cos y - \cos 2y$ such that $\tilde{f}(\tfrac{2\pi l}{n}) = f(l,n)$, i.e.\ $\tilde{f}$ is equal to the function $f$ with the discrete points~$\tfrac{2\pi j}{n}$ replaced by the continuous variable $y$, see Figure~\ref{fig:Lzevsplitting}. By the definition of $F_{\Hf}$ in~\eqref{eq:Hopfbifurcationvalue}, its positive, respectively negative, values with the smallest absolute value occur at the maximum, respectively minimum, of the function $\tilde{f}$.

  The extreme values of $\tilde{f}$ are obtained by
    \begin{equation*}
        \begin{aligned}
        0 &= \tfrac{\dd \tilde{f}}{\dd y} = 2\sin 2y - \sin y \\
        &= (4\cos y - 1)\sin y,
        \end{aligned}
    \end{equation*}
  and, hence, $y=k \pi$ for $k\in \mathbb{Z}$ or $\cos y =\tfrac{1}{4}$. The cases $y = k\pi$ give the global minimum of $\tilde{f}$ if $k$ is odd: $\tilde{f}(\pi) = -2$ (if $k$ is even, we have a local minimum). Therefore, the upper bound on the negative values of $F_{\Hf}$ is equal to $\tfrac{1}{\tilde{f}(\pi)} = -\tfrac{1}{2}$. This value can be obtained from $f(l,n)$ only if we take $l = \tfrac{n}{2}$, which is, however, excluded by the assumptions on $l$ and, hence, will never be attained. The solution $y_{\text{top}}=\cos^{-1}\left(\tfrac{1}{4}\right)$ gives the maximum of $\tilde{f}$ as $\tilde{f}(y_{\text{top}}) = \tfrac{9}{8}$. Therefore the lowest possible positive value for which a Hopf bifurcation can occur is $\tfrac{1}{\tilde{f}(y_{\text{top}})} = \tfrac{8}{9}$.

\paragraph{Upper bound of $F_{\Hf}$}The positive values with the largest absolute value are obtained as follows. Since there are only finitely many $l$ satisfying $0 < l < \tfrac{n}{2}$, we know that these values of $F_{\Hf}$ should be bounded for any $n$. The largest value for $F_{\Hf}$ is obtained when $f(l,n)$ is the smallest.

\begin{claim}
The smallest value of $f(l,n)$ for every $n\geq 4$ is obtained at $l=1$ except when $n=7$, in which case we have to take $l=2$.
\end{claim}
\begin{proof}[Proof of the claim]
    For $n = 4,5$ or $6$, this is trivial since $l=1$ is the only integer satisfying $0<l<\tfrac{n}{3}$. In the case $n=7$, we have two integers that satisfy $0<l<\tfrac{n}{3}$ and it is easily checked that $f(1,7) > f(2,7)$ gives the desired exception. So, all we need to show in order to verify our claim is that
\begin{equation}\label{eq:difference}
    f(1,n) - f(l',n) \leq 0
\end{equation}
holds for any $n\geq8$, where $l'$ is the largest integer for which $l' < \tfrac{n}{3}$. Since the function $f$ becomes negative for $l> \tfrac{n}{3}$, this $l'$ will be the largest integer for which $f$ is positive and becomes close to 0, see Figure~\ref{fig:Lzevsplitting}. This gives rise to the following three cases:
\begin{enumerate}
  \item If $n = 0 \bmod 3$, then $l' = \tfrac{n}{3}-1$. Then equation~\eqref{eq:difference} can be simplified to
    \begin{align*}
     \hspace{-2em}
     f(1,n) - f(l',n) &= \cos\tfrac{2\pi}{n} - \cos\tfrac{4\pi}{n} - \cos 2\pi(\tfrac{1}{3}-\tfrac{1}{n}) + \cos 4\pi(\tfrac{1}{3}-\tfrac{1}{n})\\
     &= 2\sqrt{3}\sin\tfrac{3\pi}{n}\sin(\tfrac{\pi}{n}-\tfrac{\pi}{6}).
    \end{align*}
    Since we have to consider $n\geq 9$ here, the first sine-term is always positive, while the second one is always negative (by the fact that its entry is negative and bigger than $-\tfrac{\pi}{6}$). Hence, $f(1,n) - f(l',n) < 0$ holds for these particular values of $n$.
  \item If $n = 1 \bmod 3$, then $l' = \tfrac{n-1}{3}$. Then equation~\eqref{eq:difference} reduces to
    \begin{align*}
     \hspace{-2em}
     f(1,n) - f(l',n) &= \cos\tfrac{2\pi}{n} - \cos\tfrac{4\pi}{n} - \cos\tfrac{2\pi}{3}(1-\tfrac{1}{n}) + \cos\tfrac{4\pi}{3}(1-\tfrac{1}{n})\\
     &= 4\sin\tfrac{\pi}{n}\cos(\tfrac{8+n}{6n}\pi) \sin(\tfrac{10-n}{6n}\pi).
    \end{align*}
    In this case, we have to take $n\geq 10$, for which the first sine-term and the cosine-term are both always positive. The second sine-term, $\sin\tfrac{10-n}{6n}\pi$ is exactly equal to 0 if $n=10$ (compare this with criterion~\ref{crit:HH1} in Corollary~\ref{cor:HHBif}) and strictly less than 0 for $n > 10$. This gives the desired inequality.
  \item If $n = 2 \bmod 3$, then $l' = \tfrac{n-2}{3}$ and we have that
    \begin{align*}
     \hspace{-2em}
     f(1,n) - f(l',n) &= \cos\tfrac{2\pi}{n} - \cos\tfrac{4\pi}{n} - \cos\tfrac{2\pi}{3}(1-\tfrac{2}{n}) + \cos\tfrac{4\pi}{3}(1-\tfrac{2}{n})\\
     &= 2\sin\tfrac{\pi}{n}\left(\sin\tfrac{3\pi}{n}- 2\cos\tfrac{\pi}{n}\sin\tfrac{\pi}{3}(1-\tfrac{2}{n})\right).
     \end{align*}
     Observe that the part in brackets is monotonically decreasing, since the first sine-term decreases as $n$ increases, while both components $\cos\tfrac{\pi}{n}$ and $\sin\tfrac{\pi}{3}(1-\tfrac{2}{n})$ increase. Moreover, for $n=8$ --- the least possible $n$ in this case --- we find that
     \begin{equation*}
      \sin\tfrac{3\pi}{n}- 2\cos\tfrac{\pi}{n}\sin\tfrac{\pi}{3}(1-\tfrac{2}{n}) = -\sin\tfrac{\pi}{8} < 0,
     \end{equation*}
     which implies that also in this case equation~\eqref{eq:difference} holds for any $n\geq 8$.
\end{enumerate}
Now, we have established equation~\eqref{eq:difference} for any $n\geq 4$, except $n=7$, and for these $n$ we can conclude that $l=1$ is the right choice to get the lowest value of $f(l,n)$.
\end{proof}

\noindent \emph{Continuation of proof of statement 2.} We can conclude that the upper bound (actually, a maximum) on $F_{\Hf}$ is given by
  \begin{equation*}
    F_{\max}(n) = \left\{
        \begin{array}{ll}
            \tfrac{1}{f(2,7)} &\quad  \text{if}\ n = 7,\\
            \tfrac{1}{f(1,n)} &\quad  \text{otherwise}.
        \end{array}
    \right.
  \end{equation*}

\paragraph{Lower bound of $F_{\Hf}$}In the case of negative $F_{\Hf}$-values, we do not have an eigenvalue crossing for $n = 4$ and 6, so $F_{\min} = -\tfrac{1}{2}$ (this gives the empty set). For all other $n$ we need the integer $l > \tfrac{n}{3}$ which is the closest to $\tfrac{n}{3}$ on the right to get the largest value of $f$, see Figure~\ref{fig:Lzevsplitting}. If we write $n = 3r + s$, where $r,s\in\Nat$ with $s = n\bmod 3$, then we see that
  \begin{equation*}
    r \leq \frac{n}{3} = \frac{3r +s}{3} = r + \frac{s}{3} < r+1.
  \end{equation*}
Therefore, we have to take $l=r+1$ to obtain the lowest integer satisfying $l > \tfrac{n}{3}$. Hence, the lower bound on $F_{\Hf}$ is given by
  \begin{equation*}
    F_{\min}(n) = \left\{
        \begin{array}{ll}
            -\tfrac{1}{2}       &\quad  \text{if}\ n = 4, 6,\\
            \tfrac{1}{f(r+1,n)} &\quad  \text{otherwise},
        \end{array}
    \right.
  \end{equation*}
  which is again an actual minimum for given $n$.
\end{proof}

\subsection{Proofs for the Hopf-Hopf bifurcation}\label{sec:HHBif}
The occurrence of Hopf-Hopf bifurcations is stated in Theorem~\ref{thm:HHBif1} and Corollary~\ref{cor:HHBif}; we will prove both here.

\begin{proof}[Proof of Theorem~\ref{thm:HHBif1}] Throughout the proof we assume that $l_1 < l_2$, without loss of generality.

Suppose that at a certain parameter value $F$ a Hopf-Hopf bifurcation occurs, that means, the $l_1$-th and $l_2$-th eigenvalue pairs both have real part equal to 0. So, we need to have that $F_{\Hf}(l_1,n) = F_{\Hf}(l_2,n)$, or, equivalently,
  \begin{equation}\label{eq:LzHHcondition1}
    f(l_1,n) = f(l_2,n),
  \end{equation}
  where
  \begin{equation}\label{eq:LzHHconditionl1}
  0 < l_1 < \tfrac{n}{2\pi}\cos^{-1}\left(\tfrac{1}{4}\right) < l_2 < \tfrac{n}{3}.
  \end{equation}
  (The second condition follows from the fact that if $l_1$ and $l_2$ give the same value for $f$, then for the continuous function $\tilde{f}$ we need $y_1 = \tfrac{2\pi l_1}{n}$ to be left and $y_2 = \tfrac{2\pi l_2}{n}$ right of the top $y_{\text{top}}=\cos^{-1}\left(\tfrac{1}{4}\right)$ in the domain of consideration, $(0,\pi)$. So, $y_1$ and $y_2$ have to satisfy $0 < y_1 < y_{\text{top}} < y_2 < \tfrac{2\pi}{3}$ (see Figure~\ref{fig:Lzevsplitting} for the picture). This is equivalent to equation~\eqref{eq:LzHHconditionl1}.)

Since $f(l,n)$ can be written as
  \begin{equation*}\label{eq:Lzfrewritten1}
    f(l,n) = -2 \cos^2 \tfrac{2\pi l}{n} + \cos\tfrac{2\pi l}{n} + 1,
  \end{equation*}
  the substitution $x=\cos\tfrac{2\pi l}{n}$ gives the function
  \begin{equation*}\label{eq:Lzsimplifiedf1}
    h(x) = -2x^2 + x + 1.
  \end{equation*}
  Condition~\eqref{eq:LzHHcondition1} then becomes
  \begin{equation*}\label{eq:LZHHbifscaled1}
    h\left(\cos\tfrac{2\pi l_1}{n}\right) = h\left(\cos\tfrac{2\pi l_2}{n}\right).
  \end{equation*}
  By condition~\eqref{eq:LzHHconditionl1} on $l_1$ and $l_2$, $\cos\tfrac{2\pi l_1}{n}$ is on the left and $\cos\tfrac{2\pi l_2}{n}$ is on the right of the maximum $x=\tfrac{1}{4}$ of $h$. Since the function $h(x)$ is symmetric around the maximum, $l_1, l_2$ and $n$ should satisfy
  \begin{equation*}
    \tfrac{1}{2}\left(\cos\tfrac{2\pi l_1}{n} +\cos\tfrac{2\pi l_2}{n}\right) = \tfrac{1}{4}.
  \end{equation*}
This provides the condition for a Hopf-Hopf bifurcation to occur.

Conversely, suppose that equation~\eqref{eq:LzHHbifcond1} holds for $l_1$ and $l_2$ satisfying $0 < l_1, l_2 < \tfrac{n}{2}$, $l_1, l_2 \neq \tfrac{n}{3}$ and $l_1 \neq l_2$. The existence of a Hopf-Hopf bifurcation can be derived from Lemma~\ref{lem:Bifevcrossing} by showing that the bifurcation parameters $F_{\Hf}$ corresponding to each of these eigenvalue pairs coincide, i.e.\ $F_{\Hf}(l_1,n) = F_{\Hf}(l_2,n)$.

Let us denote $y_1 = \tfrac{2\pi l_1}{n}$ and $y_2 = \tfrac{2\pi l_2}{n}$, then equation~\eqref{eq:LzHHbifcond1} becomes:
    \begin{equation*}
    \tfrac{1}{2} = \cos y_1 + \cos y_2.
  \end{equation*}
From this equation, we obtain $\cos y_2 = \tfrac{1}{2} - \cos y_1$ and, with a little trigonometry, its double angle reads
  \begin{equation*}
    \cos 2 y_2 = 2 \cos^2 y_2 - 1 = -\tfrac{1}{2} - 2 \cos y_1 + 2 \cos^2 y_1 = \tfrac{1}{2} - 2 \cos y_1 + \cos 2 y_1.
  \end{equation*}
Now, observe that the following holds:
  \begin{align*}
    \tilde{f}(y_2) &= \cos y_2 - \cos 2 y_2 \\
     &= \left(\tfrac{1}{2} - \cos y_1\right) - \left(\tfrac{1}{2} - 2 \cos y_1 + \cos 2 y_1\right)\\
     &= \cos y_1 -\cos 2 y_1 = \tilde{f}(y_1).
  \end{align*}
  Hence, it holds that $f(l_1,n) = \tilde{f}(y_1) = \tilde{f}(y_2) = f(l_2,n)$ and therefore $F_{\Hf}(l_1,n) = F_{\Hf}(l_2,n)$ as desired.
\end{proof}

\begin{proof}[Proof of Corollary~\ref{cor:HHBif}]
From equation~\eqref{eq:LzHHbifcond1} we can determine explicit combinations of $l_1, l_2$ and $n$ for which a Hopf-Hopf bifurcation will occur. To begin with the easiest one, criterion~\ref{crit:HH2}: choose $l_2/n$ such that $\cos\tfrac{2\pi l_2}{n} = 0$, i.e.\ $l_2/n=1/4$. This implies that $l_1/n$ has to be equal to $1/6$ to satisfy equation~\eqref{eq:LzHHbifcond1}. Since all numbers have to be integers, we should take $n=12m$ with $m\in\Nat$ and hence, $l_1 = 2m$ and $l_2 = 3m$.

Criterion~\ref{crit:HH1} is obtained by observing that
  \begin{equation*}
    \cos\tfrac{\pi}{5} = \tfrac{1}{4}(1+\sqrt{5}), \quad \textrm{and}
    \quad \cos\tfrac{3\pi}{5} = \tfrac{1}{4}(1-\sqrt{5}),
  \end{equation*}
so that we have the  relations $2l_1/n = 1/5$ and $2l_2/n = 3/5$. These are satisfied by taking multiples of $m\in\Nat$ as follows: $n=10 m$ and $l_1 =~m, l_2 = 3 m$.
\end{proof}

\subsection{Proof of Theorem~\ref{thm:HBif}}\label{sec:LC1}
To prove the occurrence of a Hopf bifurcation we need to show that (\textsl{a}) there is an eigenvalue pair crossing the imaginary axis; and (\textsl{b}) the first Lyapunov coefficient $\ell_1(l,n)$ is nonzero at $F_{\Hf}$ \cite{Kuznetsov04}. Lemma~\ref{lem:Bifevcrossing} shows that the transversality condition (\textsl{a}) holds. Below, we prove the nondegeneracy condition (\textsl{b}), while we also clarify the condition for $\ell_1(l,n)$ to be positive or negative.
In the following, let $l$ and $n$ be as in Lemma~\ref{lem:Bifevcrossing} and assume that there is no $l_2\neq l$ which satisfies both Lemma~\ref{lem:Bifevcrossing} and equation~\eqref{eq:LzHHbifcond1} (with $l_1 = l$).

The \emph{first Lyapunov coefficient} $\ell_1$ corresponding to a Hopf bifurcation of the equilibrium $x_F$ for the $l$-th eigenvalue pair is given by the following invariant expression \cite{Kuznetsov04}:
\begin{align}\label{eq:lyapunovcoefficient}
  \ell_1(F_{\Hf}(l,n)) = \frac{1}{2\omega_0}{\rm Re}\,
    &\left[\langle p, C(q,q,\bar{q})\rangle - 2\langle p, B(q, A^{-1}B(q,\bar{q}))\rangle + \right.\nonumber\\
    &\left.+ \langle p, B(\bar{q}, (2i\omega_0 I_n - A)^{-1}B(q,q))\rangle\right],
\end{align}
where $A$ is the Jacobian matrix and $B$ and $C$ are multilinear functions obtained via the Taylor expansion of the nonlinear part of~\eqref{eq:Lorenz96eq}. The vectors $q$ and $p$ are complex eigenvectors of $A$ and $A^\top$, respectively, and have to be taken such that $q$ is associated with the eigenvalue which crosses the imaginary axis at $F_{\Hf}(l,n)$ \emph{and} which has positive imaginary part $\omega_0$, while $p$ is its adjoint eigenvector. In other words, $q$ and $p$ have to be eigenvectors corresponding to $\lambda_{n-l}$ and $\bar{\lambda}_{n-l}$, respectively. We will specify them later on. Furthermore, note that the inner product on $\mathbb{C}^n$ is defined such that it is antilinear in the first component, i.e.
\begin{equation*}
  \langle x, y\rangle := \sum_{k=0}^{n-1} \bar{x}_k y_k.
\end{equation*}

In this section we will simplify formula~\eqref{eq:lyapunovcoefficient} to an analytic expression which depends on the variables $l$ and $n$ only and whose sign is easily determined. We do this by taking advantage of the fact that the Jacobian matrix at $x_F$ is circulant. The first step is to simplify the expression as much as possible in a general setting. Next, we plug in all the known terms specific to our system. In the last part we determine the condition to have either a positive or a negative Lyapunov coefficient, proving the remaining part of Theorem~\ref{thm:HBif}.

\subsubsection{Simplifying the expression}
First of all, note that by a change of coordinates, $y_j = x_j - F$ (which translates the equilibrium $x_F$ to the origin), we can write the Lorenz-96 model~\eqref{eq:Lorenz96} in the following form:
\begin{equation*}
    \dot{y} = Ay + \tfrac{1}{2}B(y,y), \qquad y \in \mathbb{R}^n,
\end{equation*}
where $A$ is the $n\times n$ Jacobian matrix at the origin and $B : \mathbb{R}^n\times\mathbb{R}^n\to\mathbb{R}^n$ is a bilinear map whose $k$-th component is given by
\begin{equation}\label{eq:Lz96bilinearmap}
  B_k(x,y) = x_{k-1}(y_{k+1}-y_{k-2}) + y_{k-1}(x_{k+1}-x_{k-2}).
\end{equation}

Since the cubic terms are absent in this model, we can immediately simplify~\eqref{eq:lyapunovcoefficient} to
\begin{equation}\label{eq:lyapcoeffsimplified}
    \hspace{-0.1em}
    \ell_1 = \frac{1}{2\omega_0}{\rm Re}\,[ - 2\langle p, B(q, A^{-1}B(q,\bar{q}))\rangle
    + \langle p, B(\bar{q}, (2i\omega_0 I_n - A)^{-1}B(q,q))\rangle].
\end{equation}
We split this equation into two components as follows:
\begin{equation*}
  \ell_1 = \frac{1}{2\omega_0}{\rm Re}[- 2 \ell_{1,a} + \ell_{1,b}],
\end{equation*}
where
\begin{align}
  \ell_{1,a} &:=\langle p, B(q, A^{-1}B(q,\bar{q}))\rangle \label{eq:LC1comp}\\
  \ell_{1,b} &:= \langle p, B(\bar{q}, (2i\omega_0 I_n- A)^{-1}B(q,q))\rangle.\label{eq:LC2comp}
\end{align}

In the Lorenz-96 case the matrix $A$ is circulant for the trivial equilibrium $x_F$ and therefore unitarily equivalent with a diagonal matrix $D$ \cite{Gray06}:
\begin{equation*}
    A = UDU^*, \qquad UU^* = U^*U = I.
\end{equation*}
Here, the diagonal entries of $D$ are the eigenvalues of $A$ and the columns of $U$ are the eigenvectors of $A$ in the same order.
Moreover, since the matrix $A$ is real, we have that $A^\top = UD^*U^*$ which means that
\begin{equation*}
  Av = \lambda v \quad\Leftrightarrow\quad A^\top v = \bar{\lambda}v.
\end{equation*}
Since $q$ and $p$ are eigenvectors corresponding to the $(n-l)$-th eigenvalue of $A$ and $A^\top$, respectively, this means that at the eigenvalue crossing we have
\begin{equation*}
  Aq = \lambda_{n-l} q = i\omega_0 q
\quad\text{and}\quad
  A^\top p = \bar{\lambda}_{n-l} p = -i\omega_0 p,
\end{equation*}
so that we can take
\begin{equation}\label{eq:vectorspq}
  q = p = v_{n-l},
\end{equation}
the normalized eigenvector~\eqref{eq:Lzeigenvector} of $A$ corresponding to the $(n-l)$-th eigenvalue. This fact already shows that the only values we need to compute formula~\eqref{eq:lyapcoeffsimplified} are $F_{\Hf}, \rho_j$, and $\omega_0$, which are all determined by the choice of $l$ and $n$.

\paragraph{Elimination of inverse matrices}
The next step is to remove the inverse matrices in formula~\eqref{eq:lyapcoeffsimplified}. The fact that the eigenvectors of $A$ form a unitary matrix implies that we can express any $x \in \mathbb{C}^n$ in terms of the eigenvectors of $A$ using a standard Fourier decomposition
\begin{equation*}
  x = \sum_{j=0}^{n-1} \langle v_j ,x\rangle v_j.
\end{equation*}
This makes it easy to determine how $A$ and its inverses act on any vector $x$:
\begin{eqnarray*}
  Ax & = & \sum_{j=0}^{n-1} \lambda_j\langle  v_j,x\rangle v_j,\\
  A^{-1}x & = & \sum_{j=0}^{n-1}  \frac{\langle v_j,x\rangle}{\lambda_j} v_j,\\
  (2i\omega_0 I_n - A)^{-1}x & = & \sum_{j=0}^{n-1}  \frac{\langle v_j, x\rangle}{2i\omega_0-\lambda_j} v_j,
\end{eqnarray*}
where we use the relation $A^{-1} = U D^{-1} U^*$ in the second line. In the following we will implement these relations in both equations~\eqref{eq:LC1comp} and~\eqref{eq:LC2comp}. Note that up to this step, we only used the property that $A$ is normal.

\paragraph{First component $\ell_{1,a}$}By the bilinearity of the operator $B$, the linearity of the inner product in the second component and the expression for the inverse of $A$, the first part of the first Lyapunov coefficient~\eqref{eq:LC1comp} can be written as
\begin{eqnarray*}
  \ell_{1,a}=\langle p, B(q, A^{-1}B(q,\bar{q}))\rangle
   & = & \langle p, B(q, \sum_{j=0}^{n-1} \frac{1}{\lambda_j} \langle v_j, B(q,\bar{q})\rangle v_j)\rangle\\
   & = & \langle p, \sum_{j=0}^{n-1} \frac{1}{\lambda_j} \langle v_j, B(q,\bar{q})\rangle B(q, v_j)\rangle\\
   & = & \sum_{j=0}^{n-1} \frac{1}{\lambda_j} \langle v_j, B(q,\bar{q})\rangle\langle p, B(q, v_j)\rangle. \\
\end{eqnarray*}
In the Lorenz-96 case, the inner product terms in the last line become
\begin{equation}\label{eq:lyapcoeff1a}
\begin{aligned}
  \langle v_j, B(q,\bar{q})\rangle &= \sum_{k=0}^{n-1} \bar{v}^{k}_j \left(q_{k-1}(\bar{q}_{k+1} - \bar{q}_{k-2}) + \bar{q}_{k-1}(q_{k+1} - q_{k-2})\right),\\
  \langle p, B(q,v_j)\rangle &= \sum_{k=0}^{n-1} \bar{p}_k \left(q_{k-1}(v_j^{k+1} - v_j^{k-2}) + v_j^{k-1}(q_{k+1} - q_{k-2})\right),
\end{aligned}
\end{equation}
by equation~\eqref{eq:Lz96bilinearmap}.

We can fill in the explicit relations from equations~\eqref{eq:Lzeigenvector} and~\eqref{eq:vectorspq}, i.e.~$v_j^k = \rho_j^k/\sqrt{n}$ and $p_k = q_k = v_{n-l}^k$, to obtain
\begin{equation}\label{eq:LC1a}
\begin{aligned}
    \langle v_j, B(q,\bar{q})\rangle &= \frac{1}{n\sqrt{n}}(\rho_l^{-2}-\rho_l^{-1}-\rho_l+\rho_l^{2}) \sum_{k=0}^{n-1}\rho_j^{-k},\\
    \langle p, B(q,v_j)\rangle &= \frac{1}{n\sqrt{n}} \left(\rho_l(\rho_j-\rho_j^{-2})+\rho_j^{-1}(\rho_l^{-1}-\rho_l^{2})\right) \sum_{k=0}^{n-1}\rho_j^{k}.
  \end{aligned}
\end{equation}
Note that the sums at the end of both equations are conjugate to each other. For $j\neq 0 \bmod n$, we can use a result from finite geometric series and the fact that $\rho_{j}^{n} = 1$ to compute
\begin{equation*}\label{eq:finitegeometricseries}
  \sum_{k=0}^{n-1} \rho_{j}^{k} = \frac{1-\rho_{j}^n}{1-\rho_{j}} = 0, \qquad j \neq 0 \bmod n.
\end{equation*}
For $j=0$, we have that $\rho_{0} = 1$ and so the sum becomes
\begin{equation*}
 \sum_{k=0}^{n-1} \rho_{0}^{k} = \sum_{k=0}^{n-1} 1 = n.
\end{equation*}
In brief, we found that
\begin{equation}\label{eq:sumrootofunity}
  \sum_{k=0}^{n-1} \rho_{j}^{k} = \left\{
        \begin{array}{ll}
            n &\quad  \text{if}\ j = 0 \bmod n,\\
            0 &\quad  \text{otherwise}.
        \end{array}
    \right.
\end{equation}
Thus, equations~\eqref{eq:LC1a} are equal to 0 for all $j$ but 0. Hence, only the term for $j=0$ is left in the summation for $\ell_{1,a}$, which then reduces to
\begin{align*}
  \ell_{1,a}(l,n) &= \frac{1}{n}\frac{1}{\lambda_0}(\rho_l^{-2}-\rho_l^{-1}-\rho_l+\rho_l^{2})(\rho_l^{-1}-\rho_l^2) \nonumber\\
  &= -\frac{1}{n}(\rho_l^{-3}-\rho_l^{-2}+2\rho_l^1+\rho_l^{3}-\rho_l^{4}-2).
\end{align*}

In the computation of formula~\eqref{eq:lyapcoeffsimplified} we only need the real part of the last expression. It mainly consists of powers of $\rho_l$, so Euler's formula yields
\begin{align}\label{eq:LC1Re}
  {\rm Re}\ \ell_{1,a}(l,n) &= -\frac{1}{n}{\rm Re}\left[\rho_l^{-3}-\rho_l^{-2}+2\rho_l^1+\rho_l^{3}-\rho_l^{4}-2\right] \nonumber\\
  &= -\frac{1}{n}\left(2\cos(\tfrac{2\pi l}{n})- \cos(\tfrac{4\pi l}{n})+2\cos(\tfrac{6\pi l}{n})-\cos(\tfrac{8\pi l}{n})-2\right)\nonumber\\
  &= -\frac{4}{n}\sin^2(\tfrac{3\pi l}{n})\left(\cos(\tfrac{2\pi l}{n})-1\right).
\end{align}

\paragraph{Second component $\ell_{1,b}$}The second part can be simplified similarly. Using the bilinearity of the operator $B$, the linearity of the inner product in the second component and the expression for the inverse matrix, the second part of the first Lyapunov coefficient~\eqref{eq:LC2comp} is given by
\begin{equation*}
  \begin{aligned}
    \ell_{1,b}= \langle p, B(\bar{q}, (2i\omega_0 I_n - A)^{-1}B(q,q))\rangle
    = &\sum_{j=0}^{n-1} \frac{\langle v_j, B(q,q)\rangle \langle p, B(\bar{q}, v_j)\rangle}{2i\omega_0-\lambda_j}.
    \end{aligned}
\end{equation*}
In the case of the Lorenz-96 model, each of the inner product parts can be written as
\begin{equation*}
\begin{aligned}
  \langle v_j, B(q,q)\rangle &= \sum_{k=0}^{n-1} 2\bar{v}_j^k \left(q_{k-1}(q_{k+1} - q_{k-2})\right),\\
  \langle p, B(\bar{q}, v_j)\rangle &= \sum_{k=0}^{n-1} \bar{p}_k \left(\bar{q}_{k-1}(v_j^{k+1} - v_j^{k-2}) + v_j^{k-1}(\bar{q}_{k+1} - \bar{q}_{k-2})\right),
\end{aligned}
\end{equation*}
by a small adjustment of equations~\eqref{eq:lyapcoeff1a}.

As before, we can replace $p,q$ and $v$ by powers of $\rho$ times a constant:
\begin{equation}\label{eq:LC2a}
\begin{aligned}
  \langle v_j, B(q,q)\rangle &= \frac{2}{n\sqrt{n}}(1-\rho_l^3) \sum_{k=0}^{n-1}\rho_j^{-k}\rho_l^{-2k},\\
  \langle p, B(\bar{q}, v_j)\rangle &= \frac{1}{n\sqrt{n}} \left(\rho_l^{-1}(\rho_j-\rho_j^{-2}) + \rho_j^{-1}(\rho_l-\rho_l^{-2})\right) \sum_{k=0}^{n-1}\rho_j^{k}\rho_l^{2k}.
\end{aligned}
\end{equation}
Note that the sums in both equations are again conjugate to each other. The summand of the sum in the second equation can be written as $\rho_j^{k}\rho_l^{2k} = \rho_{j+2l}^{k}$. Formula~\eqref{eq:sumrootofunity}, with $j$ replaced by $j+2l$, then shows that
\begin{equation*}
  \sum_{k=0}^{n-1} \rho_{j+2l}^{k} = \left\{
        \begin{array}{ll}
            n &\quad  \text{if}\ j + 2l = 0 \bmod n,\\
            0 &\quad  \text{otherwise}.
        \end{array}
    \right.
\end{equation*}
The sum in the first equation of~\eqref{eq:LC2a} gives exactly the same result, by conjugacy. Therefore, in both cases only the terms with $j=n-2l$ are nonzero:
\begin{equation*}
\begin{aligned}
  \langle v_{n-2l}, B(q,q)\rangle &= \frac{2}{\sqrt{n}}\left(1-\rho_l^3\right),\\
  \langle p, B(\bar{q}, v_{n-2l})\rangle &= \frac{1}{\sqrt{n}}\left(\rho_l^{-3}-1\right).
\end{aligned}
\end{equation*}

These results reduce $\ell_{1,b}$ to
\begin{align}\label{eq:LC2}
  \ell_{1,b}(l,n) &= \frac{\tfrac{2}{\sqrt{n}}(1-\rho_l^3)\tfrac{1}{\sqrt{n}}(\rho_l^{-3}-1)} {2i\omega_0(l,n)-\lambda_{n-2l}(l,n)}\nonumber\\
  &=\frac{2}{n} \frac{\rho_l^{-3}+\rho_l^3-2}{2i\omega_0(l,n)-\overline{\lambda_{2l}(l,n)}}.
\end{align}
Note that --- as with the first component --- the summation and indices $j$ disappeared.

Again, we need the real part of~\eqref{eq:LC2} only. This is a bit more complicated than in the case of $\ell_{1,a}$. First of all, we can reduce the numerator to
\begin{align*}
  \rho_l^{-3}+\rho_l^3-2 &= e^{6\pi i l/n} + e^{-6 \pi i l/n} -2\\
  &= \cos(\tfrac{6\pi l}{n}) + i \sin(\tfrac{6\pi l}{n}) + \cos(\tfrac{-6\pi l}{n}) + i \sin(\tfrac{-6\pi l}{n}) - 2\\
  &= 2\cos(\tfrac{6\pi l}{n}) - 2\\
  &= -4\sin^2(\tfrac{3\pi l}{n}).
\end{align*}

We can take the real part of the complex denominator easily via ${\rm Re}\big(\tfrac{1}{a+b i}\big) = \tfrac{a}{a^2+b^2}$. Thus we find, using the expressions~\eqref{eq:eigenvaluejln} and~\eqref{eq:imaginarypartev},
\begin{align*}
  \hspace{-1em}{\rm Re}\left[\frac{1}{2i\omega_0(l,n)-\overline{\lambda_{2l}(l,n)}}\right]
  &= \frac{1-\tfrac{f(2l,n)}{f(l,n)}}{\left(1-\tfrac{f(2l,n)}{f(l,n)}\right)^2 + \left(\tfrac{g(2l,n)}{f(l,n)}-\frac{2g(l,n)}{f(l,n)}\right)^2} \\
  &= -\frac{2\cos(\tfrac{2\pi l}{n})+2\cos(\tfrac{4\pi l}{n})-1}{4\cos(\tfrac{2\pi l}{n})-4\cos(\tfrac{4\pi l}{n})+9}.\nonumber
\end{align*}
Finally, by substituting these intermediate results in equation~\eqref{eq:LC2}, the real part of the second component $\ell_{1,b}$ results as
\begin{equation}\label{eq:LC2Re}
  {\rm Re}\ \ell_{1,b}(l,n) = \frac{8}{n}\sin^2(\tfrac{3\pi l}{n})\frac{2\cos(\tfrac{2\pi l}{n})+2\cos(\tfrac{4\pi l}{n})-1}{4\cos(\tfrac{2\pi l}{n})-4\cos(\tfrac{4\pi l}{n})+9}.
\end{equation}

\subsubsection{Sign of the first Lyapunov coefficient}
Observe that both main components ${\rm Re}\ \ell_{1,a}$ and ${\rm Re}\ \ell_{1,b}$ only depend on $l$ and $n$. So, combining equations~\eqref{eq:LC1Re} and~\eqref{eq:LC2Re} gives an expression of the first Lyapunov coefficient~\eqref{eq:lyapcoeffsimplified} merely in terms of $l$ and~$n$:
\begin{align}\label{eq:LC}
  \hspace{-2.3em}
  \ell_1(l,n) &= \frac{1}{2\omega_0(l,n)}{\rm Re}[-2\ell_{1,a}(l,n) +\ell_{1,b}(l,n)]\nonumber\\
  &= \frac{\sin(\tfrac{\pi l}{n})}{2\cos(\tfrac{\pi l}{n})}
     \Bigg(\frac{8}{n}\sin^2(\tfrac{3\pi l}{n})\left(\cos(\tfrac{2\pi l}{n})-1\right) + \nonumber\\
   &\quad +\left.\frac{8}{n}\sin^2(\tfrac{3\pi l}{n})
     \frac{2\cos(\tfrac{2\pi l}{n})+2\cos(\tfrac{4\pi l}{n})-1}{4\cos(\tfrac{2\pi l}{n})-4\cos(\tfrac{4\pi l}{n})+9}\right) \nonumber\\
  &= \frac{4}{n}\tan(\tfrac{\pi l}{n})\sin^2(\tfrac{3\pi l}{n})\frac{5\cos(\tfrac{2\pi l}{n})+8\cos(\tfrac{4\pi l}{n})-2\cos(\tfrac{6\pi l}{n})-8}{4\cos(\tfrac{2\pi l}{n})-4\cos(\tfrac{4\pi l}{n})+9},
\end{align}
where $l$ should be taken such that $0 < l < \tfrac{n}{2}, l\neq\tfrac{n}{3}$. It is easy to see already that $\ell_1 = 0$ if we would choose $l = 0$ or $\tfrac{n}{3}$ and, moreover, that for fixed $n$ we have $\lim_{l\rightarrow n/2} \ell_1(l,n) = -\infty$, by the tangent function. Figure~\ref{fig:PlotLC} shows these properties in the (continuous) graph of formula~\eqref{eq:LC}.
\begin{figure}[ht!]
  \centering
  \includegraphics[width=\textwidth]{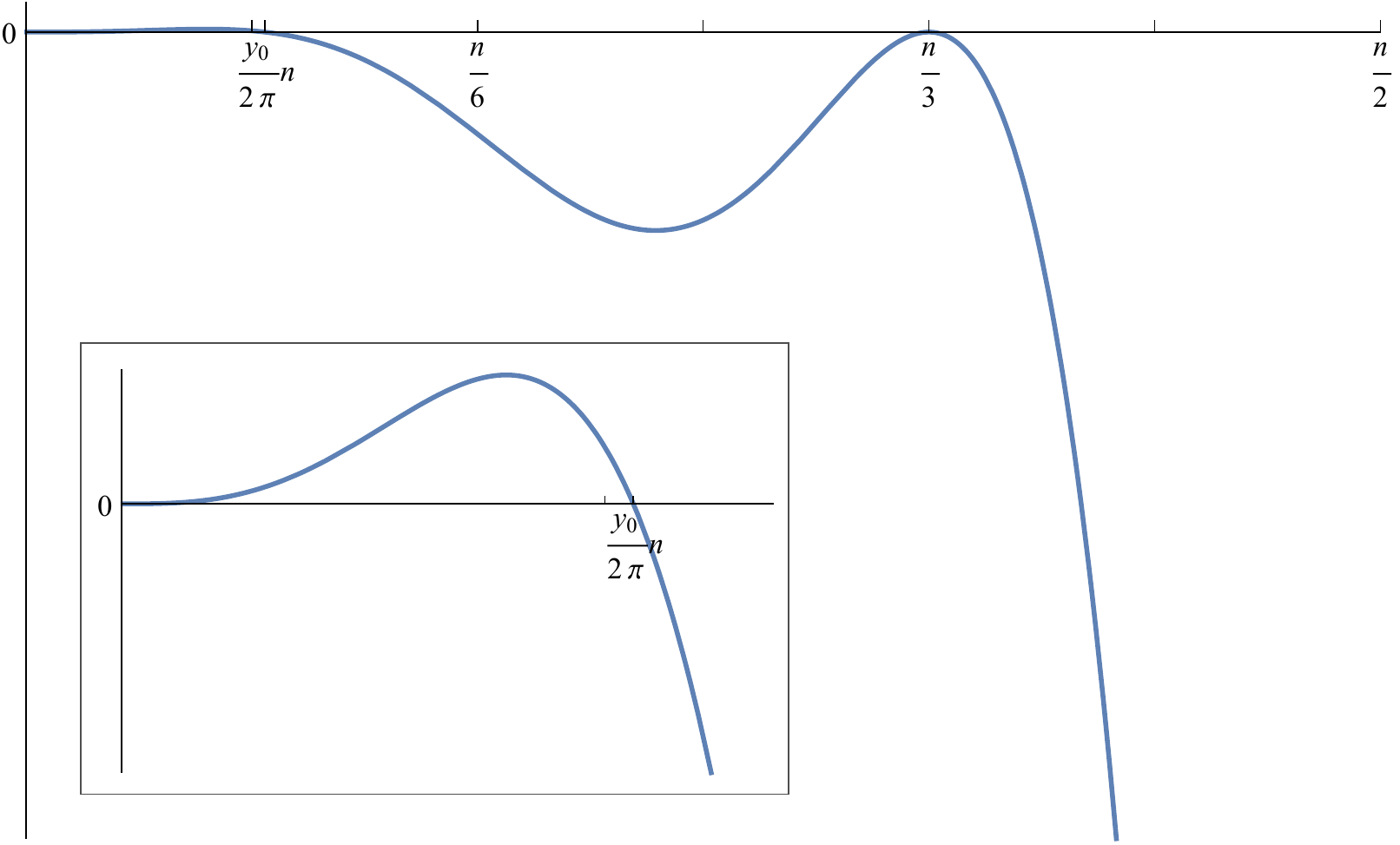}\\
  \caption{Plot of the reduced first Lyapunov coefficient $\ell_1(l,n)$ in equation~\eqref{eq:LC} as a continuous function for general $n \geq 4$ and $l\in \left[0,\tfrac{n}{2}\right)$. The shape remains the same up to scaling for different $n$. The part of the plot around the nontrivial zero $\tfrac{y_0}{2\pi}n$ is magnified in the box, showing that $\ell_1$ is only positive for $l\in(0,\tfrac{y_0}{2\pi}n)$.}\label{fig:PlotLC}
\end{figure}

As it turns out, equation~\eqref{eq:LC} is a useful and easy way to compute the first Lyapunov coefficient. In order to conclude whether the bifurcation is sub- or supercritical, we need to show which combinations of $l$ and $n$ yield a positive or negative value of $\ell_1(l,n)$.

Firstly, observe that the factors in front of the quotient in~\eqref{eq:LC} are always positive, either by the square or by the fact that the tangent function is positive for $0 < \tfrac{l}{n} < \tfrac{1}{2}$.
Secondly, it is easy to check that the denominator of the big quotient of~\eqref{eq:LC} is positive on the entire domain. It remains to determine where the numerator of the quotient is positive or negative.
Let
\begin{equation}\label{eq:LCNumerator}
  L(y) := 5\cos y+8\cos 2y -2\cos 3y -8
\end{equation}
be the numerator defined as a continuous function in $y \in [0,\pi]$, where we replaced $\tfrac{2\pi l}{n}$ by the variable $y$, see Figure~\ref{fig:PlotLCNumerator}. Its derivative satisfies
\begin{equation*}
  L'(y) = \sin y\left(24\cos^2 y-32\cos y-11\right).
\end{equation*}
It is easy to see that $L'(y)$ has a zero if $\sin y = 0$ or if the second order polynomial $24x^2-32x-11$, obtained by the substitution $x = \cos y$, has a zero. In the first case, $\sin y = 0$, we have zeroes at $y = 0$ and $y = \pi$, which give a global and a local maximum of $L(y)$. The second case provides us with exactly two zeroes of the polynomial, namely $x_{\pm} = \tfrac{2}{3} \pm \sqrt{\tfrac{65}{72}}$. However, only $x_-$ is a true solution of $L'(y)$, since the value $x_+$ lies outside the range of the cosine function. Hence, $y_{\min} = \arccos\Big(\tfrac{2}{3} - \sqrt{\tfrac{65}{72}}\Big)$ gives the global minimum of $L(y)$ (see Figure~\ref{fig:PlotLCNumerator}). These three solutions are the only zeroes of $L'(y)$ on the domain $[0,\pi]$. It follows that the sign of $L'(y)$ on each of the intervals $(0,y_{\min})$ and $(y_{\min},\pi)$ does not change.
\begin{figure}[ht!]
  \centering
  \includegraphics[width=0.49\textwidth]{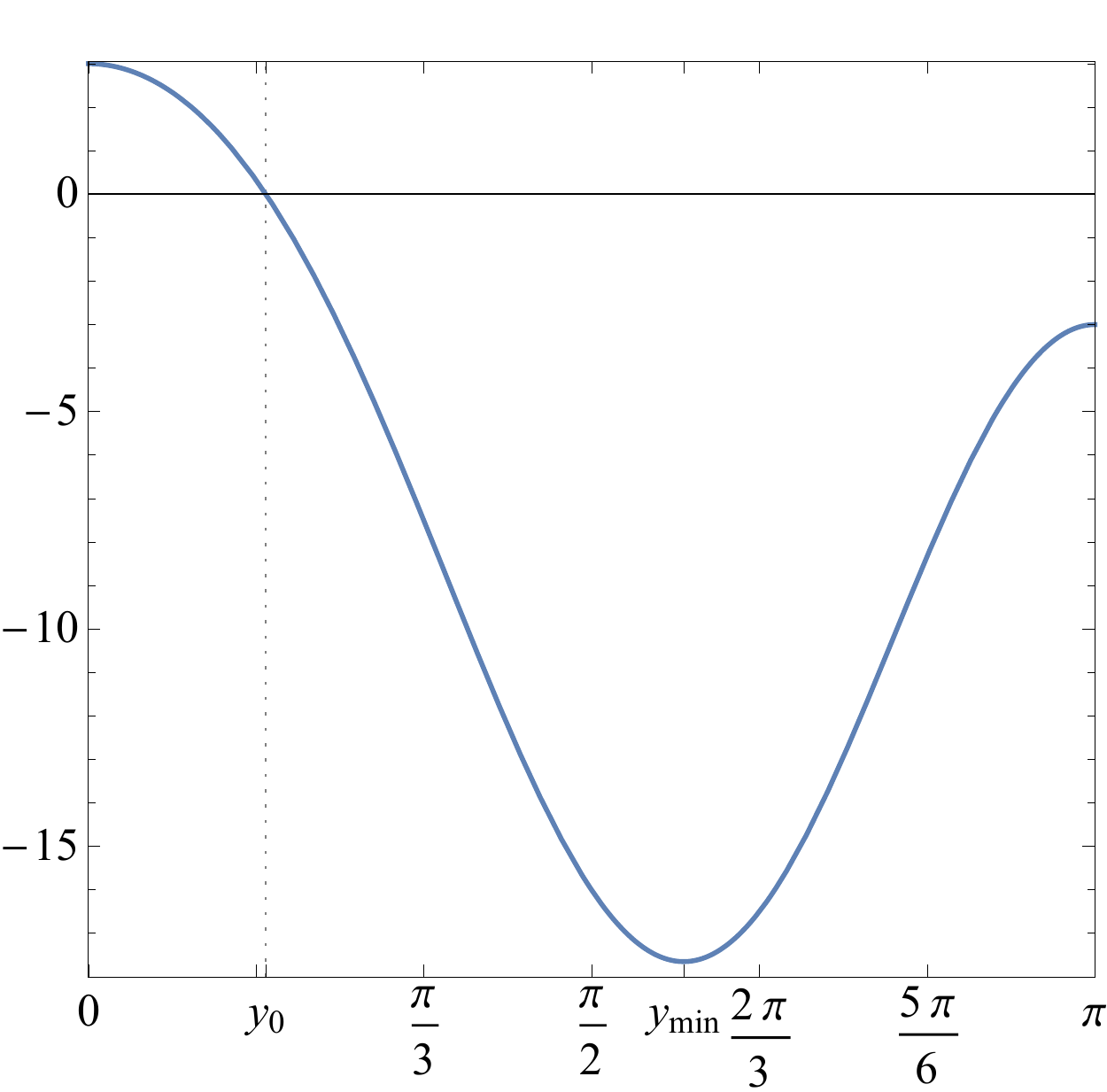}\\
  \caption{Plot of the numerator $L(y)$ in~\eqref{eq:LCNumerator} for $y \in [0,\pi]$. This also indicates where the first Lyapunov coefficient $\ell_1$ is positive or negative, since it has the same sign as $L(y)$.}\label{fig:PlotLCNumerator}
\end{figure}

In order to determine the zeroes of $L(y)$, let us consider first the interval $(0,y_{\min})$. Observe that the derivative $L'$ is negative on this interval (because $L'\left(\tfrac{\pi}{2}\right) = -11$, for example). Since $L(0) = 3$, this means that $L$ can have at most one zero on the interval $(0,y_{\min})$.
Likewise, $L'(y) > 0$ for all $y\in(y_{\min},\pi)$ (by the fact that $L'\left(\tfrac{2\pi}{3}\right) = \tfrac{11}{2}\sqrt{3}$). Since $L(\pi) = -3$, this implies that $L$ has no zero on $(y_{\min},\pi)$, while it should have at least one zero on the entire interval $(0,\pi)$. Consequently, there exists exactly one $y_0 \in (0,y_{\min})$ such that $L(y_0) = 0$, namely $y_0 \approx 0.5545380$ and, moreover, we have that
\begin{equation*}
  L(y) \left\{
        \begin{array}{ll}
            >0 &\quad  \text{if}\ y\in(0,y_0),\\
            <0 &\quad  \text{if}\ y\in(y_0,\pi).
        \end{array}
    \right.
\end{equation*}
This can also be seen from Figure~\ref{fig:PlotLCNumerator}.

To conclude: the simplified expression~\eqref{eq:LC} for the first Lyapunov coefficient $\ell_1$ is positive for $\tfrac{l}{n} < \tfrac{y_0}{2\pi}$ and negative for $\tfrac{l}{n} > \tfrac{y_0}{2\pi}$, where $\tfrac{y_0}{2\pi} \approx 0.08825746$. Therefore, the first Lyapunov coefficient itself is positive (hence, a subcritical Hopf bifurcation) for any $\tfrac{l}{n}\in(0,\tfrac{y_0}{2\pi})$ and negative (hence, a supercritical Hopf bifurcation) for $\tfrac{l}{n}\in\left(\tfrac{y_0}{2\pi},\tfrac{1}{2}\right)\setminus\{\tfrac{1}{3}\}$ (see Figure~\ref{fig:PlotLC}). \qed

\subsection{Proof of Proposition~\ref{prop:firsthopf}}\label{sec:proofTW}
Lemma~\ref{lem:Bifevcrossing} implies that the trivial equilibrium undergoes a Hopf bifurcation at the parameter value $F_{\Hf}(l,n)=1/f(l,n)$. The first Hopf bifurcation for $F>0$ takes place for the integer $l \in (0, \tfrac{n}{3})$ that minimizes the value of $F_{\Hf}(l,n)$, which is equivalent to maximizing $f(l,n)$. (If $l \in (\tfrac{n}{3},\tfrac{n}{2})$, then we obtain a negative value $F_{\Hf}$.)

For all $n \geq 4$ except $n=7$ there exists at least one integer $l \in [\tfrac{n}{6}, \tfrac{n}{4}]$. Indeed, for $n=4, 5$ and $6$ this follows by simply taking $l=1$, and for $n=8, 9, 10$ and $11$ it follows by taking $l=2$. For $n\geq 12$ the width of the interval is larger than~1. We now claim that this implies that
\begin{equation*}
    l_1(n) = \argmax_{0 < l < n/3} f(l,n) \in [\tfrac{n}{6}, \tfrac{n}{4}], \quad n\neq7,
\end{equation*}
as well. To that end we use $\tilde{f}(y) = \cos y-\cos 2y$. Note that $y \in [\tfrac{\pi}{3},\tfrac{\pi}{2}]$ implies that $\tilde{f}(y) \geq 1$ and $y \in (0,\tfrac{\pi}{3})\cup(\tfrac{\pi}{2},\tfrac{2\pi}{3})$ implies that $0 < \tilde{f}(y) < 1$. Moreover, $l \in [\tfrac{n}{6}, \tfrac{n}{4}]$ implies that $\frac{2\pi l}{n} \in [\tfrac{\pi}{3},\tfrac{\pi}{2}]$. Therefore, $f(l,n)$ is maximized for some integer $l \in [\tfrac{n}{6}, \tfrac{n}{4}]$.

In case $n=7$, we can easily compute the smallest value $F_{\Hf}(l,n)$ for which a Hopf bifurcation occurs. We have shown in the proof of Lemma~\ref{lem:Bifevcrossing} (see Section~\ref{sec:evcrossing}) that this is the case for $l=1$.

Finally, assume that the bifurcation is a Hopf bifurcation, i.e., only one eigenvalue pair crosses the imaginary axis. Since $l_1(n) / n \in [\tfrac{1}{7}, \tfrac{1}{4}]$ it follows immediately from Theorem~\ref{thm:HBif} that the first Lyapunov coefficient is negative, which means that the bifurcation is supercritical. \qed

\subsection{Proofs for the two-parameter system}\label{sec:proof2par}
For the two-parameter system~\eqref{eq:Lorenz96eqExt} we have stated the existence of trapping regions in Proposition~\ref{prop:Exttrappingregion} and the occurrence of Hopf bifurcations in Lemma~\ref{lem:HopfExt}. These two results are proven below.

\begin{proof}[Proof of Proposition~\ref{prop:Exttrappingregion}]
Recall that $E = \tfrac{1}{2}\sum_j x_j^2$ and let $V^2 = 2E$, which means that $V = \|x\|$, using the Euclidean norm. Fix $n\in\Nat$ and $F\in\Real$. Take the time derivative of $E$ to obtain
\begin{align}\label{eq:LzExttotalenergyder}
\frac{\dd E}{\dd t} &= V\frac{\dd V}{\dd t} = \frac{1}{2}\frac{\dd (V^2)}{\dd t}
     = \sum_{j=1}^n x_j \dot{x}_j\nonumber\\
    &= \sum_{j=1}^n x_j (x_{j-1}(x_{j+1}-x_{j-2}) - x_j + G(x_{j-1}-2x_j+x_{j+1}) + F)\nonumber\\
    &= \sum_{j=1}^n x_j \left(- x_j + G(x_{j-1}-2x_j+x_{j+1})\right) + F\sum_{j=1}^n x_j\nonumber\\
    &= \langle x, Ax\rangle + F\sum_{j=1}^n x_j,
\end{align}
using the standard inner product on $\Real^n$ with $A$ the matrix equal to the Jacobian at the origin. This matrix is circulant and symmetric with first row equal to
\begin{equation*}
    (c_0, c_1, \dots, c_{n-1}) = (-1-2G, G, 0, \ldots, 0, G).
\end{equation*}
Therefore, all eigenvalues of $A$ are real and given by \cite{Gray06}:
\begin{align*}
    \lambda^A_j(G,n) &= \sum_{k=0}^{n-1} c_k \rho_j^k\nonumber\\
        &= -1-2G +G\rho_j^1 + G\rho_j^{n-1}\nonumber\\
        &= -1 - 2G\left(1-\cos\tfrac{2\pi j}{n}\right).
\end{align*}

To estimate the first term in the right-hand-side of equation~\eqref{eq:LzExttotalenergyder}, we use that the Rayleigh quotient satisfies
\begin{equation*}
  \frac{\langle x, Ax\rangle}{\langle x, x\rangle} \leq \lambda^A_{\max}, \quad \text{for all}\ x\neq0.
\end{equation*}
Therefore, we have to find the largest eigenvalue $\lambda^A_{\max}$ of $A$. It is easy to check that for negative values of $G$ the largest eigenvalue occurs for $j=\tfrac{n}{2}$ (or the nearest integer). For non-negative values of $G$, the largest eigenvalue is always equal to $\lambda^A_0 = -1$.

The second term of equation~\eqref{eq:LzExttotalenergyder} can be estimated by the Cauchy-Schwarz inequality. Equation~\eqref{eq:LzExttotalenergyder} then becomes
  \begin{equation*}
    \begin{aligned}
        V\frac{\dd V}{\dd t} &= \langle x, Ax\rangle + F\sum_{j=1}^n x_j\\
        &\leq \lambda^A_{\max}\sum_{j=1}^n x_j^2 + \sqrt{n}|F| V\\
        &=\lambda^A_{\max}V^2 + \sqrt{n}|F| V,
    \end{aligned}
  \end{equation*}
which can be simplified to
\begin{equation*}
  \frac{\dd V}{\dd t} \leq \lambda^A_{\max}V + \sqrt{n}|F|.
\end{equation*}

Since $V$ is defined as the norm of $x$, $\lambda^A_{\max}$ needs to be negative in order to obtain a solution for $V$ such that $\frac{\dd V}{\dd t} < 0$. This is certainly the case when $G\geq 0$. For negative $G$ and even $n$ the largest eigenvalue is obtained by $\lambda^A_{\max} = \lambda^A_{n/2} = -1-4G$, which is less than 0 if $G > -\tfrac{1}{4}$ for any $n$. For odd values of $n$ we may allow slightly smaller values of $G$, which makes $G = -\tfrac{1}{4}$ a suitable upper bound for the range of $G$ without trapping region.

It follows that $\tfrac{\dd V}{\dd t} < 0$, whenever $G > -\tfrac{1}{4}$ and whenever $V$ is larger than the radius
\begin{equation*}
    R(n) = -\frac{\sqrt{n}|F|}{\lambda^A_{\max}}.
\end{equation*}
Therefore, under these conditions, each sphere with radius $r > R(n)$ is a trapping region for dimension $n$. In particular, if $G=0$ then $R= \sqrt{n}|F|$.
\end{proof}

\begin{proof}[Proof of Lemma~\ref{lem:HopfExt}]
Let $n$ and $l$ be as given. We choose $F$ as the bifurcation parameter. In order to have a Hopf bifurcation for the $l$-th eigenvalue pair $\{\kappa_l,\kappa_{n-l}\}(F,G,n)$, the real part $\mu_l := \RE \kappa_l$ has to vanish. This occurs if $F$ equals
\begin{equation}\label{eq:FHExt}
    F_{\Hf}(G,l,n) = \frac{1}{f(l,n)}\left(1+2G(1-\cos\tfrac{2\pi j}{n})\right).
\end{equation}
For these parameter values we have a purely imaginary eigenvalue pair with the absolute value of the imaginary part given by
\begin{equation*}
    \omega_0(G,l,n) = -F_{\Hf}(G,l,n) g(l,n) = \left|1+2G\left(1-\cos\tfrac{2\pi j}{n}\right)\right|\frac{\cos\tfrac{\pi l}{n}}{\sin\tfrac{\pi l}{n}}.
\end{equation*}
Note that $\omega_0 = 0$ if $l = \tfrac{n}{2}$ or if $G = -\tfrac{1}{2}(1-\cos\tfrac{2\pi l}{n})^{-1}$ for some $l$. The last condition is equivalent to $F_{\Hf} = 0$ by formula~\eqref{eq:FHExt} and implies even that $\kappa_l = 0$. Therefore, this parameter value needs to be excluded.

Furthermore, the eigenvalue pair crosses the imaginary axis with nonzero speed, by the fact that for our restriction of $l$ we have
\begin{equation*}
    \mu_l'(F,G,n) = f(l,n) \neq 0,
\end{equation*}
where the derivative is with respect to the bifurcation parameter $F$. (If we would have taken $G$ as bifurcation parameter, then $\mu_l'$ will be nonzero as well.)

Equation~\eqref{eq:FHExt} thus gives us for general $n$ and for each allowed $l$ a whole line of Hopf bifurcations, which is linear in $G$. Rewritten in terms of $F$ gives the linear curves~\eqref{eq:L96ExtHopf}.
\end{proof}

%-----------------------------------------------------
\section{Conclusions and outlook}
The main goal of the current study was to investigate the dynamical features of the Lorenz-96 model and subsequently prove their existence by analytical means. In this partial inventory, we have proven the existence of Hopf and Hopf-Hopf bifurcations for the trivial equilibrium in all dimensions $n\geq 4$. It is also shown that if the first bifurcation for $F>0$ is a Hopf bifurcation, then it is supercritical by providing an exact formula for the first Lyapunov coefficient, which holds for all possible $n$ and $l$. Our analytical results coincide with \textsc{MatCont}'s numerical estimates of the parameter value $F_{\Hf}$ and (up to a scaling) the value of the first Lyapunov coefficient. Furthermore, the periodic orbits born at the first Hopf bifurcation have the physical interpretation of travelling waves.

We have proved a necessary and sufficient condition for the occurrence of a Hopf-Hopf bifurcation in the original Lorenz-96 model. To unfold these Hopf-Hopf bifurcations we introduced an extra parameter $G$ via a Laplace-like diffusion term. For this particular unfolding the Hopf bifurcations of the trivial equilibrium are given by straight lines in the $(F,G)$-plane and their intersections give Hopf-Hopf bifurcations. In the special case of dimension $n=12$ a Hopf-Hopf bifurcation lies on the line $G=0$ and is in fact the first bifurcation through which the trivial equilibrium bifurcates for $F>0$. We have shown that this codimension two bifurcation point acts as an organising centre. For dimension $n=36$ we find a Hopf-Hopf bifurcation close to the line $G=0$. In both cases normal form analysis shows that two periodic attractors coexist in a region of the $(F,G)$-plane between two Neimark-Sacker bifurcation curves. This region actually intersects the line $G=0$, which means that multistability also occurs in the original Lorenz-96 model. We expect that this mechanism leading to multistability can be found in other dimensions as well. Therefore, by adding a new parameter to the Lorenz-96 model the dynamics observed in the original model can be explained better.

Finally, we have numerically investigated the dynamics of the model for dimensions up to $n=100$ and parameter values $F$ beyond the first Hopf bifurcation value. In dimension $n=4$, a periodic attractor disappears through a fold bifurcation. After this fold bifurcation intermittency is detected, which is possibly explained through a nearby heteroclinic cycle between four equilibria. For dimensions up to $n=100$ that are multiples of $5$, a pattern is discovered with a persisting period-doubling bifurcation. However, we do not expect this pattern to persist for high dimensions. For general $n$, the routes to chaos are numerous and can comprise intermittent transitions, period-doubling cascades and possibly Newhouse-Ruelle-Takens scenarios.

Based on our results we can draw three major conclusions. Firstly, contrary to the persistence of Hopf and Hopf-Hopf bifurcations for any dimension $n\geq 4$, no clear pattern on bifurcations of periodic attractors is found. Secondly, the dependence of the dynamics on $n$ shows the importance of choosing appropriate values of the parameters in specific applications of the Lorenz-96 model such as those listed in Table~\ref{tab:AppLz96}. Lastly, the observation that the wave number of the travelling waves increases with $n$ indicates that the Lorenz-96 cannot be interpreted as a discretised PDE model.

Despite the lack of a clear bifurcation pattern for all dimensions $n$, the Lorenz-96 remains an interesting model to study for its rich dynamics. There are several open questions that have not been addressed in the present paper. For example, how can chaotic attractors with more than one positive Lyapunov exponent arise? Or, how typical is the multistability property in higher dimensions and how does it influence the dynamics? What dynamics can be expected for $F<0$? What is the role of symmetries in the Lorenz-96 model? We are therefore continuing to investigate this model on its dynamical behaviour using both analytical and numerical methods.

\paragraph{Acknowledgements} We would like to thank the two anonymous reviewers for their useful comments and suggestions that have helped to improve this paper.

\newpage\label{sec:Bibliography}
\bibliographystyle{Lz96Hopf-DvKAES}
\bibliography{Lz96Hopf-DvKAES}
\addcontentsline{toc}{section}{References}

\end{document}